\documentclass [11pt,oneside,a4paper,mathscr]{amsart}
\usepackage{geometry}
\newgeometry{margin=1.1in}

\usepackage{amsfonts, amsmath, amssymb, amsthm,mathtools, stmaryrd}
\usepackage{verbatim}
\usepackage[normalem]{ulem}
\usepackage{tikz-cd}
\usepackage{enumitem}

\usepackage{array}

\usepackage{hyperref}

\theoremstyle{plain}
\newtheorem{thm}{Theorem}[section]
\newtheorem{cor}[thm]{Corollary}

\newtheorem{prop}[thm]{Proposition}

\numberwithin{equation}{section}

\newtheorem{conjecture}{Conjecture}  
\newtheorem{conj}[conjecture]{Conjecture}  

\theoremstyle{definition}
\newtheorem{defn}[thm]{Definition}
\newtheorem{example}[thm]{Example}
\newtheorem{lemma}[thm]{Lemma}
\newtheorem{rmk}[thm]{Remark}

\theoremstyle{remark}

\newcommand{\BA}{{\mathbb{A}}}

\newcommand{\BC}{{\mathbb{C}}}

\newcommand{\BE}{{\mathbb{E}}}

\newcommand{\BH}{{\mathbb{H}}}

\newcommand{\BQ}{{\mathbb{Q}}}
\newcommand{\BR}{{\mathbb{R}}}

\newcommand{\BZ}{{\mathbb{Z}}}

\newcommand{\CB}{{\mathcal B}}
\newcommand{\CC}{{\mathcal C}}
\newcommand{\CD}{{\mathcal D}}
\newcommand{\CE}{{\mathcal E}}
\newcommand{\CF}{{\mathcal F}}

\newcommand{\CH}{{\mathcal H}}

\newcommand{\CL}{{\mathcal L}}
\newcommand{\CM}{{\mathcal M}}

\newcommand{\CO}{{\mathcal O}}
\newcommand{\CP}{{\mathcal P}}

\newcommand{\CX}{{\mathcal X}}
\newcommand{\CY}{{\mathcal Y}}

\newcommand{\blangle}{\big\langle}
\newcommand{\brangle}{\big\rangle}

\newcommand{\pt}{{\mathsf{p}}}
\newcommand{\td}{{\mathrm{td}}}

\DeclareFontFamily{OT1}{rsfs}{}
\DeclareFontShape{OT1}{rsfs}{n}{it}{<-> rsfs10}{}
\DeclareMathAlphabet{\curly}{OT1}{rsfs}{n}{it}

\newcommand\Ext{\operatorname{Ext}}

\newcommand{\Aut}{\operatorname{Aut}}
\newcommand{\p}{\mathbb{P}}

\newcommand{\Mbar}{{\overline M}}
\newcommand{\vir}{{\text{vir}}}

\newcommand{\Coh}{\mathrm{Coh}}

\newcommand{\Pic}{\mathop{\rm Pic}\nolimits}
\newcommand{\PT}{\mathsf{PT}}
\newcommand{\DT}{\mathsf{DT}}
\newcommand{\dt}{\mathsf{dt}}

\newcommand{\GW}{\mathsf{GW}}
\newcommand{\KM}{\mathsf{KM}}

\newcommand{\VW}{\mathsf{VW}}

\newcommand{\kk}{\mathsf{k}}

\newcommand\ev{\operatorname{ev}}
\newcommand{\QMod}{\mathsf{QMod}}
\newcommand{\Mod}{\mathsf{Mod}}

\newcommand{\Hilb}{\mathsf{Hilb}}

\newcommand{\Span}{\mathrm{Span}}
\newcommand{\Stab}{\mathrm{Stab}}

\newcommand{\id}{\mathrm{id}}
\newcommand{\Tr}{\mathrm{Tr}}

\newcommand{\ch}{\mathsf{ch}}

\newcommand{\wt}{\mathsf{wt}}

\newcommand{\rel}{\mathsf{rel}}

\newcommand{\AHJ}{\mathsf{AHJac}}
\newcommand{\QJac}{\mathsf{QJac}}

\newcommand{\Jac}{\mathsf{Jac}}
\newcommand{\SL}{\mathrm{SL}}
\newcommand{\GL}{\mathrm{GL}}

\newcommand{\ct}{\mathrm{ct}}
\newcommand{\AH}{\mathrm{AH}}

\newcommand{\pr}{\mathrm{pr}}
\DeclareMathOperator{\Wt}{\mathsf{WT}}

\newcommand{\Mon}{\mathrm{Mon}}
\newcommand{\DMon}{\mathrm{DMon}}

\newcommand{\taut}{{\mathrm{taut}}}

\begin{document}
\baselineskip=14.5pt
\title{Curve counting on the Enriques surface and the Klemm-Mari\~{n}o formula}

\author{Georg Oberdieck}

\address{KTH Royal Institute of Technology, Department of Mathematics}
\email{georgo@kth.se}
\date{\today}

\begin{abstract}
We determine the Gromov-Witten invariants of the local Enriques surfaces for all genera and curve classes and prove the Klemm-Mari\~{n}o formula.
In particular, we show that the generating series of genus $1$ invariants of the Enriques surface is the Fourier expansion of a certain power of Borcherds automorphic form on the moduli space of Enriques surfaces. We also determine all Vafa-Witten invariants of the Enriques surface.

The proof uses the 
correspondence between Gromov-Witten and Pandharipande-Thomas theory.
On the Gromov-Witten side we prove the relative Gromov-Witten potentials of an elliptic Enriques surfaces are quasi-Jacobi forms and satisfy a holomorphic anomaly equation.
On the sheaf side, we relate the Pandharipande-Thomas invariants of the Enriques-Calabi-Yau threefold in fiber classes to the $2$-dimensional Donaldson-Thomas invariants
by a version of Toda's formula for local K3 surfaces.
Altogether, we obtain sufficient modular constraints to determine all invariants from basic geometric computations.
\end{abstract}

\maketitle

\setcounter{tocdepth}{1} 
\tableofcontents

\section{Introduction}
\subsection{Main result}
An Enriques surface is a smooth complex projective surface $Y$ 
with non-trivial canonical bundle satisfying $\omega_Y^{\otimes 2} \cong \CO_Y$ and 
$H^1(Y,\CO_Y)=0$. 
Equivalently, an Enriques surface is the quotient of a K3 surface by a fixed point free involution.
Let $Y$ be an Enriques surface
and let $\beta \in H_2(Y,\BZ)$ be a class {\em modulo torsion}.\footnote{Throughout the paper, $H_k(X,\BZ)$ and $H^k(X,\BZ)$ will denote homology and cohomology groups of a topological space $X$ {\em modulo torsion}.
A discussion on the dependence on the torsion can be found in Appendix~\ref{appendix:torsion}.}

The moduli space $\Mbar_g(Y,\beta)$ of genus $g$ degree $\beta$ stable maps to $Y$ 
has a virtual fundamental class of dimension $g-1$,
\[ [ \Mbar_{g}(Y,\beta) ]^{\vir} \in A_{g-1}( \Mbar_g(Y,\beta) ). \]
Let $\BE \to \Mbar_g(Y,\beta)$ be the Hodge bundle
which has fiber $H^0(C,\omega_C)$ over a point $[f:C \to Y]$,
and let its Chern classes be denoted by 
\[ \lambda_i = c_i(\BE) \in H^{2i}( \Mbar_g(Y,\beta) ). \]
The most fundamental Gromov-Witten invariant of the Enriques surface
is the Hodge integral
\begin{equation} \label{Ngbeta} N_{g,\beta} := \int_{[ \Mbar_{g}(Y,\beta) ]^{\vir}} (-1)^{g-1} \lambda_{g-1}. \end{equation}
The significance is that $N_{g,\beta}$  equals the Gromov-Witten invariant
of the local Calabi-Yau threefold $K_Y$ given by the total space of the canonical bundle on $Y$.
The analogue of $N_{g,\beta}$ for K3 surfaces 
was determined in the celebrated
Katz-Klemm-Vafa formula \cite{KKV} which was proven in \cite{PTKKV}.
For abelian surfaces an analogue of $N_{g,\beta}$ was considered in \cite{BOPY},
where an explicit formula was conjectured and partially proven.
We also refer to \cite{LP,CI,FRZZ,L2} for other recent work
on Gromov-Witten invariants of local del Pezzo surfaces.

Based on string theory, Klemm and Mari\~{n}o conjectured in \cite{KM1} an explicit formula for $N_{g,\beta}$, following earlier work on the genus $1$ case by Harvey and Moore \cite{HM}.
In \cite{MP} Maulik and Pandharipande gave a proof of this formula in genus $1$, which relied on the conjectural Virasoro constraints for the Enriques surface in genus $2$. The Virasoro constraints are not yet known in this instance, so their proof remains conditional at this point.\footnote{The proof of the Virasoro constraints for the Enriques surface is an interesting open question.}

We recall the formula conjectured by \cite{KM1}, but in a slightly digested form.
Define coefficients $\omega_g(n)$ by the product formula\footnote{In the notation of \cite[(4.25)]{KM2} we have
$\omega_g(n) = \frac{1}{4} 2^{3-2g} c_g(2n)$.}
\begin{align*}
\sum_{g \geq 0} \sum_{n \geq 0} \omega_g(n) (-1)^{g-1} z^{2g-2} q^n = 
\prod_{m \geq 1} \frac{ (1- e^z q^{2m})^2 (1- e^{-z} q^{2m})^2 (1-q^{2m})^4 }{ (1-e^z q^m)^2 (1- e^{-z} q^m)^2 (1-q^m)^{12} }
\end{align*}
where the right hand side is viewed as a power series in formal variables $z,q$.

\begin{conj}[The Klemm-Mari\~{n}o formula, \cite{KM1}] For all $g,\beta$ one has
\[ N_{g,\beta} = 2 \sum_{\substack{\textup{odd } k | \beta}} k^{2g-3} \omega_g\left( \frac{\beta^2}{2 k^2} \right). \]
\end{conj}

The main result of this paper is the following:
\begin{thm} \label{thm:main theorem}
The Klemm-Mari\~{n}o formula holds for all genera $g$ and all curve classes $\beta$. 
\end{thm}

Before explaining the strategy of the proof, we consider
several reformulations
of the Klemm-Mari\~{n}o formula:
In genus $1$, our invariant simply reads
\[ N_{1,\beta} = \int_{[ \Mbar_{1}(Y,\beta) ]^{\vir}} 1. \]
Since elliptic curves on an Enriques surface are expected to be rigid,
this can be seen as a virtual count of elliptic curves on $Y$ in class $\beta$.
The Klemm-Mari\~{n}o formula specializes to:
\begin{cor} \label{cor:genus 1} We have
\[ \exp\left( \sum_{\beta \neq 0} N_{1,\beta} q^{\beta} \right) = \prod_{\beta > 0} \left( \frac{1+q^{\beta}}{1-q^{\beta}} \right)^{a(\beta^2/2)} \]
where the coefficients $a(n)$ are defined by
\[
\sum_{n \geq 0} a(n) q^n = \prod_{n \geq 1} \frac{(1+q^n)^8}{(1-q^n)^8} 
=
1 + 16q + 144q^{2} + 960q^{3} + 5264q^{4} + \ldots
\]
and $q^{\beta}$
is the canonical basis element in the group ring $\BC[H_2(Y,\BZ)]$
completed along the cone of effective curve class ($\beta > 0$ stands for $\beta$ effective).
\end{cor}

Consider the lattice $M = U \oplus U(2) \oplus E_8(-2)$
and the associated Hermitian symmetric domain of type IV,
\[ \CD_M = \{ x \in \p(M \otimes \BC) | x \cdot x = 0, x \cdot \overline{x} > 0 \}. \]
The moduli space of Enriques surfaces is isomorphic to the arithmetic quotient $\CD_M/O(M)$ with an irreducible hypersurface (corresponding to mildly singular Enriques surfaces, the Coble surfaces) removed, see \cite{EnriquesBook}.
Let $\Phi(t)$ be Borcherds weight $4$ automorphic form on $\CD_M$ for the group $O(M)$ \cite{BorcherdsEnriques}.
The form $\Phi(t)$ vanishes precisely on the locus of singular Enriques,
and admits Fourier expansions around the two cusps of the moduli space of Enriques surfaces.
Corollary~\ref{cor:genus 1} then precisely says that
the generating series $\exp( \sum_{\beta \neq 0} N_{1,\beta} q^{\beta} )$
is the Fourier expansion of
the automorphism form
\[ \frac{1}{\Phi(t)^{1/8}} \]
around one of these cusps
(more precisely, the 'level $1$ cusp' as named in \cite[Sec.7.2]{Yoshikawa}.)

This observation matches the main prediction of genus $1$ mirror symmetry,
which relates the genus $1$ Gromov-Witten invariants in the $A$-model to the analytic torsion in the $B$-model \cite{BCOV, HM, EFM}. Here, Enriques surfaces are mirror to itself
and the analytic torsion of an Enriques surfaces
was computed by Yoshikawa in terms of $\Phi(t)$ in \cite{YoshikawaInvent}.

We can also compare Corollary~\ref{cor:genus 1} to the case when the target variety is an elliptic curve $E$, for which the genus $1$ Gromov-Witten invariants $N_{1,d}^E$ are given by
\[ q^{-1/24} \exp\left( \sum_{d \geq 1} N_{1,d}^E q^d \right) = \frac{1}{\Delta(q)^{1/24}}, \quad \Delta(q) = q \prod_{n \geq 1} (1-q^n)^{24}. \]
The weight $12$ cusp form $\Delta(q)$ on the upper half plane 
plays here the role of $\Phi(t)$.

The automorphic form $\Phi(t)$ satisfies a natural second order differential equation (more precisely, it lies in the kernel of the heat operator).
As noted in \cite[4.4]{KM2}, this translates to the following recursion 
which appeared first in the work of Maulik and Pandharipande \cite{MP}:
\begin{cor} For all $\beta \neq 0$ we have
\[ (\beta,\beta) N_{1,\beta} = 8 \sum_{\substack{ \beta_1 + \beta_2 = \beta \\ \beta_1, \beta_2 > 0}} (\beta_1, \beta_2) N_{1,\beta_1} N_{1,\beta_2}. \]
\end{cor}

In higher genus, we can state the full topological string partition function
of the local Enriques $K_Y$: Consider the genus $g$ Gromov-Witten potential
\[ F^{K_Y}_g(q) = \sum_{\beta > 0} q^{\beta} N_{g,\beta}. \]

\begin{cor}
The partition function of the local Calabi-Yau threefold $K_Y$ is
\[
\exp\left( \sum_{g \geq 1} F^{K_Y}_g(q) (-1)^{g-1} z^{2g-2} \right)
=
\prod_{\beta > 0} \prod_{r \in \BZ} \left( \frac{1 + e^{rz} q^{\beta}}{1- e^{rz} q^{\beta}} \right)^{\omega(r,\beta^2/2)} \]
where the coefficients
$\omega(r,n)$ are defined by
\begin{align*}
\sum_{n \geq 0} \sum_{r \in \BZ} \omega(r,n) p^r q^n & = 
\prod_{m \geq 1} \frac{ (1-p q^{2m})^2 (1-p^{-1} q^{2m})^2 (1-q^{2m})^4 }{ (1-p q^m)^2 (1-p^{-1} q^m)^2 (1-q^m)^{12} }  \\
& = 1 + (2 p^1 + 12 + 2 p^{-1})q + 
\ldots 
\end{align*}
\end{cor}

It is very rare in enumerative geometry of Calabi-Yau threefolds
to see completely explicit evaluations.
The invariants of the local Enriques appear as the fiber class invariants
of the (compact) Enriques Calabi-Yau threefold $Q=(X \times E)/\BZ_2$,
where $X \to Y$ is the covering K3 surface and $E$ is an elliptic curve.
Hence the partition function of the local Enriques
is the specialization of the partition function of a compact Calabi-Yau threefold.
It would be interesting to obtain an explicit evaluation of the full partition function of $Q$, see \cite{KM2} for a conjectural computation up to genus $6$.

Another equivalent reformulation of our computation of $N_{g,\beta}$ concerns sheaf counting on the local Calabi-Yau threefold $K_Y$.
Tanaka and Thomas \cite{TT1, TT2} gave a conjectural definition\footnote{After this paper appeared on the arXiv, the author was made aware of the recent preprint of Liu \cite{Liu} in which the Tanaka-Thomas conjecture \cite[Conjecture 1.2]{TT2} is proven in general.}
of an invariant $\VW_S(v)$ which counts compactly supported semi-stable sheaves $F$ on local surfaces $K_S$ with Chern character $\ch(p_{\ast} F) = v$, where $p : K_S \to S$ is the projection.
The {\em Vafa-Witten invariant} $\VW_S(v)$
is 
a mathematical incarnation of the 
sheaf counting invariant envisioned by Vafa and Witten \cite{VW} using string theory.
We prove the following for the Enriques surface:

\begin{cor}  \label{cor:Vafa Witten}
Conjecture 1.2 in \cite{TT2} holds for the Enriques surface. In particular,
the Vafa-Witten invariant of $Y$ is well-defined. It is given by
\[ \VW_Y(r,\beta,n) = 2 \sum_{\substack{k|(r,\beta,n) \\ k \geq 1 \textup{ odd}}}
\frac{1}{k^2} e\left( \Hilb^{\frac{\beta^2 - 2rn - r^2}{2k^2} + \frac{1}{2}}(Y) \right). \]
where $e(\Hilb^n Y)$ is the topological Euler number of the Hilbert scheme of $n$ points.
\end{cor}

Concretely, by G\"ottsche's formula \cite{Goettsche} we have
\[ \sum_{n=0}^{\infty} e(\Hilb^n Y) q^n = \prod_{n \geq 1} \frac{1}{(1-q^n)^{12}} \]
and we set $e(\Hilb^n Y)=0$ if $n$ is negative or fractional.

Corollary~\ref{cor:Vafa Witten} is an analogue of the multiple cover formula for sheaf counting invariants on local K3 surfaces conjectured by Toda \cite{Toda} and proven by Maulik and Thomas \cite{MT}.

\subsection{Overview of the proof}
Let $X \to Y$ be the covering K3 surface of the Enriques and let $\tau$ be the covering involution. Consider the Enriques Calabi-Yau threefold
\[ Q = (X \times E)/ \langle (\tau, -1) \rangle. \]
The Calabi-Yau threefold $Q$ admits a K3 fibration
\[ p : Q \to E/\langle -1 \rangle = \p^1 \]
with $4$ double Enriques fibers.

There are three different counting invariants that can be defined for $Q$:
\begin{enumerate}
\item[(i)] Gromov-Witten invariants of $Q$ in fiber classes counting stable maps,
\item[(ii)] Pandharipande-Thomas invariants of $Q$ in fiber classes counting stable pairs,
\item[(iii)] Generalized Donaldson-Thomas invariants of $Q$
counting $2$-dimensional sheaves supported on fibers of $p$.
\end{enumerate}
The Gromov-Witten invariants of $Q$ in fiber classes are equal (up to factor of $4$)
to the Hodge integrals $N_{g,\beta}$ we are interested in.
Moreover, 
the correspondence between Gromov--Witten theory and Pandharipande--Thomas theory conjectured in \cite{MNOP, PT} was proven for the threefold $Q$ by Pandharipande and Pixton in \cite{PP}.
This yields a correspondence between (i) and (ii).
Here 
we prove that (ii) and (iii) determine each other by a simple formula.
More precisely, we show a version of Toda's formula for $K3 \times \BC$ \cite{Toda},
see Theorem~\ref{thm:Todas formula}.\footnote{Recently, Feyzbakhsh-Thomas \cite{FT} proved that for many Calabi-Yau threefolds, Pandharipande-Thomas invariants and 2-dimensional generalized Donaldson-Thomas invariants determine each other by a very general but usually quite complicated formula. We do not need their more advanced result here.}

In conclusion all three counting theories are equivalent.

The theories (i) and (iii) yields quite different constraints on the counting invariants.
On the Gromov-Witten side we prove that
for an elliptic fibration $\pi : Y \to \p^1$
on the Enriques surface
the $\pi$-relative generating series of descendent Gromov-Witten invariants of $Y$
are $\Gamma_0(2)$ quasi-Jacobi forms for the $E_8$-lattice and satisfy a holomorphic anomaly equation (Theorem~\ref{thm:HAE}).
This opens the door to using tools from the theory of modular forms.
In particular we use that, for any $a>2$, a $\Gamma_0(2)$-modular form is uniquely determined by those of its $n$-th Fourier coefficients where $n$ is not divisibile by $a$ (Lemma~\ref{lemma:quasimodular vanishing}).
In other words, the $n$-th Fourier coefficients, where $a$ does not divide $n$,
determine those where $a$ does divide $n$.
This will allow us to determine invariants of higher divisibility from lower divisibility.
On the sheaf side,
autoequivalences of $Y$ will imply the key property that
$N_{g,\beta}$ only depends upon $\beta$ through the square $\beta^2$ and the divisibility $\mathrm{div}(\beta)$.
The final argument can then be best described as Sudoku: Everything is determined from these conditions by modularity constraints and a few geometric computations (essentially we only use G\"ottsche's formula for the Euler characteristic of the Hilbert scheme of $Y$).

\subsection{Plan of the paper}
In Section~\ref{sec:Modular and Jacobi forms} we introduce the background from modular and Jacobi forms that we need. Proofs for this part are deferred to Appendix~\ref{sec:Appendix Background quasi-Jacobi forms}.
In Section~\ref{sec:Mon Enriques} we determine the derived monodromy group of the Enriques in terms of the Mukai lattice of the covering K3 surface,
and classify the orbits of primitive invariant Mukai vectors under its action.
In Section~\ref{sec:GW theory} we study the Gromov-Witten theory of an elliptic Enriques surface
by a degeneration to the rational elliptic surface. The main result is the holomorphic anomaly equation for the relative potentials (Theorem~\ref{thm:HAE}).
In Section~\ref{sec:DT of Enriques CalabiYau} we consider the Enriques-Calabi-Yau threefold and its three counting theories.
First, by using the derived momodromy group of the Enriques, we show that the generalized 2-dimensional Donaldson-Thomas invariants only depend on the square, the divisibility and the type. Then we prove Toda's formula and discuss its consequences. 
In Section~\ref{sec:putting everything together} we put the constraints from Gromov-Witten and Donaldson-Thomas theory together and finish the proof of Theorem~\ref{thm:main theorem}.
Section~\ref{sec:Vafa Witten} proves Corollary~\ref{cor:Vafa Witten} on the Vafa-Witten invariants.

\subsection{Acknowledgements}
I thank Honglu Fan, Daniel Huybrechts, Jan-Willem van Ittersum, Giacomo Mezzedimi, 
Rahul Pandharipande, Maximilian Schimpf, and Richard Thomas for discussions on the Enriques surface.
I also thank the referees for careful reading and useful comments.
The author was supported by the starting grant 'Correspondences in enumerative geometry: Hilbert schemes, K3 surfaces and modular forms', No 101041491
 of the European Research Council.
 

\section{Modular and Jacobi forms} \label{sec:Modular and Jacobi forms}
Modular forms are holomorphic functions $f : \BH \to \BC$
on the upper half plane $\BH = \{ \tau \in \BC : \mathrm{Im}(\tau) > 0 \}$
which satisfy a transformation property with respect to a congruence subgroup of $\SL_2(\BZ)$ and are bounded at infinity. Jacobi forms are generalizations of modular forms that also depend on elliptic parameters $x=(x_1,\ldots,x_n)$. Quasi-Jacobi forms are holomorphic parts of non-holomorphic Jacobi forms.
For detailed expositions on these subjects we refer to \cite{123, Koblitz} for modular forms, \cite{EZ, Ziegler} for Jacobi forms and \cite{Lib, RES, IOP} for quasi-Jacobi forms. 

Here we introduce only the background we need later on. In particular,
Proposition~\ref{prop:theta identity} is used in Section~\ref{subsec:Consequences of Toda} and not before.
Lattice quasi-Jacobi forms (Section~\ref{subsec:Lattice index quasi Jacobi forms}) and their Hecke operators (Section~\ref{subsec:Hecke operators}) appear only in Sections~\ref{sec:GW theory} and~\ref{sec:putting everything together}.
The vanishing results of Section~\ref{subsec:Vanishing results} only appear
in the final step of Section~\ref{proof:main thm}.

\subsection{Modular forms}
For $\Gamma \subset \mathrm{SL}_2(\BZ)$ a congruence subgroup,
let $\Mod_k(\Gamma)$ and $\QMod_k(\Gamma)$ be the vector space of weight $k$ modular and quasi-modular forms for $\Gamma$. The algebra of (quasi)-modular forms is $\Mod(\Gamma) = \oplus_k \Mod_k(\Gamma)$ and $\QMod(\Gamma) = \oplus_k \QMod_k(\Gamma)$.  
Throughout we identify a quasi-modular form $f(\tau)$
with its Fourier expansion in the variable $q=e^{2 \pi i \tau}$. We often write $f(q)$ instead of $f(\tau)$ by a slight abuse of notation.

In the case $\Gamma = \SL_2(\BZ)$, we drop the group from the notation.
The basic examples here are the weight $k$ Eisenstein series
defined  for odd $k>0$
by $G_k=0$, and for even $k>0$ by
\[ G_k(\tau) = - \frac{B_k}{2 \cdot k} + \sum_{n \geq 1} \sum_{d|n} d^{k-1} q^n. \]
Each $G_k$ is modular for $k>2$ and $G_2$ is quasi-modular for $\SL_2(\BZ)$. We have
\[ \Mod = \BC[G_4, G_6], \quad \QMod = \BC[G_2, G_4, G_6]. \]
Define the Dedekind function $\eta(\tau) = q^{1/24} \prod_{n \geq 1} (1-q^n)$. We have
\[ \Delta(\tau) := \eta^{24}(\tau) \in \Mod_{12}. \]

More generally, consider the groups $\Gamma_0(N) = \{ \binom{a\ b}{c\ d} \in \mathrm{SL}_2(\BZ) | c \equiv 0 (N) \}$.
The series
\begin{equation} F_2(\tau) = G_2(\tau) - 2 G_2(2 \tau) = \frac{1}{24} + \sum_{\substack{\text{odd } d|n}} d q^n, \label{F2} \end{equation}
is a modular form of weight $2$ for $\Gamma_0(2)$ and we have \cite[Sec.12]{Borcherds}
\[ \Mod(\Gamma_0(2)) = \BC[ F_2, G_4 ], \quad \QMod(\Gamma_0(2)) = \BC[ G_2, F_2, G_4 ]. \]
Moreover, it is well-known \cite{Koblitz} that, for $f \in \Mod_k$ we have $f(2 \tau) \in \Mod_k(\Gamma_0(2))$, and
\begin{gather} \label{example modualr forms for gamma02}
(\eta(\tau) \eta(2 \tau))^8 \in \Mod_8(\Gamma_0(2)),
\quad \quad
\frac{\Delta(\tau)^2}{\Delta(2 \tau)} \in \Mod_{12}(\Gamma_0(2)).
\end{gather}

\begin{lemma} \label{lemma:function vanishing at cusp}
The function $f(\tau) = \eta^{16}(2 \tau) / \eta^{8}(\tau)$ is a modular form for $\Gamma_0(2)$ of weight $4$ which vanishes at the cusp $\tau = i \infty$ and is non-vanishing at $\tau = 0$.
\end{lemma}
\begin{proof}
That $f$ is modular of the given weight follows from \eqref{example modualr forms for gamma02}. Its Fourier expansion is $f=q + O(q^2)$ so it vanishes at $\tau = i \infty$. Moreover, $\tau^{-4} f(-1/\tau) = \frac{1}{2^8} + O(q)$ by the transformation property of the $\eta$-function \cite{Koblitz}, which shows the non-vanishing at $\tau=0$.
\end{proof}

\subsection{Jacobi forms in one elliptic variable} \label{subsec:Jacobi forms one variable}
Let $z \in \BC$, $\tau \in \BH$, and set $p=e^z$, $q=e^{2 \pi i \tau}$. 
Consider the renormalized odd Jacobi theta function:
\[ \Theta(z,\tau) = \frac{1}{\eta^3(\tau)} \sum_{\nu\in \mathbb{Z}+\frac{1}{2}} (-1)^{\lfloor \nu \rfloor} e^{z \nu} q^{\nu^2/2}. \]

We require the following set of identities:
\begin{prop} \label{prop:theta identity} We have
\begin{align*}
 \frac{\Theta(z,2 \tau)^2}{\Theta(z,\tau)^2} \frac{\eta(2 \tau)^8}{\eta(\tau)^4} \ 
& \overset{(i)}{=} \ q^{1/2} \prod_{m \geq 1} \frac{ (1-p q^{2m})^2 (1-p^{-1} q^{2m})^2 (1-q^{2m})^4  }{ (1-pq^m)^2 (1-p^{-1} q^m)^2 } \\
 & \overset{(ii)}{=}
\frac{1}{4} \left( \frac{1}{\Theta\left( \frac{z}{2}, \frac{\tau}{2} \right)^2} - \frac{1}{\Theta\left( \frac{z}{2}, \frac{\tau+1}{2} \right)^2} \right)  \\
& \overset{(iii)}{=}
\sum_{\substack{r \geq 1 \\ r \text{ odd}}} \left( r q^{r^2/2} + 
\sum_{n \geq 1} (n+r) (p^n+p^{-n}) q^{rn + r^2/2} \right) \\
& = 
q^{1/2} + 2 (p + p^{-1}) q^{3/2} + 3 (p^2 + p^{-2}) q^{5/2} + \ldots 
\end{align*}
\end{prop}
\begin{proof}
Equality (i) follows by the well-known Jacobi triple product which reads:
\[ \Theta(z,\tau) =  (p^{1/2}-p^{-1/2})\prod_{m\geq 1} \frac{(1-pq^m)(1-p^{-1}q^m)}{(1-q^m)^{2}}. \]
We prove (ii). Let
\[
F(z) = 
\frac{1}{4} \frac{\Theta(z,\tau)^2}{\Theta(z,2 \tau)^2} \frac{\eta(\tau)^4}{\eta(2 \tau)^8}
\left( \frac{1}{\Theta\left( \frac{z}{2}, \frac{\tau}{2} \right)^2} - \frac{1}{\Theta\left( \frac{z}{2}, \frac{\tau+1}{2} \right)^2} \right)
\]
By a direction computation
or since $\Theta(z,\tau)^2$ is a (weak) Jacobi form of index $1$ we have
\[ \Theta( z + 2 \pi i \lambda \tau + 2 \pi i \mu, \tau)^2= q^{-\lambda^2} p^{-2\lambda} \Theta(z,\tau)^2 \]
for all $\lambda, \mu \in \BZ$. One finds that
\[ F(z+4 \pi i) = F(z), \quad F(z+4 \pi i \tau) = F(z). \]
The theta function $\Theta(z)$ has a simple zero at each of the lattice points $\frac{z}{2 \pi i} \in \BZ+ \BZ \tau$. Hence $F$ can have poles only at the points $\frac{z}{2 \pi i} \in \{ 0, 1, \tau, 1+\tau \} + 2 \BZ + 2 \BZ \tau$. By a direct check $F$ is holomorphic at all of these points and the constant term at $z=0$ is one. Hence $F(z) = 1$.

For (iii) we use the following identity of Zagier \cite[Sec.3]{Zagier}:
\[
\frac{\Theta(z+w)}{\Theta(z) \Theta(w)}
=
\frac{1}{2} \left( \coth \frac{w}{2}  + \coth \frac{z}{2} \right)
-2 \sum_{n=1}^{\infty} \left( \sum_{d|n} \sinh( dw + \frac{n}{d} z ) \right) q^n. \] 
By computing $\frac{d}{dw}|_{w=-z}$ on both sides, one gets
\begin{equation} \label{30swdifsd}
\frac{1}{\Theta^2(z,\tau)} = \frac{1}{(p^{1/2} - p^{-1/2})^2} + \sum_{r \geq 1} \left( 2 r q^{r^2} + \sum_{n \geq 1} (2r + n) (p^n + p^{-n}) q^{rn+r^2} \right).
\end{equation}
Observe that
\[
\frac{1}{2} \left( \frac{1}{\Theta\left( \frac{z}{2}, \frac{\tau}{2} \right)^2} - \frac{1}{\Theta\left( \frac{z}{2}, \frac{\tau+1}{2} \right)^2} \right)
=
\left[ \frac{1}{\Theta(z,\tau)^2} \right]_{q^{\text{odd}}} (z/2, \tau/2)
\]
Hence (iii) follows immediately from \eqref{30swdifsd} by taking the odd $q$-exponents. 
\end{proof}

We also note the expansion
\begin{equation} \label{theta gk expansion}
\frac{\Theta(z,2 \tau)^2}{\Theta(z,\tau)^2}
=
\exp\left( 4 \sum_{k \geq 2} ( G_k(\tau) - G_k(2 \tau) ) \frac{z^{k}}{k!} \right)
\end{equation}
which follows by the well-known Taylor expansion:
$\Theta(z) = z \exp(-2\sum_{k\geq 2} G_k z^k/k!)$.

\subsection{Lattice index quasi-Jacobi forms} \label{subsec:Lattice index quasi Jacobi forms}
Let\footnote{If $n=1$, the variable $x \in \BC$ is related to the variable $z \in \BC$ of Section~\ref{subsec:Jacobi forms one variable} by $z = 2 \pi i x$.} $x = (x_1, \ldots, x_n) \in \BC^n$ and consider the following real analytic functions on $\BC^n \times \BH$:
\[ 
\nu(\tau) = \frac{1}{8 \pi \mathrm{Im}(\tau)},
\quad
\alpha_i(x,\tau) = \frac{x_i - \overline{x_i}}{\tau - \overline{\tau}} = \frac{\mathrm{Im}(x_i)}{\mathrm{Im}(\tau)},
\quad i = 1, \ldots, n.
\]
An \emph{almost holomorphic function} on $\BC^n \times \BH$ is a function
$\Phi : \BC^n \times \BH \to \BC$ of the form
\[ \Phi(x, \tau) = \sum_{i \geq 0} \sum_{j = (j_1, \ldots, j_n) \in (\BZ_{\geq 0})^n}
\phi_{i, j}(x,\tau) \nu^{i} \alpha^j, \quad \quad \alpha^j = \alpha_1^{j_1} \cdots \alpha_n^{j_n} \]
such that each of the finitely many non-zero $\phi_{i, j}(z,\tau)$ is holomorphic
on $\BC^n \times \BH$.

Write $R^{(m,n)}$ for the group of $m \times n$-matrices with coefficients in a ring $R$. Consider a congruence subgroup
and a finite index subgroup
\[ \Gamma \subset \SL_2(\BZ), \quad \Lambda \subset \BZ^{(n,2)} \]
such that $\Lambda$ is preserved under the action of $\Gamma$ on $\BZ^{(n,2)}$ by multiplication on the right. 

An {\em index} is a symmetric rational $n \times n$-matrix $L$ such that
\[ \Tr(L \kappa) \in \BZ \text{ for all symmetric } \kappa \in \Span_{\BZ}( \mu \lambda^t, \lambda \mu^t | (\lambda,\mu) \in \Lambda ). \]
For example, if $\Lambda = \BZ^{(n,2)}$
this says that 
$L_{ij} \in \frac{1}{2} \BZ$ and $L_{ii} \in \BZ$ for all $i,j$.

\begin{defn} \label{defn:quasi jacobi forms main text}
An \emph{almost holomorphic weak Jacobi form} of weight $k$ and index $L$
for the group $\Gamma \ltimes \Lambda$ is an almost-holomorphic function $\Phi(x,\tau) : \BC^n \times \BH \to \BC$ satisfying:
\begin{enumerate}
\item[(i)] For all $\binom{a\ b}{c\ d} \in \Gamma$ and $(\lambda, \mu) \in \Lambda$ we have\footnote{We write
$e(x) = e^{2\pi i x}$ for $x \in \BC$.}
\begin{equation} \label{TRANSFORMATIONLAWJACOBI*}
\begin{aligned}
\Phi\left( \frac{x}{c \tau + d}, \frac{a \tau + b}{c \tau + d} \right)
& = (c \tau + d)^k e\left( \frac{c x^t L x}{c \tau + d} \right) \Phi(x,\tau) \\
\Phi\left( x + \lambda \tau + \mu, \tau \right)
& = e\left( - \lambda^t L \lambda \tau - 2 \lambda^t L x \right) \Phi(x,\tau).
\end{aligned}
\end{equation}
\item[(ii)] 
For all $\binom{a\ b}{c\ d} \in \SL_2(\BZ)$, the almost-holomorphic function 
\[ (c \tau + d)^{-k} e\left( - \frac{ c x^t L x }{c \tau + d} \right) \Phi\left( \frac{x}{c \tau + d}, \frac{a \tau + b}{c \tau + d} \right) \]
is of the form $\sum_{i,j} \phi_{i,j} \alpha^i \nu^j$ such that each $\phi_{i,j}$ is holomorphic on $\BC^n \times \BH$ and for some $N \geq 1$ admits a Fourier expansion of the form $\sum_{v \in \BZ_{\geq 0}} \sum_{r \in \BZ^n} c(v,r) q^{v/N} e\left( \frac{1}{N} \sum_i x_i r_i \right)$ in the region $|q|<1$. 
\end{enumerate}
\end{defn}

\begin{defn}
A quasi-Jacobi form of weight $k$ and index $L$ 
for the group $\Gamma \ltimes \Lambda$
is the coefficient of $\nu^0 \alpha^0$
of an almost-holomorphic Jacobi form of the same kind.

A Jacobi form is an almost-holomorphic Jacobi form which is already a holomorphic function (i.e. does not depend on $\alpha$ and $\nu$).\footnote{This is called a weak Jacobi form in \cite{EZ}, but we drop the word 'weak' here since we do not need the distinction.}

The vector spaces of almost-holomorphic Jacobi forms,
of quasi-Jacobi forms, and of Jacobi forms
of weight $k$ and index $L$ for $\Gamma \ltimes \Lambda$
are denoted by:
\[
\AHJ_{k,L}(\Gamma \ltimes \Lambda),\quad  \QJac_{k,L}(\Gamma \ltimes \Lambda), \quad \Jac_{k,L}(\Gamma \ltimes \Lambda).
\]
If $\Lambda = \BZ^{(n,2)}$ we write $\Gamma$ instead of $\Gamma \ltimes \Lambda$. If $\Gamma = \SL_2(\BZ)$ we often drop it from the notation.
The space of quasi-Jacobi forms of index $L$ is denoted by
\[
\QJac_{L} = \bigoplus_{k} \QJac_{k,L}. \]
\end{defn}

Taking the constant term of an almost-holomorphic Jacobi form
defines an isomorphism:
\begin{equation} \mathrm{ct} : \AHJ_{k,L}(\Gamma \ltimes \Lambda) \xrightarrow{\cong}
\QJac_{k,L}(\Gamma \ltimes \Lambda), \quad 
\sum_{i,j} \phi_{i,j} \alpha^i \nu^j \mapsto \phi_{0,0}. \label{constant term map} \end{equation}
\begin{defn}[{\cite[Sec.1]{RES}}] Define the holomorphic anomaly operators by:\footnote{As explained in \cite[Sec.1.3.4]{RES} one can interprete the first holomorphic anomaly operator as the formal derivative in the second Eisenstein series $G_2(\tau)$.}
\[ \frac{d}{dG_2} :=  \mathrm{ct} \circ \frac{d}{d\nu} \circ \mathrm{ct}^{-1}:
\QJac_{k,L}(\Gamma \ltimes \Lambda)
\to \QJac_{k-2,L}(\Gamma \ltimes \Lambda). \]
and for all $\lambda = (\lambda_1, \ldots, \lambda_n) \in \BZ^n$,
\[ \xi_{\lambda} := \mathrm{ct} \circ \left( \sum_i \lambda_i \frac{d}{d\alpha_i} \right) \circ \mathrm{ct}^{-1}:
\QJac_{k,L}(\Gamma \ltimes \Lambda) \to \QJac_{k-1,L}(\Gamma \ltimes \Lambda). \]
\end{defn}

\subsection{Jacobi forms for the $E_8$-lattice}
\label{subsec:Jacobi forms for E8}
Let $E_8 = (\BZ^8,\ \cdot\ )$ be the unique even integral unimodular lattice of signature $(8,0)$.
Let $b_1, \ldots, b_8$ be an integral basis of $E_8$,
and let
\[ Q_{E_8} = ( b_i \cdot b_j )_{i,j=1}^{8} \]
be the intersection matrix of the lattice in this basis.
For example, for a suitable basis,
\[ Q_{E_8} = \begin{pmatrix}
2 &  & -1 & & & & & \\
& 2 & & -1 & & & & \\
-1 & & 2 & -1 & & & \\
& -1 & -1 & 2 & -1 & & & \\
& & & -1 & 2 & -1 & & \\
& & & & -1 & 2 & -1 & \\
& & & & & -1 & 2 & -1 \\
& & & & & & -1 & 2
\end{pmatrix}
\]

We identify $x = (x_1, \ldots, x_8) \in \BC^8$ with $\sum_i x_i b_i$,
and for $\alpha \in E_8$ we write
\[ \zeta^{\alpha} = e(x \cdot \alpha) 
= \prod_i \zeta_i^{b_i \cdot \alpha}, \quad \text{ where } \quad \zeta_i = e^{2 \pi i x_i}. \]
The theta function of the $E_8$-lattice is defined by:
\[ \Theta_{E_8}(\zeta, q) = \sum_{ \alpha \in E_8 } \zeta^{\alpha} q^{\alpha \cdot \alpha/2}. \]
It is a Jacobi form of weight $4$ and index $\frac{1}{2} Q_{E_8}$
for the group $\SL_2(\BZ) \ltimes \BZ^{(8,2)}$, see \cite{Ziegler}.

\begin{lemma} \label{lemma: QJac scaling to Gamma0(2)}
Let $\alpha_0 \in E_8$ be a vector of square $4$.
Let $f(x,\tau)$ be a quasi-Jacobi form for $\SL_2(\BZ) \ltimes (\BZ^8 \oplus \BZ^8)$
of weight $k$ and index $m Q_{E_8}$ where $m \in \BZ_{\geq 0}$.
Then the series
\[
M_{\alpha_0,m}(f)(x,\tau) :=
q^{2m} e\left( m (\alpha_0 \cdot_{E_8} x) \right) \left( e^{\xi_{\alpha_0}/2}  f \right) (x + \alpha_0 \tau, 2 \tau) \]
is a quasi-Jacobi form of weight $k$ and index $\frac{m}{2} Q_{E_8}$
for the group $\Gamma_0(2) \ltimes (2 \BZ^8 \oplus \BZ^8)$.
Moreover,
\[ \frac{d}{d G_2} M_{\alpha_0,m}(f) =
\frac{1}{2} M_{\alpha_0,m}\left( \frac{d}{dG_2} f \right). \]
\end{lemma}
\begin{proof}
This will be proven in Appendix~\ref{subsec:appendix modifications}.
\end{proof}

\subsection{Hecke operators for $\Gamma_0(2)$} \label{subsec:Hecke operators}
For a Laurent series
\[ f(\zeta,q) = \sum_{v} \sum_{r \in \BZ^n} c(v,r) q^v \zeta^r \]
and an integer $\ell \geq 1$ define the series
\begin{equation} \label{Hecke:main text}
 f|_{k,L} V_{\ell}
=
 \sum_{v \geq 0} \sum_{r \in \BZ^n}  \sum_{\substack{ \text{odd }a|(v,r,\ell) }}
 a^{k-1} c\left( \frac{\ell v}{a^2}, \frac{r}{a} \right) q^v \zeta^r
 \end{equation}
 
\begin{prop} \label{prop:Hecke operator}
For every $\ell,s \geq 1$ the mapping $f \mapsto f|_{k,L} V_{\ell}$ defines a homomorphism
\[ \frac{1}{\Delta(q)^{s}} \QJac_{k,L}(\Gamma_0(2)) \to \frac{1}{\Delta(q)^{\ell s}} \QJac_{k, \ell L}(\Gamma_0(2)) \]
such that
\begin{gather*}
\frac{d}{dG_2} (f |_{k,L} V_{\ell}) = \ell \left( \frac{d}{dG_2} f \right)\Big|_{k-2,L} V_{\ell} \\
\xi_{\lambda} (f |_{k,L} V_{\ell}) = \ell (\xi_{\lambda} f)|_{k-1,L}
\end{gather*}
for all $f \in \frac{1}{\Delta(q)^s} \QJac_{k,L}(\Gamma_0(2))$ and $\lambda \in \BZ^n$
\end{prop}
\begin{proof} This is proven in Appendix~\ref{subsec:Hecke operators appendix}. \end{proof}

\subsection{Vanishing results}
\label{subsec:Vanishing results}
\begin{lemma} \label{lemma:non discrete}
Let $m$ be a positive integer with $m>2$. The subgroup $\Gamma$ of $\SL_2(\BR)$ generated by $-\id$ and the matrices
\[ A := \begin{pmatrix} 1 & 1/m \\ 0 & 1 \end{pmatrix}, 
\quad B:= \begin{pmatrix} 1 & 0 \\ 2 & 1 \end{pmatrix}. \]
is dense in $\SL_2(\BR)$.
\end{lemma}
\begin{proof}
We first show that $\Gamma$ is not discrete in $\SL_2(\BR)$. Indeed, the matrix
\[ M = A B^{-1} = \begin{pmatrix} \frac{m-2}{m} & \frac{1}{m} \\ -2 & 1 \end{pmatrix}. \]
has characteristic polynomial $P(t)=\det(M-t I) = t^2 + (-2 + 2/m) t + 1$
and two distinct eigenvalues $\lambda_1, \lambda_2$ of absolute value $1$.
If $\lambda_1$ is a primitive $k$-th root of unity, then its minimal polynomial over $\BQ$ (a cyclotomic polynomial) divides $P(t)$.
Since $\pm 1$ is not a root of $P(t)$ we hence must have that $P(t)$ equals a cyclotomic polynomial,
so must have integer coefficients, hence $m \in \{ 1,2 \}$.
Hence for $m>2$, $\lambda_i$ is not a root of unity, so $\{ M, M^2, M^3, \ldots \} \subset \Gamma$ has an accumulation point. Hence $\Gamma$ is not discrete.

Let now $G$ be the closure of $\Gamma$ in the Lie group $\mathrm{SL}_2(\BR)$.
We need to show that $G = \SL_2(\BR)$.
By the Cartan's closed-subgroup theorem, $G$ is a Lie group,
and since $\Gamma$ is non-discrete, $G$ has dimension $\geq 1$.
The modular group $\Gamma_0(2)$ is generated by
$-I$ and $\binom{1\ 1}{0\ 1}$ and $\binom{1\ 0}{2\ 1}$, see \cite{Koblitz}, and hence $\Gamma_0(2) \subset \Gamma$.
The conjugation action of $\Gamma_0(2)$ on the tangent space $\mathfrak{sl}_2 = T_{\id} \mathrm{SL}_2(\BR)$
is irreducible.\footnote{Concretely, in a suitable basis of $\mathfrak{sl}_2(\BR)$ 
conjugation with $\binom{1\ 1}{0\ 1}$ and $\binom{1\ 0}{2\ 1}$  act by the matrices
\[ 
S=\begin{pmatrix} 1 & 0 & 1 \\ -2 & 1 & -1 \\ 0 & 0 & 1 \end{pmatrix} , \quad 
T=\begin{pmatrix} 1 & -2 & 0 \\ 0 & 1 & 0 \\ 4 & -4 & 1 \end{pmatrix}.
\]
The matrices $S,T$ have a single eigenvalue of geometric multiplicity $1$, but the corresponding eigenvectors are linearly independent. 
Hence any subspace invariant under them is all of $\mathfrak{sl}_2(\BR)$.}
The $\Gamma_0(2)$ action on $\mathfrak{sl}_2$ preserves the tangent space to $G$.
Hence $T_{\id} G$ is a non-zero subrepresentation of $T_{\id} \mathrm{SL}_2(\BR)$, so by Schur's lemma we get $T_{\id} G = \mathfrak{sl}_2$, and since $\SL_2(\BR)$ is connected, we find $G=\mathrm{SL}_2(\BR)$.
\end{proof}

\begin{rmk} \label{rmk:m=2 generation}
If $m=2$, the group generated by $A$, $B$ and $-\id$ in Lemma~\ref{lemma:non discrete} is
the discrete group $C^{-1} \SL_2(\BZ) C$ for $C=\binom{2\ 0}{0\ 1}$.
\end{rmk}

\begin{rmk}
The problem to decide when a finitely generated subgroup of $\mathrm{SL}_2(\BR)$ is discrete
was answered in full generality in \cite{Gilman}.
\end{rmk}

\begin{lemma} \label{lemma:quasimodular vanishing}
Let $m>2$ and let $f(\tau)$ be a quasi-modular form of weight $k$ for $\Gamma_0(2)$
with Fourier expansion $f(\tau) = \sum_{\ell \geq 0} a_{\ell} q^{\ell m}$.
If $k \neq 0$, then $f = 0$. If $k=0$, then $f$ is constant.
\end{lemma}
\begin{proof}
Let $\Gamma$ be the group generated by $\Gamma_0(2)$ and $\binom{1\ 1/m}{0 \ \phantom{/m}1}$.
Because $\binom{1\ 0}{2\ 1} \in \Gamma_0(2)$
we have that $\Gamma$ is define by Lemma~\ref{lemma:non discrete}.
%

Assume first that $f(\tau)$ is a modular form. 
By our assumption on the Fourier expansion of $f$ we have
$f(\tau + 1/m) = f(\tau)$ and hence that
\begin{equation} f\left( \frac{ a \tau + b}{c \tau + d} \right) = (c \tau + d)^k f(\tau) \label{dfsdokfs} \end{equation}
for all $\gamma = \binom{a\ b}{c\ d} \in \Gamma$. 
Since both sides of \eqref{dfsdokfs} depends continously on $\gamma$ and $\tau$, we must have that \eqref{dfsdokfs} holds for all $\gamma$ in the closure of $\Gamma$, that is for all $\gamma \in \SL_2(\BR)$.
The group $\mathrm{SL}_2(\BR)$ acts transitively on the upper half plane $\BH$.
Hence, if $k=0$ we get that $f$ is constant.
If $k \neq 0$, consider the element
$\gamma = \binom{\phantom{-}\lambda\ \mu}{-\mu\ \lambda} \in \SL_2(\BR)$
for some $\lambda, \mu \in \BR$ with $\lambda^2 + \mu^2 = 1$ and $\mu \neq 0$.
Inserting into \eqref{dfsdokfs} we get $f(i) = (-\mu \tau+ \lambda)^k f(i)$, and hence $f(i)=0$.
By the transitivity of the $\mathrm{SL}_2(\BR)$ action, \eqref{dfsdokfs} implies hence $f=0$.

If $f$ is quasi-modular, let $F = \sum_{j} f_j y^{-j}$, where $y=\mathrm{Im}(\tau)$, be the almost-holomorphic modular form whose constant term is $f_0 = f$.
For $\gamma = \begin{pmatrix} 1 & 1/m \\ 0 & 1 \end{pmatrix}$,
observe that $\Gamma_0(2m^2) \subset \gamma^{-1} \Gamma_0(2) \gamma \cap \Gamma_0(2)$.
This implies that $F(\tau + 1/m)$ and $F(\tau)$ are both almost-holomorphic modular forms for $\Gamma_0(2m^2)$. By assumption, their holomorphic parts are equal to the same function $f(\tau) = f(\tau+1/m)$. Since the constant term map \eqref{constant term map} is an isomorphism, we get $F(\tau+1/m) = F(\tau)$.
Hence $F$ satisfies  \eqref{dfsdokfs} for all $\gamma \in \Gamma$, so by the same argument as before it satisfies \eqref{dfsdokfs} for all $g \in \SL_2(\BR)$.
Hence $F$ constant if $k=0$, and $F=0$ if $k \neq 0$.
\end{proof}

\begin{lemma} \label{lemma:taylor}
Let $f(x,\tau)$ be a quasi-Jacobi form for $\Gamma_0(2) \ltimes \BZ^{(n,2)}$ of lattice index $L$ and weight $k$, and let
\[ f(x,\tau) = \sum_{i_1, \ldots, i_n} f_{i_1, \ldots, i_n}(\tau) x_{1}^{i_1} \cdots x_n^{i_n} \]
be its Taylor expansion around the point $x=0$. Then
each $f_{i_1, \ldots, i_{n}}$ is a quasi-modular form for $\Gamma_0(2)$ of weight $k+\sum_{l} i_l$.
\end{lemma}
\begin{proof}
The restriction to $x=0$ of any quasi-Jacobi form for $\Gamma_0(2) \ltimes \BZ^{(n,2)}$ of index $L$ and weight $k$ is a quasi-modular form for $\Gamma_0(2)$,
see \cite[Sec.1.3.5]{RES}.
Moreover, the derivative operator $\frac{d}{d x_i}$ acts on the space of quasi-Jacobi forms of index $L$ and increases the weight by $1$,
see \cite[Sec.1]{RES}.
Hence
\[ f_{i_1, \ldots, i_n} = \frac{1}{i_1! \cdots i_n!} \left( \frac{d}{d x_1} \right)^{i_1} \cdots \left( \frac{d}{d x_n}\right)^{i_n} f(x,\tau) \Big|_{x=0} \]
is a quasi-modular form of weight $\sum_l i_l$ increased.
\end{proof}

\begin{prop} \label{prop:quasi Jacobi form vanishing}
Let $m > 2$ and let $f(x,\tau)$ be a quasi-Jacobi form for some lattice index $L$ and weight $k$
for the group $\Gamma_0(2) \ltimes (\BZ^n \oplus \BZ^n)$.
Assume that $f$ has a Fourier-Jacobi expansion
$f = \sum_{n,r} c(n,r) q^n \zeta^r$
such that $c(n,r) \neq 0$ only if $n=m \ell$ for some $\ell \in \BZ$.
If $L \neq 0$ or $k \neq 0$, we have that $f=0$, otherwise $f$ is constant.
\end{prop}
\begin{proof}
By Lemma~\ref{lemma:taylor} each coefficient in the Taylor expansion of $f$
is a quasi-modular form satisfying the assumptions of Lemma~\ref{lemma:quasimodular vanishing},
and hence is a constant. We see that $f$ does not depend on $q$,
so it vanishes if $(L,k) \neq 0$. If $(L,k)=0$, then $f$ is a $2$-periodic function of $x$ so a constant.
\end{proof}
\begin{rmk} \label{rmk:vanishing with poles}
Since we never used the boundedness at poles in the above argument,
Proposition~\ref{prop:quasi Jacobi form vanishing} holds also for quasi-Jacobi forms in the space
\[ \frac{1}{\Delta(q)^{s}} \QJac_{k,L}( \Gamma_0(2) \ltimes (\BZ^{n} \oplus \BZ^n)), \quad s \geq 0. \]
\end{rmk}

\section{Monodromy and autoequivalences of the Enriques surface}
We study monodromy and auto-equivalences on Enriques surfaces $Y$
by relating them to the covering K3 surface $X \to Y$ and using the global Torelli theorem. In Section~\ref{subsec:monodromy group} we first compute the monodromy group of the Enriques surface.
In Section~\ref{subsec:derived monodromy group} we also consider auto-equivalences and determine in Proposition~\ref{prop:DMon tilde} the derived monodromy group of an Enriques surface.
The classification of the orbits of vectors under the derived monodromy group is given in Section~\ref{subsec:orbit of vectors}.
For an introduction to Enriques surfaces we refer to \cite[VIII]{BHPV} and \cite{EnriquesBook}.

\label{sec:Mon Enriques}
\subsection{Notation}
Let $U = \binom{0\ 1}{1\ 0}$ be the hyperbolic lattice
and let $E_8$ be the unique unimodular even lattice of signature $(8,0)$.
Given a lattice $L$, we write $L(m)$ for the lattice with intersection form multiplied by $m$.
\subsection{Cohomology}
Let $Y$ be an Enriques surface, let $\pi : X \to Y$ be the covering K3 surface,
and let $G = \langle \tau \rangle \cong \BZ_2$ be the group generated by the covering involution $\tau : X \to X$.
There exists an isometry (a 'marking')
\begin{equation} \varphi : H^2(X,\BZ) \xrightarrow{\cong} 
U \oplus U \oplus U \oplus E_8(-1) \oplus E_8(-1)
\label{esdf} \end{equation}
such that $\varphi \tau^{\ast} \varphi^{-1} = \widetilde{\tau}$,
where $\widetilde{\tau}(x_1, x_2, x_3, y_1, y_2) = (-x_1, x_3, x_2, y_2, y_1)$.
In particular, we have the invariant part:
\[ H^2(X,\BZ)^G \cong E_8(-2) \oplus U(2) \]
and the anti-invariant part:
\[ H^2(X,\BZ)^{G,-} = \{ \alpha \in H^2(X,\BZ) |  \tau^{\ast} \alpha = -\alpha \} \cong U \oplus U(2) \oplus E_8(-2). \]
Since any invariant class is Hodge and curves on K3 surfaces vary in a linear system, any invariant class in $H^2(X,\BZ)$ descends to the Enriques surface, so that one has the 
isomorphism 
\[ \pi^{\ast} : H^2(Y,\BZ) \xrightarrow{\cong} H^2(X,\BZ)^G. \]
It satisfies $\pi^{\ast}(a) \cdot \pi^{\ast}(b) = 2 a \cdot b$.
In particular, $H^2(Y,\BZ) \cong E_8(-1) \oplus U$.

\subsection{Monodromy group}
\label{subsec:monodromy group}
Let $Y_1, Y_2$ be Enriques surfaces.
An isomorphism $f : H^{\ast}(Y_1,\BZ) \to H^{\ast}(Y_2,\BZ)$ is a {\em parallel transport operator},
if there exists a smooth projective morphism $\epsilon : \CY \to B$ over a smooth simply-connected curve, points $b_1, b_2 \in B$ and isomorphisms 
$\varphi_{i} : Y_i \to Y_{b_i}$ such that $f$ is the composition
\[ f : H^{\ast}(Y_1,\BZ) \xrightarrow{\varphi_{1 \ast}}
H^{\ast}(Y_{b_1},\BZ) \xrightarrow{} H^{\ast}(Y_{b_2},\BZ) \xrightarrow{\varphi_2^{\ast}} H^{\ast}(Y_2,\BZ), \]
where the middle arrow is the parallel transport obtained from trivializing $R \epsilon_{\ast} \BZ$. When $Y_1 = Y_2 =: Y$, then $f$ is called a monodromy operator. Let $\Mon(Y) \subset \mathrm{GL}(H^{\ast}(Y,\BZ))$ be the subgroup generated by all monodromy operators.

Let $X_i \to Y_i$ be the K3 cover for $i=1,2$.
A {\em lifted parallel transport operator} (of a parallel transport operator of Enriques surfaces) is the composition
\[ \widetilde{f} : H^{\ast}(X_1,\BZ) \xrightarrow{\widetilde{\varphi}_{1 \ast}}
H^{\ast}(X_{b_1},\BZ) \xrightarrow{} H^{\ast}(X_{b_2},\BZ) \xrightarrow{\widetilde{\varphi}_2^{\ast}} H^{\ast}(X_2,\BZ), \]
where $\CX \to \CY$ is the double cover defined by the relative canonical bundle of $\epsilon$, giving rise to a family of K3 surfaces $\CX \to B$,
the isomorphisms
$\widetilde{\varphi}_i : X_i \to \CX_{b_i}$ are lifts of $\varphi_i$ (which always exist, since K3 surfaces are simply connected),
and the middle arrow is the parallel transport operator of the family $\CX \to B$ (which is automatically $G$-equivariant, since it arises as an double cover).
Each lifted parallel transport operator $\widetilde{f}$ is $G$-equivariant, and its restriction to the invariant part recovers the parallel transport operator $f$ under the isomorphism $H^2(X,\BZ)^G \cong H^2(Y,\BZ)$.
In case $Y_1 = Y_2$, so $X_1 = X_2 = X$, we call $\widetilde{f}$ a lifted monodromy operator.
Let $\mathrm{GL}(H^{\ast}(X,\BZ))_G$ denote the group of automorphisms $f : H^{\ast}(X,\BZ) \to H^{\ast}(X,\BZ)$ such that $f \circ g^{\ast} = g^{\ast} \circ f$ for all $g \in G$.
The lifts $\widetilde{\varphi}_i$ are unique only up to composing with the covering involution; hence we define the lifted monodromy
$\widetilde{\Mon}(Y)$ as the subgroup of $\mathrm{GL}(H^{\ast}(X,\BZ))_G/G$ generated by all lifted parallel transport operators.

Since (lifted) parallel transport operators are degree-preserving with fixed action on $H^0$ and $H^4$, and K3 and Enriques surfaces have no odd cohomology,
we often identify them with their restriction to $H^2$, which are isometries. Hence we naturally write:
\[ \Mon(Y) \subset O(H^2(Y,\BZ)), \quad \widetilde{\Mon}(Y) \subset O(H^2(X,\BZ))_G/G. \]
We have a natural commutative diagram
\begin{equation} \label{commutative diag}
\begin{tikzcd}
\widetilde{\Mon}(Y) \ar[hook]{r} \ar{d}{r_{\Mon}} & O(H^2(X,\BZ))_{G}/G \ar{d}{r} \\
\Mon(Y) \ar[hook]{r} & O(H^2(Y,\BZ))
\end{tikzcd}
\end{equation}
where $r(\psi) = (\pi^{\ast})^{-1} \circ \psi|_{H^2(X,\BZ)^G} \circ \pi^{\ast}$.
By construction $r_{\Mon}$ is surjective.
Since the invariant lattice $H^2(X,\BZ)^G$ is $2$-elementary,
also $r$ is surjective by a criterion of Nikulin \cite{Nikulin, Namikawa}.

\begin{defn}
Let $L$ be a lattice of signature $(m,n)$, $m>0$. The unit sphere in any positive-definite $m$-dimensional subspace of $L \otimes \BR$ is a deformation retract of the positive cone $\{ x \in L_{\BR} : (x,x)>0 \}$.
The top cohomology of the sphere is hence a $1$-dimensional representation of $O(L)$, corresponding to a character $\nu : O(L) \to \{ \pm 1 \}$. We write 
$O^{+}(L) = \mathrm{Ker}(\nu)$ for the subgroup of $O(L)$ of orientaton-preserving isometries.
\end{defn}

\begin{example}
Since $H^2(Y,\BZ)$ is of signature $(1,9)$, the cone $\{ x \in H^2(Y,\BR) | x \cdot x > 0 \}$ has two connected components. 
Then $O^{+}(H^2(Y,\BZ)) \subset O(H^2(Y,\BZ))$
is the subgroup of orthogonal transformations which preserves the components.
\end{example}

We will also denote by
\[ O^{+,+}(H^2(X,\BZ))_G \subset O(H^2(X,\BZ))_G \] 
the group of $G$-equivariant isometries of $H^2(X,\BZ)$ such that their restriction to {\em both the invariant and anti-invariant part} are orientation-preserving.

\begin{prop} \label{prop:Mon-tilde}
$\widetilde{\Mon}(Y) = O^{+,+}(H^2(X,\BZ))_G/G$.
\end{prop}
\begin{proof}
Every lifted monodromy operator is a monodromy of a K3 surface,
so it preserves the orientation on $H^2(X,\BZ)$.
Moreover, its restriction to the invariant part $H^2(X,\BZ)^G$
can be identified with the monodromy of the Enriques surface,
hence is also orientation-preserving. 
Since the orientation-character on $H^2(X,\BZ)$ is the product of the orientation characters of the antiinvariant and invariant parts, 
this shows that every lifted monodromy operator lies in $O^{+,+}(H^2(X,\BZ))_G$.

The converse direction is a consequence of the Torelli theorem for Enriques surfaces \cite{Namikawa, BHPV},
as we explain now.
A marked Enriques surfaces is a triple $(X,\tau,\varphi)$ where $X$ is a K3 surface with a fixed-point free involution $\tau$ and $\varphi : H^2(X,\BZ) \to U^3 \oplus E_8(-1)^2$ is a marking such that $\varphi \tau^{\ast} \varphi^{-1} = \widetilde{\tau}$.
By the Torelli theorem for K3 surfaces, the moduli space of marked Enriques surfaces $\CM_{\mathrm{Enriques}}$
is the subspace of the moduli space of marked K3 surfaces $(X,\varphi)$ such that
$\varphi^{-1} \widetilde{\tau} \varphi$ is a Hodge isometry which sends an ample class to an ample class.
Consider the orthogonal complement of the invariant lattice in
$U^3 \oplus E_8(-1)^{2}$:
\[ M := (U(2) \oplus E_8(-2))^{\perp} \cong U \oplus U(2) \oplus E_8(-2). \]
Define the period domain
\[ \CD_M = \{ x \in \p(M \otimes \BC) | x \cdot x = 0, x \cdot \overline{x} > 0 \} \]
and let $\CD_M^{\circ} \subset \CD_M$ be the complement
of the locus of periods orthogonal to a $(-2)$-class:
\[
 \CD_M^{\circ} := \CD_M \setminus \CH, \quad 
 \CH := \{ x \in \p(M \otimes \BC) | x \cdot d = 0 \text{ for some } d \in M, d \cdot d = -2. \}. \]
The period map on the moduli space of marked K3 surfaces,
then restricts to a surjection $\mathrm{per} : \CM_{\mathrm{Enriques}} \to \CD_M^{\circ}$.
The moduli space of marked K3 surfaces has two connected components
interchanged by $(X,\varphi) \mapsto (X,-\varphi)$ \cite{HuybrechtsK3}. This gives rise to two (not necessarily connected) components $\CM_{\mathrm{Enriques},i}$, $i \in \{ 1, 2 \}$ of $\CM_{\mathrm{Enriques}}$. 
By the Torelli theorem of Enriques surfaces the restriction
$\mathrm{per} : \CM_{\mathrm{Enriques},i} \to \CD_M^{\circ}$ remains surjective
and is generically 1-to-1.
Since $M$ is of signature $(2,\ast)$, $\CD_M$ and hence $\CD_M^{\circ}$
has two connected components. Since $\mathrm{per}$ is generically $1$-to-$1$, 
it follows that each $\CM_{\mathrm{Enriques},i}$ then decomposes as the disjoint union of two {\em connected} components. 
Let $\varphi$ be a marking of $(X,\tau)$ and let $x = \mathrm{per}(X,\tau, \varphi) = \varphi(H^{2,0}(X))$ be its period.
Let $g \in O^{+,+}(H^2(X,\BZ))_G$.
Then also $(X,\tau, \varphi g)$ is a marked Enriques surface.
Since the restriction of $g$ to the anti-invariant part is
orientation-preserving, the period $\mathrm{per}(X,\tau, \varphi g) = g'(x)$, where $g' = \varphi \circ g \circ \varphi^{-1}$,
lies in the same connected component as $x$.
Moreover, since $g$ is orientation-preserving on $H^2(X,\BZ)$,
$(X,\varphi)$ and $(X, \varphi g)$ lie in the same $\CM_{\mathrm{Enriques},i}$ for some $i$. 
We see that $(X,\tau, \varphi)$ and $(X,\tau,g \varphi)$ lie in the same connected component of  $\CM_{\mathrm{Enriques}}$
and hence can be connected by a path.
It follows that $\varphi^{-1} \circ \varphi g=g$ is a lifted parallel transport operator.
\end{proof}

Using the diagram \eqref{commutative diag} we obtain the following by restriction:
\begin{cor} \label{cor:Monodromy Enriques}
$\Mon(Y) = O^{+}(H^2(Y,\BZ))$.
\end{cor}

The invariant part of $H^2(X,\BZ)$ carries a canonical orientation determined by any invariant ample class, and the anti-invariant part has a canonical orientation
given by $\mathrm{Re}([\sigma]), \Im([\sigma])$
where $\sigma$ is the symplectic form on $X$.
If $X_i \to Y_i$ are covering K3 surfaces of Enriques surfaces for $i=1,2$,
a $G$-equivariant isometry $\varphi : H^2(X_1,\BZ) \to H^2(X_2,\BZ)$
is {\em orientation-preserving}, if its restrictions to the invariant and antiinvariant part preserve the natural orientations.
\begin{cor} \label{cor:lifted parallel transport}
Let $X_i \to Y_i$ be K3 surfaces covering an Enriques surface for $i=1,2$.
Any $G$-equivariant orientation-preserving morphism $\varphi : H^2(X_1, \BZ) \to H^2(X_2, \BZ)$ is a lifted parallel-transport operator.
\end{cor}
\begin{proof}
Clearly, any lifted parallel transport operator is orientation-preserving.
Conversely,
since the moduli space of Enriques surfaces is connected, there
exists a lifted parallel transport operator $\psi : H^2(X_1,\BZ) \to H^2(X_2, \BZ)$. The claim then follows by applying Corollary~\ref{cor:Monodromy Enriques} to $\psi^{-1} \circ \varphi$.
\end{proof}

\subsection{Derived monodromy group}
\label{subsec:derived monodromy group}
For any smooth projective variety $X$, let $D^b(X)$ be the bounded derived category of coherent sheaves on $X$.
Let $v(E) = \ch(E) \sqrt{\td_X} \in H^{\ast}(X,\BQ)$ denote the Mukai vector
of an element $E \in D^b(X)$.
For $\gamma, \gamma' \in H^{2 \ast}(X,\BQ)$ let
\[ (\gamma,\gamma') = -\int_X \gamma^{\vee} \gamma' \]
be the Mukai pairing, where the dualization morphism $(- )^{\vee}$ acts by multiplication by $(-1)^{i}$ on $H^{2i}(X,\BQ)$.
In paricular, for all $E, E' \in D^b(X)$ we have
\[ (v(E), v(E')) = -\chi(E,E'), \]
where $\chi(E,E') = \sum_{i} (-1)^i \dim \Ext^i(E,E')$.


Any derived equivalence $\Phi : D^b(X_1) \to D^b(X_2)$ 
is the Fourier-Mukai transform
along a kernel $\CE \in D^b(X \times Y)$.
Let $\mathrm{pr}_1,\mathrm{pr}_2$ be the projections of $X \times Y$ to its factors.
Then the induced transform on cohomology is:
\[ \Phi^H : H^{\ast}(X_1,\BQ) \to H^{\ast}(X_2, \BQ), \quad 
\gamma \mapsto \mathrm{pr}_{2 \ast}( \mathrm{pr}_1^{\ast}(\gamma) \cdot v(\CE) ) \]
We have $\Phi^H(v(E)) = v( \Phi(E))$ and $(\Phi_1 \circ \Phi_2)^{H} = \Phi_1^H \circ \Phi_2^H$.

Let $Y$ be an Enriques surfaces and let $Y_1, Y_2$ be smooth deformations of $Y$.
For any equivalence\footnote{In fact, by a result of Bridgeland-Macioca \cite{BM}, one has that $Y_1 \cong Y_2$ in this situation.} $\Phi : D^b(Y_1) \to D^b(Y_2)$, consider the composition
\begin{equation*} H^{\ast}(Y) \xrightarrow{} H^{\ast}(Y_1,\BQ) \xrightarrow{\Phi^H} H^{\ast}(Y_2, \BQ) \xrightarrow{} H^{\ast}(Y,\BQ), 
\end{equation*}
where the outher arrows are the parallel transport operators induced by the deformations. The {\em derived monodromy group} of $Y$ is the group $\DMon(Y) \subset O( H^{\ast}(Y,\BQ))$ generated by all these operators.
 Clearly, we have $\Mon(Y) \subset \DMon(Y)$.

Any auto-equivalence $\Phi : D^b(Y_1) \to D^b(Y_2)$ between Enriques surfaces lifts to an auto-equivalence $\widetilde{\Phi} : D^b(X_1) \to D^b(X_2)$ of the covering K3 surfaces,
which is equivariant with respect to, and unique up to, the covering involutions,
see \cite[Sec.3.3]{Ploog} and \cite[Prop.3.5]{MMS}.
We let $\widetilde{\DMon}(Y) \subset O(H^{\ast}(X,\BZ))$ then be the subgroup generated by all compositions
\begin{equation}
\label{composition}
 H^{\ast}(X,\BZ) \xrightarrow{} H^{\ast}(X_1,\BZ) \xrightarrow{\widetilde{\Phi}^H} H^{\ast}(X_2, \BZ) \xrightarrow{} H^{\ast}(X,\BZ),
 \end{equation}
where the outer arrows are the lifted parallel transport operators.
Recall here that for K3 surfaces equivalences act on the integral cohomology, see \cite{HuybrechtsK3}.

Let $O(H^{\ast}(X,\BZ))_G$ be the subgroup of $G$-equivariant isometries,
and let
\[ O^{+,+}(H^{\ast}(X,\BZ))_G \subset O(H^{\ast}(X,\BZ))_G \]
be the index $4$ subgroup of $G$-equivariant automorphisms $\varphi : H^{\ast}(X,\BZ) \to H^{\ast}(X,\BZ)$ such their restriction to both the invariant and anti-invariant part preserves the orientation.

\begin{prop} \label{prop:DMon tilde}
$\widetilde{\DMon}(Y) = O^{+,+}(H^{\ast}(X,\BZ))_G$.
\end{prop}
\begin{proof}
We first explain the direction '$\subset$'.
For any K3 surface $X$,
the lattice $H^{\ast}(X,\BZ)$ has a natural orientation determined by
$( \mathrm{Re}(\sigma), \mathrm{Im}(\sigma), 1-\omega^2, \omega )$ for any K\"ahler class $\omega$ and for any symplectic form $\sigma$.
Moreover, any equivalence between K3 surfaces preserves this orientation,
 see \cite[Sec.4.5]{HMS}.
 If $X \to Y$ covers an Enriques, the anti-invariant part
 has a natural orientation
 determined by $(\mathrm{Re}(\sigma), \mathrm{Im}(\sigma) )$.
Moreover, any equivalence acts as a Hodge isometry on cohomology,
so preserves this orientation.
Taken together, it follows that every equivalence also preserves the orientation
on the invariant given by $(1-\omega^2, \omega )$ for some $G$-invariant K\"ahler class $\omega$.
Further, any lifted parallel transport operator
is orientation preserving on the invariant and anti-invariant part of $H^{\ast}(X,\BZ)$, see Corollary~\ref{cor:lifted parallel transport}.
Hence we get that any element in $\widetilde{\DMon}(Y)$ is orientation-preserving.
Thus $\widetilde{\DMon}(Y) \subset O^{+,+}(H^{\ast}(X,\BZ))_G$.
%

Conversely,
let $\psi \in O^{+,+}(H^{\ast}(X,\BZ))_G$.
We will show $\psi \in \widetilde{\DMon}(Y)$
by following \cite[Proof of Prop.3.5]{MMS}.
Since the statement is independent of the choice of the Enriques surface,
we may assume that $Y$ is generic,
so that the Hodge classes in $H^{\ast}(X,\BZ)$
are precisely the invariant classes, that is:
\[ H^{1,1}(X,\BC) \cap H^{2}(X,\BZ) = H^{2}(X,\BZ)^{G}. \]

\vspace{3pt}
\noindent
\textbf{Case 1: $\psi(0,0,1) = \pm (0,0,1)$}

\vspace{2pt}
After composing $\psi$ with the shift functor $[1]^{H} = -\id_{H^{\ast}(X,\BZ)}$, we may assume $\psi(0,0,1) = (0,0,1)$.
Let $v=\psi(1,0,0)$ and consider the decomposition $v = (r,\ell,s)$ according to degree. Since $\psi$ is $G$-equivariant
and $(1,0,0)$ is $G$-invariant,
$v$ and hence $\ell$ is $G$-invariant.
Since all invariant classes on $H^2(X,\BZ)$ are Hodge,
there exists a line bundle $L \in \Pic(Y)$ with $c_1(L) = \ell$.
This shows that $v=\exp(c_1(L))$.
After composing $\psi$ with tensoring with $L^{\vee}$ we may hence assume 
$\psi(1,0,0) = (1,0,0)$.
Thus $\psi = \id_{H^0(X,\BZ) \oplus H^4(X,\BZ)} \oplus \psi_2$ for some $\psi_2 \in O^+(H^2(X,\BZ))_G$.
By Proposition~\ref{prop:Mon-tilde} $\psi_2$ is in the image of $\Mon(Y)$ so we are done.

\vspace{3pt} \noindent
\textbf{Case 2:} $\psi(0,0,1) = (r,\ell,s) =: v$ with $r \neq 0$

\vspace{2pt}
After composing with the shift, we may again assume that $r>0$.
As before we have that $v$ is $G$-invariant, and hence that it is Hodge.
Consider the moduli space $M = M_h(v)$ of stable sheaves on $X$ of Mukai vector $v$ with respect to a generic ample polarization $h$ (which by our assumption is automatically $G$-invariant).
Then $M$ is a K3 surface and there exists a universal family $\CE$ on $M \times X$ inducing an equivalence $\Phi_{\CE} : D^b(M) \to D^b(X)$.

Consider the composition
\[ \Psi = \Phi_{\CE}^{-1} \circ \tau^{\ast} \circ \Phi_{\CE} : D^b(M) \to D^b(M). \]
For any $x \in M$, we have $\Psi( k(x)) = \Phi_{\CE}^{-1}( \tau^{\ast} \CE_x )$,
and since $\tau^{\ast}(\CE_x)$ is again stable of Mukai vector $\tau^{\ast}(v) =v$ we have that $\tau^{\ast}(\CE_x) = \CE_{x'}$ for some $x' \in M$ and hence
$\Psi(k(x)) = k(x')$. Since $\Psi$ takes sykscraper sheaves to skyscraper sheaves,
we have
$\Psi = (L \otimes (-)) \circ \widetilde{\tau}^{\ast}$ for some $L \in \Pic(M)$ and automorphism $\widetilde{\tau}$.
Since $\tau$ is an involution, the same holds for $\Psi$, and hence $\widetilde{\tau}^2 = \id$ and $\widetilde{\tau}^{\ast}(L) = L^{\vee}$.
Moreover, let $T(X)$ be the transcendental lattice. Then by our assumption that $Y$ is generic, we have that $T(X)$ is rank $12$.
The equivalence induces an isomorphism $\Phi_{\CE}^{H}: T(M) \xrightarrow{\cong} T(X)$,
so $T(M)$ is of rank $12$ and $\Psi^H|_{T(M)} = -\id_{T(M)}$.
This shows that $\Psi^H : H^{\ast}(M,\BQ) \to H^{\ast}(M,\BQ)$ acts as the identity on all algebraic classes. In particular, $\Psi^H(c_1(L)) = 0$ shows that $L=\CO$, so $\Psi = \widetilde{\tau}^{\ast}$.
Moreover, since the antiinvariant lattice of the action of $\langle \widetilde{\tau} \rangle$ on cohomology is the same as that of $G$, namely $U(2) \oplus E_8(-2)$,
by \cite[Thm.0.1]{AST} $\widetilde{\tau}$ is fixed point free.

Consider now the composition 
\[ \widetilde{\psi} := \psi^{-1} \circ \Phi_{\CE}^H : H^{\ast}(M,\BZ) \to H^{\ast}(X,\BZ). \]
By construction we have $\widetilde{\psi}(0,0,1) = (0,0,1)$.
Thus we get $\widetilde{\psi}(1,0,0) =: \widetilde{v} =  (1,\widetilde{\ell}, \widetilde{s})$.
As in Case 1 we have that $\widetilde{\ell}$ is $G$-invariant (since $\Psi^H(1,0,0) = \widetilde{\tau}(1,0,0) = (1,0,0)$ and $\widetilde{\psi}$ is $G$-equivariant), hence it is algebraic 
and $\widetilde{v} = \exp(c_1(\widetilde{L}))$ for some $\widetilde{L} \in \Pic(X)$.
After replacing $\psi$ by $\psi \circ (\widetilde{L} \otimes (- ))^H$ we hence may assume that $\widetilde{\psi}(1,0,0) = (1,0,0)$, so that
\[ \widetilde{\psi}(a,\beta,b) = (a, \widetilde{\psi}_2(\beta), b), \]
for some isometry $\psi_2 : H^2(M,\BZ) \to H^2(X,\BZ)$
which is $G$-equivariant, i.e. $\psi_2 \circ \widetilde{\tau}^{\ast} = \tau^{\ast} \psi_2$.
Since $\Phi_{\CE}^H$ is orientation-preserving (as explained when proving the direction '$\subset$'), and $\psi^{-1}$ is orientation preserving by assumption,
we find that $\psi_2$ is orientation-preserving (both on the invariant and anti-invariant parts).
%
Hence by Corollary~\ref{cor:lifted parallel transport} $\psi_2$ is a lifted parallel transport operator of an Enriques surface.
In conclusion, we get 
$\psi = \Phi_{\CE}^H \circ \widetilde{\psi}^{-1}$ lies in $\widetilde{\DMon}(Y)$.

\vspace{3pt} \noindent
\textbf{Case 3:} $\psi(0,0,1) = (0,\ell,s)$: Reduce to case 2 as in \cite[Case 3]{MMS}
by composing with the spherical twist along $\CO_X$.
\end{proof}

We get a description of the derived monodromy group of the Enriques surface $Y$ as follows:
Let $\Lambda_Y \subset H^{\ast}(Y,\BQ)$ be the lattice generated by the Mukai vectors $v(E)$ for all $E \in D^b(Y)$.
Consider the pullback to the K3 cover which is given by:\footnote{On the other hand, $\pi_{\ast}(r, \beta', n) = (2r, \pi_{\ast} \beta', n)$.}
\[ \pi^{\ast} : H^{\ast}(Y,\BZ) \to H^{\ast}(X, \BZ), \quad \pi^{\ast}(r,\beta,n) = (r, \pi^{\ast} \beta, 2n). \]
By a direct check, with respect to the Mukai pairing we have the isomorphism of lattices
\[ \Lambda_Y(2) \cong \pi^{\ast}(\Lambda_Y) = \{ (r,\beta, n) \in H^{\ast}(X,\BZ) | r + n \text{ even}, \beta \in H^2(X,\BZ)^G \} \]

The invariant lattice is the index $2$ overlattice:
\[ H^{\ast}(X,\BZ)^{G} =
\{ (r,\beta, n) \in H^{\ast}(X,\BZ) | \beta \in H^2(X,\BZ)^G \}
\cong U \oplus U(2) \oplus E_8(-2). \]
Nevertheless, by Lemma~\ref{lemma:isometry preserving sublattice} below any isometry $\varphi \in O(H^{\ast}(X,\BZ)^G)$ preserves the sublattice $\pi^{\ast}(\Lambda_Y)$, so we get an inclusion
\[ O(H^{\ast}(X,\BZ)^G) \subset O(\Lambda_Y). \]
We thus have the commutative diagram:
\[
\begin{tikzcd}
\widetilde{\DMon}(Y) \ar[hook]{r} \ar{d} & O(H^{\ast}(X,\BZ))_G \ar{d}{r} \\
\DMon(Y) \ar[hook]{r} & O(H^{\ast}(X,\BZ)^G) \ar[hook]{r} & O(\Lambda_Y),
\end{tikzcd}
\]
where $r$ is the restriction to the invariant part.

\begin{cor} \label{cor:Dmon}
$\DMon(Y) = O^{+}(H^{\ast}(X,\BZ)^G)$
\end{cor}
\begin{proof}
By construction, $\widetilde{\DMon}(Y) \to \DMon(Y)$ is surjective.
Moreover, since $H^{\ast}(X,\BZ)^G$ is $2$-elementary,
by a criterion of Nikulin \cite{Nikulin, Namikawa}
the restriction map $r : O^{+}(H^{\ast}(X,\BZ))_G \to O^{+}(H^{\ast}(X,\BZ)^G)$ is surjective. So the claim follows from Proposition~\ref{prop:DMon tilde}.
\end{proof}

\begin{rmk} \label{rmk:distinguished component}
In the proof above we only used derived auto-equivalences $\widetilde{\Phi} : D^b(X_1) \to D^b(X_2)$ with the property:
\begin{enumerate}
\item[(i)] $X_i$ are covering K3 surfaces of generic Enriques surfaces,
\item[(ii)] $\widetilde{\Phi}$ maps $\Stab^{\circ}(X_1)$ to $\Stab^{\circ}(X_2)$
(see \cite{HuybrechtsFM}),
\end{enumerate}
where we let $\Stab^{\circ}(X)$ denote the distinguished component of the space of Bridgeland stability conditions of a K3 surface $X$ constructed by Bridgeland \cite{Bridgeland}.
Hence the group $\widetilde{\DMon}(Y)$ is generated by compositions \eqref{composition}, where $\widetilde{\Phi}$ is an auto-equivalence satisfying (i) and (ii).
\end{rmk}

\subsection{Orbit of vectors}
\label{subsec:orbit of vectors}
Consider the lattice
\[ M = U \oplus U(2) \oplus E_8(-2). \]
A primitive vector $v \in M$ is called of even type if $v \cdot w$ is even for all $w \in M$,
and it is odd otherwise.
If we let $e_1, f_1$ and $e_2, f_2$ be a symplectic basis of $U$ and $U(2)$ respectively (so $e_i \cdot e_i=0, f_i \cdot f_i = 0$, and $e_1 \cdot f_1 = 1$ and $e_2 \cdot f_2 = 2$), and write an element $v \in M$ as
\[ v = a_1 e_1 + b_1 f_1 + a_2 e_2 + b_2 f_2 + \alpha, \quad \alpha \in E_8(-2), \]
then $v$ is even if both $a_1, b_1$ are even, and it is odd otherwise.
Clearly, the type of a primitive vector is preserved under any isometry of the lattice.

\begin{lemma}
\label{lemma:isometry preserving sublattice}
Define the sublattice $L = \{  a_1 e_1 + b_1 f_1 + a_2 e_2 + b_2 f_2 + \alpha | a_1 + b_1 \text{ even} \} \subset M$ . Any $\varphi \in O(M)$ satisfies $\varphi(L) = L$.
\end{lemma}
\begin{proof}
Let $v=a_1 e_1 + b_1 f_1 + a_2 e_2 + b_2 f_2 + \alpha \in L$.
If $a_1, b_1$ are both even, then $v$ is of even type, hence $gv$ is even,
hence $gv \in L$.
If $v$ is odd, then $v^2 \equiv 2$ modulo $4$ if and only if $a_1,b_1$ are both odd, that is if and only if $v \in L$.
Hence if $v$ is odd and in $L$, so is $gv$.
\end{proof}

\begin{prop} \label{prop:orbit in M}
Any two primitive vectors in $M$ of the same norm and the same type
lie in the same $O(M)$ orbit.
\end{prop}

Proposition~\ref{prop:orbit in M} was proven for primitive vectors of norm $-2$ and $-4$ in \cite[Theorems 2.13 and 2.15]{Namikawa} using Nikulin's criterion.
One can check that the same proof also implies the general case.
Alternatively, the general case was proven in \cite{KS} using an argument of Allcock.

\begin{cor} \label{cor:orbit of primitive vector in M}
Any two primitive vectors in $M$ of the same norm and the same type
lie in the same $O^{+}(M)$ orbit.
\end{cor}
\begin{proof}
By the previous proposition a primitive vector of odd type (resp. even type)
lies in the $O(M)$ orbit of $e_1 + m f_1$ (resp. of $e_2 + m f_2$) for some $m$.
There exists orientation-reserving isometries fixing these vectors, namely
$\id_U \oplus (- \id_{U(2)}) \oplus \id_{E_8(-2)}$ for $e_1 + m f_1$,
and $-\id_{U} \oplus \id_{U(2) \oplus E_8(-2)}$ for $e_2 + m f_2$,
so the $O^{+}(M)$ and $O(M)$ orbits of these vectors agree.
\end{proof}

\section{Gromov-Witten theory of the Enriques surface}
In this section we study the relative generating series 
of descendent Gromov-Witten invariants
of an elliptic Enriques surface $Y \to \p^1$.
We prove that these series are $\Gamma_0(2)$ quasi-Jacobi forms for the lattice $E_8$ and satisfy a holomorphic anomaly equation (Theorem~\ref{thm:HAE}). Examples are discussed in Section~\ref{subsec:Examples GW section}.
The proof takes place in Section~\ref{subsec:proof of thm HAE enriques}.

\label{sec:GW theory}
\subsection{Elliptic fibration}
Consider an elliptic fibration on a generic Enriques surface
\[ \pi : Y \to \p^1. \]
The fibration $\pi : Y \to \p^1$ has $12$ rational nodal fibers and two double fibers $2f_1, 2 f_2$,
where the {\em half-fibers} $f_1, f_2$ are smooth rigid elliptic curves on $Y$.
The generic fiber has class $2 f_1 = 2 f_2$ and the canonical bundle is $\omega_Y = \CO_Y(f_1 - f_2)$.
We will denote the image of $f_1$ (or equivalently $f_2$) modulo torsion by
\[ f \in H_2(Y,\BZ). \]
Since $Y$ is generic, by \cite[Thm.17.7]{BHPV} we may assume that $\pi$ has a $2$-section $s$ which is represented by a smooth rigid elliptic curve,
which appears as a half-fiber of another elliptic fibration of $Y$.
We have
\[ s \cdot s = 0, \quad s \cdot f = 1, \quad f \cdot f = 0. \]
We will use a fixed identification
\begin{equation} H^2(Y,\BZ) \cong U \oplus E_8(-1) \label{UE8 decomposition} \end{equation}
where $s,f \in H^2(Y,\BZ)$ are identified with the canonical basis $e,f \in U$.

\subsection{Gromov-Witten invariants}
The moduli space $\Mbar_{g,n}(Y,\beta)$
of degree $\beta \in H_2(Y,\BZ)$ stable maps to $Y$ from connected $n$-marked genus $g$ curves has a virtual fundamental class
of dimension $g-1+n$. Let $\psi_i \in H^2(\Mbar_{g,n}(Y,\beta))$ be the cotangent line classes.
If $2g-2+n>0$, 
let 
$\tau : \Mbar_{g,n}(Y,\beta) \to \Mbar_{g,n}$
be the forgetful morphism
to the moduli space of stable curves. Consider the pullback of a {\em tautological class} \cite{FP}
\[ \taut := \tau^{\ast}(\alpha), \quad \alpha \in R^{\ast}(\Mbar_{g,n}). \] 
In the unstable cases $2g-2+n \leq 0$ we always set $\taut := 1$.

For $\gamma_1, \ldots, \gamma_n \in H^{\ast}(Y)$ and $k_1, \ldots, k_n \geq 0$
define the Gromov-Witten invariants of $Y$ by
\begin{equation} \label{defn GW invt}
\blangle \taut ; \tau_{k_1}(\gamma_1), \, \ldots, \, \tau_{k_n}(\gamma_n) \brangle^{Y}_{g,\beta}
=
\int_{[ \Mbar_{g,n}(Y,\beta) ]^{\vir}}
\taut \cup
\prod_{i=1}^{n} \ev_i^{\ast}(\gamma_i) \psi_i^{k_i}.
\end{equation}

\begin{rmk} \label{rmk:GW stuff}
(a) If $\taut = 1$, we often omit it from the notation in \eqref{defn GW invt}.
Similarly, if $k_1 = \ldots = k_n = 0$, we omit the symbols $\tau_{k_i}$ from the notation. \\
(b) If $\beta \neq 0$ is not effective, then the Gromov-Witten invariant vanishes by definition.
If $\beta = ks + df + \alpha$ is effective, then since $s,f$ are half-fibers of elliptic
fibrations, we mast have $k, d \geq 0$ and $(k,d) \neq (0,0)$. \\
(c) In genus zero, the virtual class of $\Mbar_{0,n}(Y,\beta)$ vanishes for dimension reasons.\footnote{The virtual class of $\Mbar_{0,n}(Y,\beta)$ is the pullback of the virtual class of $\Mbar_0(Y,\beta)$ and hence vanishes.}
Hence for $2g-2+n>0$, the classes $\psi_i$ and $\tau^{\ast}(\psi_i)$ on $\Mbar_{g,n}(Y,\beta)$ 
differ by a term that vanishes after intersecting with the virtual class. We find that:
\[ 
\blangle \taut ; \tau_{k_1}(\gamma_1), \ldots, \tau_{k_n}(\gamma_n) \brangle^{Y}_{g,\beta}
=
\blangle \taut \prod_{i} \psi_i^{k_i} ; \gamma_1, \ldots, \gamma_n \brangle^{Y}_{g,\beta}.
\]
\end{rmk}

\subsection{Generating series}
Let
$\alpha_1, \ldots, \alpha_8 \in H^2(Y,\BZ)$ be an integral basis of the summand $E_8(-1)$ in \eqref{UE8 decomposition}.
We will identify the element $x = (x_1, \ldots, x_8) \in \BC^8$ with the element
$\sum_i x_i \alpha_i \in E_8(-1) \otimes \BC \subset H^2(Y,\BC)$.
Let $\zeta_i = e^{2 \pi i x_i}$.
Then for any $\beta \in H_2(Y,\BZ)$ we will write:
\begin{equation} \zeta^{\beta} := \exp( 2 \pi i x \cdot \beta ) = \prod_{i=1}^{8} \zeta_i^{\alpha_i \cdot \beta}. \label{zeta beta} \end{equation}

\begin{defn} \label{defn:relative gen series} Assume either $k > 0$, or $k=0$ and $2g-2+n>0$.
We define the $\pi$-relative generating series of Gromov-Witten invariants of $Y$ by
\begin{align*} 
F_{g,k}( \taut ; \tau_{m_1}(\gamma_1) \cdots \tau_{m_n}(\gamma_n) )
& =
\sum_{\substack{ \beta \in H_2(Y,\BZ) \\ \beta \cdot f = k }}
q^{s \cdot \beta} \zeta^{\beta} \left\langle \taut ; \tau_{m_1}(\gamma_1) \cdots \tau_{m_n}(\gamma_n) \right\rangle^{Y}_{g, \beta} \\
& = \sum_{d \geq 0} \sum_{\alpha \in E_8(-1)} 
q^{d} \zeta^{\alpha} \left\langle \taut ; \tau_{m_1}(\gamma_1) \cdots \tau_{m_n}(\gamma_n) \right\rangle^{Y}_{g, k s + df + \alpha}.
\end{align*}
where for the last equality we used Remark~\ref{rmk:GW stuff}(b).
\end{defn}

Consider the negative of the pairing matrix of the basis $\alpha_i$,
\[ Q_{E_8} := \left( -\alpha_i \cdot \alpha_j \right)_{i,j=1}^{8}, \]
which is a intersection matrix of the $E_8$-lattice.

Our main result about the series $F_{g,k}$ is the following:
\begin{thm} \label{thm:HAE}
Each series $F_{g,k}( \taut; \tau_{m_1}(\gamma_1) \cdots \tau_{m_n}(\gamma_n) )$ is a quasi-Jacobi form
for $\Gamma_0(2) \ltimes (\BZ^{8} \oplus \BZ^8)$ of index $\frac{1}{2} k Q_{E_8}$, with pole of order $\leq k$ at cusps. More precisely, we have
\[ F_{g,k}(  \taut; \tau_{m_1}(\gamma_1) \cdots \tau_{m_n}(\gamma_n) ) \in 
\left(\frac{ \eta^{8}(q^2) }{\eta^{16}(q)} \right)^k
\QJac_{\frac{1}{2} k Q_{E_8}}(\Gamma_0(2)). \]
It satisfies the holomorphic anomaly equation:
\begin{equation} \label{HAE}
\begin{aligned}
& \frac{d}{dG_2} F_{g,k}( \taut ; \tau_{m_1}(\gamma_1) \cdots \tau_{m_n}(\gamma_n) ) \\
& =
\sum_i F_{g-1,k}( \taut' ;  \tau_{m_1}(\gamma_1) \cdots \tau_{m_n}(\gamma_n) \tau_0(\delta_i) \tau_0(\delta_i^{\vee}) ) \\
& + \sum_{\substack{ g=g_1+g_2 \\ k = k_1 + k_2 \\ \{ 1 , \ldots, n \} = A \sqcup B \\ i }}
F_{g_1,k_1}\left( \taut_1 ; \prod_{i\in A} \tau_{m_i}(\gamma_i) \tau_0(\delta_i) \right) 
F_{g_2,k_2}\left( \taut_2 ; \prod_{i\in B} \tau_{m_i}(\gamma_i) \tau_0(\delta_i^{\vee}) \right) \\
& -2 \sum_{i=1}^{n} F_{g,k}( \taut ; \tau_{m_1}(\gamma_1) \cdots \tau_{m_{i-1}}(\gamma_{i-1}) \tau_{m_i+1}( U(\gamma_i)) \tau_{m_{i+1}}(\gamma_{i+1}) \cdots \tau_{m_n}(\gamma_n) ).
\end{aligned}
\end{equation}
where 
\begin{itemize}
\item $\sum_i \delta_i \boxtimes \delta_i^{\vee}$ is a K\"unneth decomposition of the class
\[ U = \frac{1}{2} (\pi \times \pi)^{\ast} \Delta_{\p^1} = \pr_1^{\ast}(f) + \pr_2^{\ast}(f) \in H^{\ast}(Y \times Y) \]
where $\Delta_{\p^1} \in H^{2}(\p^1 \times \p^1)$
is the diagonal class
and $\pr_i : Y^2 \to Y$ are the projections,
\item in the last line we let $U$ act as a correspondence, i.e. $U(\gamma) = \pr_{2 \ast}( \pr_1^{\ast}(\gamma) \cup U )$,
\item in the stable case, where $\taut = \tau^{\ast}(\alpha)$, we let $\taut':=\tau^{\ast} \iota^{\ast}(\alpha)$
where $\iota : \Mbar_{g-1,n+2} \to \Mbar_{g,n}$ is the gluing map, in the unstable case, where $\taut=1$, we set $\taut':=1$,
\item where $\taut_1, \taut_2$ stands for summing over the K\"unneth decomposition
of $\xi^{\ast}(\taut)$ with $\xi$ the gluing map
\[ \xi : \Mbar_{g_1,|A|+1}(Y,\beta_1) \times_{Y} \Mbar_{g_2,|B|+1}(Y,\beta_2) \to \Mbar_{g,n}(Y,\beta). \]
\end{itemize}
\end{thm}

Before turning to the proof, we note some basic consequences
of the theorem.
Consider the semisimple weight operator
\[ \Wt = [s \cup - , U] : H^{\ast}(Y) \to H^{\ast}(Y) \]
and let $\wt(\gamma) \in \{ -1, 0, 1 \}$ be the eigenvalue of an eigenvector.
Concretely, we have
\[
\wt(\gamma) = 
\begin{cases}
1 & \text{ if } \gamma \in \{ s, \pt \} \\
0 & \text{ if } \gamma \in E_8(-1) \\
-1 & \text{ if } \gamma \in \{ 1, f \}.
\end{cases}
\]

\begin{cor} \label{cor:weight} If all $\gamma_i$ are $\wt$-homogeneous, then the series $F_{g,k}( \taut; \tau_{m_1}(\gamma_1) \cdots \tau_{m_n}(\gamma_n) )$ is of quasi-Jacobi form of weight
$2g-2+n + \sum_i \wt(\gamma_i)$ and index $\frac{1}{2} k Q_{E_8}$.
\end{cor}
\begin{proof}
This can be proven by the same argument as in \cite[Sec.3.1]{RES}.
\end{proof}

For any $\lambda \in E_8(-1) \subset H^2(Y,\BZ)$ consider the operator:\footnote{Under the well-known identification $\mathfrak{so}(H^2(Y,\BQ)) = \wedge^2 H^2(Y,\BQ)$, we have $t_{\lambda} = \lambda \wedge f$.}
\[ t_{\lambda} : H^{\ast}(Y,\BZ) \to H^{\ast}(Y, \BZ), \quad t_{\lambda}(x) = (f \cdot x) \lambda - (\lambda \cdot x) f. \]

\begin{cor} \label{cor:elliptic anomaly}
For any $\lambda \in E_8(-1) \subset H^2(Y,\BZ)$ we have the elliptic anomaly equation:
\begin{equation} \label{elliptic anomaly} \xi_{\lambda} F_{g,k}( \tau_{m_1}(\gamma_1) \cdots \tau_{m_n}(\gamma_n) )
= \sum_{i=1}^{n} F_{g,k}(  \cdots \tau_{m_{i-1}}(\gamma_{i-1}) \tau_{m_i}( t_{\lambda}(\gamma_i) ) \tau_{m_{i+1}}( \gamma_{i+1} )  \cdots ).
\end{equation}
\end{cor}
\begin{proof}
This is proved as in \cite[Sec.3.3]{RES}.
\end{proof}

\subsection{Examples of Theorem~\ref{thm:HAE}}
\label{subsec:Examples GW section}
\subsubsection{Genus 0}
Consider the genus 0 series
\[ F_{0,0}( \tau_0(\gamma_1) \tau_0(\gamma_2) \tau_0(\gamma_3)) = \int_{Y} \gamma_1 \gamma_2 \gamma_3. \]
A direct check shows that the right hand side vanishes unless $\sum_i \wt(\gamma_i) = -1$,
which is in agreement with Theorem~\ref{thm:HAE} and Corollary~\ref{cor:weight}.

\subsubsection{Degree 0}
The virtual class in degree $\beta=0$ is as follows:
\[
[ \Mbar_{g,n}(Y,0) ]^{\text{vir}} = 
\begin{cases}
[\Mbar_{0,n} \times Y] & \text{ if } g=0, n \geq 3 \\
[\Mbar_{1,n} \times Y] c_2(Y) & \text{ if } g=1, n \geq 1 \\
0 & \text{ if } g \geq 2.
\end{cases}
\]
Using the string equation for $d>0$ and the constant evaluation we hence obtain:
\[ F_{1,0}( \tau_1(1)) = \frac{1}{24} \int_Yc_2(Y) = \frac{1}{2} \]
which is as predicted of weight $0$.

\subsubsection{Fiber classes} \label{subsubsec:fiber classes}
Maulik and Pandharipande \cite{MP} proved that for $d>0$ we have
\begin{gather*}
N_{1,df} = 2 \sigma_{-1}(d) - \sigma_{-1}(d/2), \quad \quad 
N_{g,df} = 0, \quad g > 1,
\end{gather*}
where
\[ \sigma_{-1}(d) =
\begin{cases}
\sum_{k|d} \frac{1}{k} & \text{ if } d \in \BZ_{> 0} \\
0 & \text{ otherwise }.
\end{cases}
\]
For the series $F_{1,0}(\tau_0(s))$ which is of weight $2$ according to Corollary~\ref{cor:weight} we obtain:
\begin{align*}
F_{1,0}(\tau_0(s))
& = 2 G_2(q) - 2 G_2(q^2) 
= G_2(q) + \frac{1}{24} F_2(q)
\end{align*}
where the modular form $F_2(q)$ was defined in \eqref{F2}.
The holomorphic anomaly equation of Theorem~\ref{thm:HAE} gives correctly:
\[
\frac{d}{dG_2} F_{1,0}(\tau_0(s))
=
2 F_{0,0}(\tau_0(s) \tau_0(f) \tau_0(1))
-2 F_{1,0}( \tau_1(1)) = 2 - 1 = 1.
\]

\subsubsection{Hodge integrals} \label{subsubsec:Hodge integrals}
For $g \geq 2$ or $k>0$ define the series
\[ F_{g,k} := F_{g,k}( (-1)^{g-1} \lambda_{g-1} ). \]
By Corollary~\ref{cor:weight} $F_{g,k}$ is a quasi-Jacobi form of weight $2g-2$ and index $\frac{1}{2} k Q_{E_8}$.

By the splitting formulas for the Chern classes $c(\BE)$ proved in \cite[Proof of Prop.2]{FPHodge}, the holomorphic anomaly equation (Theorem~\ref{thm:HAE}) says
\begin{align*}
\frac{d}{dG_2} F_{g,k} = 2 F_{g-1,k}( c(\BE^{\vee}) ; \tau_0(1) \tau_0(f) )
+ 2 \sum_{g_1,g_2, k_1, k_2}
F_{g_1,k_1}(c(\BE^{\vee}) ; \tau_0(1) ) F_{g_2,k_2}( c(\BE^{\vee}) ; \tau_0(f) ).
\end{align*}
Because of the $\tau_0(1)$ insertion, the first term on the right vanishes.
Similarly, the only term that contributes in the second term is $g_1=1$ and $k_1=0$.
We have
\[ F_{1,0}( c(\BE^{\vee}) ; \tau_0(1) ) = \int_{[ \Mbar_{1,1}(Y,0) ]^{\vir}} (-1) \lambda_1
= \int_{\Mbar_{1,1}} -\lambda_1 \int_Y c_2(Y) = -\frac{1}{2}. \]
Combined with the elliptic anomaly equation we conclude 
\[ \frac{d}{dG_2} F_{g,k} = -k F_{g-1,k},
\quad
\xi_{\lambda} F_{g,k} = 0 \text{ for all } \lambda \in E_8(-1). \]


\subsection{Proof of Theorem~\ref{thm:HAE}}
\label{subsec:proof of thm HAE enriques}
The elliptic Enriques $Y \to \p^1$ admits a degeneration
\[ Y \rightsquigarrow R \cup_E X_1, \]
where $R$ is the rational elliptic surface and $X_1$ is an elliptic surface with two double fibers, glued along a common smooth fiber $E$ \cite{MP}. The degeneration respects the elliptic fibrations and Theorem~\ref{thm:HAE} follows by applying the degeneration formula \cite{Li1,Li2}. Concretely, below we first introduce the elliptic surfaces $R$ and $X_1$, and state their holomorphic anomaly equations in the relative case $(R,E)$ and $(X_1,E)$ following \cite{RES}. In Section~\ref{subsec:Degeneration} we then construct the degeneration
and determine how the cohomology classes of $Y$ specialize to the central fiber.
In Section~\ref{subsubsec:applying degeneration} we apply the degeneration formula.
There is a small extra step: the degeneration formula does not give us all of the modular behaviour we want. The remaining bit is proven by using the monodromy in Section~\ref{subsubsec:monodromy}.

\subsubsection{Rational elliptic surface}
Let $R$ be a rational elliptic surface, which is obtained by blowing up $\p^2$ along the base locus of a (sufficiently generic) pencil of cubic curves.
Let $\pi : R \to \p^1$ be the elliptic fibration induced from the pencil, let $B \subset R$ be a fixed section, let $F \in \Pic(R)$ be the class of a fiber and set
\[ W := B + \frac{1}{2} F \in H^2(R,\BQ). \]
There is a natural orthogonal decomposition
\begin{equation} H^2(R, \BZ) \cong \mathrm{Span}_{\BZ}(B, F) \oplus E_8(-1). \label{R E8 decomposition} \end{equation}
We identify $E_8(-1)$ with its image in $H^2(R,\BZ)$ under this decomposition.
Following a parallel convention as in \eqref{zeta beta},
we pick a basis of $E_8(-1)$ and use the symbols $\zeta^{\beta}$, $\beta \in H_2(R,\BZ)$.

Let $E \subset R$ be a fixed smooth fiber of $\pi : R \to \p^1$ over a point $\mathrm{pt} \in \p^1$, and let
\[ \underline{\eta} = \big( (\eta_1, \delta_1), \ldots, (\eta_{l(\eta)}, \delta_{\ell(\eta)}) \big), \quad \text{ with } \quad \eta_i \geq 1, \ \  \delta_i \in H^{\ast}(E,\BQ) \]
be an ordered cohomology weighted partition underlying the partition
$\eta = ( \eta_i )$ of $k$. Let $\Mbar_{g,n}'(R/E,\beta;\eta)$ be the moduli space of relative stable maps $f : C \to R[\ell]$ from possibly disconnected genus $g$ curves in class $\beta$ to $R$ with ordered ramification profile $\eta$ over the relative divisor $E$, with the requirement that every connected component $C'$ of the domain $C$ satisfies that (i) $\pi \circ f|_{C'}$ is non-constant, or (ii) $C'$ has genus $g'$ and carries $n'$ markings with $2g'-2+n'>0$, see also \cite[Sec.3.2]{RES} for the condition on the domain.
Let $\ev_i$ and $\ev_i^{\mathrm{rel}}$ be the interior and relative evaluation maps
of the moduli space and define the relative Gromov-Witten invariants:
\[
\left\langle \taut ; \tau_{m_1}(\gamma_1) \cdots \tau_{m_n}(\gamma_n) ; \underline{\eta} \right\rangle^{R/E,\bullet}_{g, \beta}
=
\int_{[ \Mbar_{g,n}'(R/E,\beta;\eta) ]^{\vir}} \taut
\prod_{i=1}^{n} \ev_i^{\ast}(\gamma_i) \psi_i^{k_i} \cdot \prod_{i=1}^{\ell(\eta)} \ev_i^{\mathrm{rel} \ast}(\delta_i),
\]
The generating series of relative invariants is defined by
\[ F_{g,k}^{R/E,\bullet}( \taut ; \tau_{m_1}(\gamma_1) \cdots \tau_{m_n}(\gamma_n) ; \underline{\eta} )
=
\sum_{\substack{ \beta \in H_2(R,\BZ) \\ \pi_{\ast} \beta = k }}
q^{W \cdot \beta} \zeta^{\beta} \left\langle \taut ; \tau_{m_1}(\gamma_1) \cdots \tau_{m_n}(\gamma_n) ; \underline{\eta} \right\rangle^{R/E,\bullet}_{g, \beta}
\]

\begin{thm} \label{thm:R HAE}
We have
\[
F_{g,k}^{R/E,\bullet}( \taut ; \tau_{m_1}(\gamma_1) \cdots \tau_{m_n}(\gamma_n) ; \underline{\eta} )
\in 
\Delta(q)^{-k/2} \QJac_{\frac{k}{2} Q_{E_8}}( \SL_2(\BZ) \ltimes (\BZ^8 \oplus \BZ^8)). \]
We have the $G_2$ holomorphic anomaly equation
\begin{align*}
& \frac{d}{dG_2}
F_{g,k}^{R/E,\bullet}( \taut ; \tau_{m_1}(\gamma_1) \cdots \tau_{m_n}(\gamma_n) ; \underline{\eta} )
= \\
& F_{g-1,k}^{R/E,\bullet}( \taut' ; \tau_{m_1}(\gamma_1) \cdots \tau_{m_n}(\gamma_n) \tau_0\tau_0(\Delta_{\p^1/\mathrm{pt}}^{\mathrm{rel}}) ; \underline{\eta} ) \\
& + 2 \sum_{\substack{ \{ 1, \ldots, n \} = S_1 \sqcup S_2  \\ m \geq 0 \\ g = g_1 + g_2 + m }}
\sum_{\substack{ b ; b_1, \ldots, b_m \\ \ell ; \ell_1, \ldots, \ell_m}}
\frac{\prod_{i=1}^{m} b_i}{m!}
\Bigg[
F_{g_1,k}^{R/E,\bullet}\Big( \taut_1 ; \prod_{i \in S_1} \tau_{m_i}(\gamma_i) ;
\big( (b, 1), (b_i, \Delta_{E, \ell_i})_{i=1}^{m}\big) \Big) \\
& \hspace{7em} \times
F_{g_2, k}^{\p^1 \times E/E_0 \sqcup E_{\infty}, \bullet, \mathrm{rubber}}\Big( \taut_2 ; \prod_{i \in S_2} \tau_{m_i}(\gamma_i) ; \big( (b, 1), (b_i, \Delta^{\vee}_{E, \ell_i})_{i=1}^{m} \big), \underline{\eta} \Big) \Bigg] \\
& - 2 \sum_{i=1}^{n}
F_{g,k}^{R/E,\bullet}( \taut ; \tau_{m_1}(\gamma_1) \cdots
\tau_{m_{i-1}}(\gamma_{i-1}) \tau_{m_i+1}( \pi^{\ast} \pi_{\ast}(\gamma_i)) \tau_{m_{i+1}}(\gamma_{i+1}) \cdots
 \tau_{m_n}(\gamma_n) ; \underline{\eta} ) \\
 & - 2 \sum_{i=1}^{l(\eta)} 
 F_{g,k}^{R/E,\bullet}( \psi_i^{\rel} \cdot \taut ; \tau_{m_1}(\gamma_1) \cdots \tau_{m_n}(\gamma_n) ;
  \big( (\eta_1, \delta_1), \ldots, \underbrace{(\eta_i, \pi_E^{\ast} \pi_{E\ast} \delta_i )}_{i\text{-th}}, \ldots, (\eta_n, \delta_n)\big) \Big)
\end{align*}
and the elliptic anomaly equation
\begin{multline*}
\xi_{\lambda} F_{g,k}^{R/E,\bullet}( \taut ; \tau_{m_1}(\gamma_1) \cdots \tau_{m_n}(\gamma_n) ; \underline{\eta} )
= \\
\sum_{i=1}^{n}
F_{g,k}^{R/E,\bullet}( \taut ; \tau_{m_1}(\gamma_1) \cdots \tau_{m_{i-1}}(\gamma_{i-1}) \tau_{m_i}( t_{\lambda}(\gamma_i)) \tau_{m_{i+1}}(\gamma_{i+1}) \cdots \tau_{m_n}(\gamma_n) ; \underline{\eta} )
\end{multline*}
for all $\lambda \in E_8(-1) \subset H^2(R,\BZ)$,
where $t_{\lambda}(x) = (F \cdot x) \lambda - (\lambda \cdot x) F$.
\end{thm}
\begin{proof}
In the statement of the $G_2$-holomorphic anomaly equation we used additional notation:
For the relative diagonal $\Delta_{\p^1/x}^{\rel}$
and for $\tau_0 \tau_0(\cdots)$
see \cite[Sec.2.4]{Marked} and \cite[Sec.3.2,4.3]{Marked} respectively.
The supscript 'rubber' stands for integrating over the moduli space of rubber relative stable maps to $\p^1 \times E/E_0 \sqcup E_{\infty}$, see \cite[Sec.4.5]{Marked}.
We let $\psi^{\mathrm{rel}}$ be the relative $\psi$-class on the moduli space of stable maps to $(R,E)$, and let $\pi_E:E \to \mathrm{pt}$ be the projection.

The proof of the first two parts of the theorem follows by 
translating Proposition 26 and
Theorems 23 and 24 in \cite{RES}
into the notation that we use here.
The descendent insertions $\psi_i^{k_i}$ on $\Mbar'_{g,n}(R,\beta; \eta)$
can always be traded for tautological classes pulled back from the moduli space of curves. The series
$F_{g,k}^{R/E,\bullet}( \taut ; \gamma_1, \ldots, \gamma_n ; \underline{\eta} )$
in our notation is then precisely $\int p^{\ast}(\taut) \CC_{g,k}^{\pi/E, \bullet}(\gamma_1, \ldots, \gamma_n ; \underline{\eta})$
in the notation of \cite{RES}.
A more subtle detail concerns the formula for the holomorphic anomaly equation.
In \cite{RES} the holomorphic anomaly equation uses the descendent classes $\psi_i \in H^2(\Mbar^{\bullet}_{g,n}(B/\mathrm{pt},k))$ in the third term (see \cite[Sec.3.2]{RES} for the bullet convention here), while we use $\psi_i \in H^2(\Mbar'_{g,n}(R/E,\beta))$.
This difference does not make a difference when integrating.
Indeed, if $q : \Mbar'_{g,n}(R,\beta) \to \Mbar^{\bullet}_{g,n}(B,k)$ is the projection map, the contribution from the difference $q^{\ast}(\psi_i) - \psi_i$ consists of genus 0 invariants in fiber classes, which vanish for dimension reasons (Remark~\ref{rmk:GW stuff}(c)).
Finally, the elliptic anomaly equation follows from Proposition~\ref{prop:elliptic anomaly appendix} in the Appendix.
\end{proof}

We will require a modification of the generating series.
Let $\alpha_0 \in E_8(-1)$ be a class with $\alpha_0^2=-4$. 
Recall from Section~\ref{subsec:Jacobi forms for E8} the operator on quasi-Jacobi forms\footnote{Here the intersection product is on $E_8(-1)$; we have $\alpha_0 \cdot_{E_8} x = -\alpha_0 \cdot_{E_8(-1)} x$.}
\[
M_{\alpha_0,m}(f) :=
q^{2m} e\left( -m (\alpha_0 \cdot_{E_8(-1)} x) \right) \left( e^{\xi_{\alpha_0}/2}  f \right) (x + \alpha_0 \tau, 2 \tau), \]
which defines a morphism
\[
\Delta(q)^{-m} \QJac_{m Q_{E_8}}( \SL_2(\BZ) \ltimes (\BZ^8 \oplus \BZ^8))
\to
\Delta(q^2)^{-m} \QJac_{\frac{1}{2} m Q_{E_8}}( \Gamma_0(2) \ltimes (2 \BZ^8 \oplus \BZ^8) ).
\]

Define the modified series
\[
\widetilde{F}_{g,k}^{R/E,\bullet}( \taut ; \tau_{m_1}(\gamma_1) \cdots \tau_{m_n}(\gamma_n) ; \underline{\eta} )
:=
M_{\alpha_0,k}\Big( 
F_{g,2k}^{R/E,\bullet}( \taut ; \tau_{m_1}(\gamma_1) \cdots \tau_{m_n}(\gamma_n) ; \underline{\eta} ) \Big)
\]

Using the elliptic anomaly equation of Theorem~\ref{thm:R HAE} we 
can write explicitly
\begin{align*}
& \widetilde{F}_{g,k}^{R/E,\bullet}( \taut ; \tau_{m_1}(\gamma_1) \cdots \tau_{m_n}(\gamma_n) ; \underline{\eta} ) \\
=&
q^{2k} 
e\left( -k (\alpha_0 \cdot_{E_8(-1)} x) \right) 
F_{g,2k}^{R/E,\bullet}( \taut ; \tau_{m_1}(e^{t_{\alpha_0}/2} \gamma_1) \cdots \tau_{m_n}(e^{t_{\alpha_0}/2}  \gamma_n)  ; \underline{\eta} )
(x + \alpha_0 \tau, 2 \tau) \\
= & 
\sum_{\substack{ \beta \in H_2(R,\BZ) \\ \pi_{\ast} \beta = 2k }}
q^{(2W+\alpha_0+F) \cdot \beta} e\left( \Big(x-\frac{1}{2} (\alpha_0 \cdot_{E_8(-1)} x) F \Big) \cdot \beta \right) \left\langle \taut ; \tau_{m_1}(e^{t_{\alpha_0}/2} \gamma_1) \cdots \tau_{m_n}(e^{t_{\alpha_0}/2}  \gamma_n); \underline{\eta} \right\rangle^{R/E,\bullet}_{g, \beta} \\
= & 
\sum_{\substack{ \beta \in H_2(R,\BZ) \\ \pi_{\ast} \beta = 2k }}
q^{e^{t_{\alpha_0}/2}(2W) \cdot \beta} e\left( e^{t_{\alpha_0}/2}(x) \cdot \beta \right) \left\langle \taut ; \tau_{m_1}(e^{t_{\alpha_0}/2} \gamma_1) \cdots \tau_{m_n}(e^{t_{\alpha_0}/2}  \gamma_n)  ; \underline{\eta} \right\rangle^{R/E, \bullet}_{g, \beta}
\end{align*}

As a corollary of Theorem~\ref{thm:R HAE} and Lemma~\ref{lemma: QJac scaling to Gamma0(2)} we have:
\begin{cor} \label{cor:F tilde}
We have
\[
\widetilde{F}_{g,k}^{R/E,\bullet}( \taut ; \tau_{m_1}(\gamma_1) \cdots \tau_{m_n}(\gamma_n) ; \underline{\eta})
\in \Delta(q^2)^{-k} \QJac_{\frac{k}{2} Q_{E_8}}( \Gamma_0(2) \ltimes (2 \BZ^8 \oplus \BZ^8) ).
\]
It satisfies the $G_2$-holomorphic anomaly equation 
of Theorem~\ref{thm:R HAE} but with the right hand side multiplied by $\frac{1}{2}$.
\end{cor}

\subsubsection{The elliptic surface $X_1$}
Let $E$ be a smooth elliptic curve.
Let $t_E$ be translation on $E$ by a $2$-torsion point, and let $\mathrm{inv}_{\p^1}$ denote an involution of $\p^1$.
The involution $\tau = (t_E, \mathrm{inv}_{\p^1}) \in \Aut(E \times \p^1)$ is fixed-point free. Define the quotient
\[ X_1 := (E \times \p^1) / \langle \tau \rangle. \]
By projecting to the second factor the surface $X_1$ admits
an isotrivial elliptic fibration
\[ \pi : X_1 \to \p^1 / \langle \mathrm{inv}_{\p^1} \rangle \cong \p^1 \]
with general fiber $E$ and with two double fibers (the half-fibers are both isomorphic to $E' := E/\langle t_E \rangle$.)
Write $f_{X_1} \in H^2(X_1)$ for the class of the half-fiber,
and let $s_{X_1} \in H^2(X_1)$ be the class of the image of $\{ e \} \times \p^1$ for any $e \in E$ under the quotient map $E \times \p^1 \to X_1$. In particular, $s_{X_1}$ is the class of a $2$-section isomorphic to $\p^1$.
We have
\[ s_{X_1}^2 = 0, \quad s_{X_1} \cdot f_{X_1} = 1, \quad f_{X_1}^2 = 0. \]
Define the generating series
\[
F_{g,k}^{X_1}( \taut ; \tau_{m_1}(\gamma_1) \cdots \tau_{m_n}(\gamma_n) )
=
\sum_{d \geq 0} q^d \left\langle \taut ; \tau_{m_1}(\gamma_1) \cdots \tau_{m_n}(\gamma_n) \right\rangle^{X_1}_{g, k s_{X_1} + d f_{X_1}}
\]
\begin{thm} \label{thm:X_1 HAE} Each $F_{g,k}^{X_1}( \taut ; \tau_{m_1}(\gamma_1) \cdots \tau_{m_n}(\gamma_n) )$ is a quasi-modular form for $\SL_2(\BZ)$ (i.e. an element of $\QMod = \BC[G_2, G_4, G_6]$),
	which satisfies the holomorphic anomaly equation \eqref{HAE} with $U$ replaced by
\[ U = \frac{1}{2} (\pi \times \pi)^{\ast} \Delta_{\p^1} = \pr_1^{\ast}(f_{X_1}) + \pr_2^{\ast}(f_{X_1}) \in H^{\ast}(X_1 \times X_1). \]
\end{thm}

For the proof consider the projection to the first factor
\[ p : X_1 \to E' = E/ \langle t_E \rangle, \]
which is a $\p^1$-bundle.
\begin{lemma}
We have $X_1 \cong \p( \CO_{E'} \oplus \CL)$ for a $2$-torsion line bundle $\CL \in \Pic(E')$.
Moreover, under this isomorphism $p$ is the morphism to the base $E'$.
\end{lemma}
\begin{proof}
The $\p^1$-bundle $p : X_1 \to E'$ has two disjoint sections corresponding to the two fixed points of the involution $\mathrm{inv}_{\p^1}$. Hence it is the projectivization of the direct sum of two line bundles.
Let us say $X_1 = \p( \CO_{E'} \oplus \CL)$ for some $\CL \in \Pic(E')$.
A local neighbourhood of one of the sections is given by $(E \times \BC)/\langle (t_E, -1) \rangle$ which shows that the section has normal bundle which is the descent of the line bundle $\CO_E$ along $E \to E'$, where $\CO_E$ carries the negative of the canonical linearization. Hence it is $2$-torsion. Since the normal bundle is isomorphic to $\CL$ or $\CL^{\vee}$, this shows that $\CL$ is $2$-torsion.
\end{proof}

\begin{proof}[Proof of Theorem~\ref{thm:X_1 HAE}]
Gromov-Witten invariants are invariant under deformations of the complex structure.
The relative projective bundle
$\p( \CL_{\mathrm{univ}} \oplus \CO_{E'}) \to E' \times \Pic^0(E')$, where
$\CL_{\mathrm{univ}}$ is the Poincare bundle on $E' \times \Pic^0(E')$,
defines a deformation of $X_1 \to \p^1$ to the trivial elliptic fibration $E' \times \p^1 \to \p^1$.
The fiber class $f_{X_1}$ deforms to the class $[E' \times \mathrm{pt}]$,
and the section class $s_{X_1}$ deforms to $[\mathrm{pt} \times \p^1]$.
The result hence follows immediately from the holomorphic anomaly equation
proven for the trivial elliptic fibration in \cite[Cor.2]{HAE}.
\end{proof}

\begin{rmk} \label{rmk:X1 relative}
Let $E \subset X_1$ be a generic fiber of the elliptic fibraton $X_1 \to \p^1$.
By the degeneration formula applied to the normal cone degeneration
$X_1 \rightsquigarrow X_1 \cup_E (\p^1 \times E)$
the relative Gromov-Witten invariants of $(X_1,E)$ can be expressed
in terms of the absolute invariants of $X_1$ and the relative invariants of $(\p^1 \times E, E_0)$, see the arguments of \cite[Sec.5]{RES} in a parallel situation.
The quasi-modularity property and holomorphic anomaly equation is known
for both the absolute invariants of $X_1$ (by Theorem~\ref{thm:X_1 HAE})
and the relative invariants of $(\p^1 \times E/E)$ (by \cite[Sec.5]{RES}).
The same arguments as in \cite[Sec.5]{RES} then imply that\footnote{Note that the natural fiber class of $\p^1 \times E$ corresponds to the {\em twice} the half-fiber of $X_1$, so needs to be measured by the variable $q^2$.
In other words, we have to apply the operator $R_2$ of Lemma~\ref{lemma:G2 derivative and scaling} to the natural generating series of Gromov-Witten invariants of $(\p^1 \times E,E)$. 	The compatibility of $\frac{d}{dG_2}$ with $R_2$ in Lemma~\ref{lemma:G2 derivative and scaling} shows that in this variable convention the holomorphic anomaly equation for $(\p^1 \times E,E)$ has to be stated for the class $U = \frac{1}{2} \Delta_{\p^1} = \frac{1}{2} (\pr_1^{\ast}([\mathrm{pt} \times E]) + \mathrm{pr}^{\ast}_2([\mathrm{pt} \times  E]))$ which matches precisely the desired formula for $X_1$. The operator $R_2$ is also the source of the $\Gamma_0(2)$-modularity for the relative invariants.}
\[
F_{g,k}^{X_1/E, \bullet}( \taut ; \tau_{m_1}(\gamma_1) \cdots \tau_{m_n}(\gamma_n) ; \underline{\eta} )
=
\sum_{d \geq 0} q^d \left\langle \taut ; \tau_{m_1}(\gamma_1) \cdots \tau_{m_n}(\gamma_n) ; \underline{\eta} \right\rangle^{X_1/E, \bullet}_{g, k s_{X_1} + d f_{X_1}}
\]
is an element in $\QMod(\Gamma_0(2))$ and satisfies the relative $G_2$-holomorphic anomaly equation:
\begin{align*}
& \frac{d}{dG_2}
F_{g,k}^{X_1/E,\bullet}( \taut ; \tau_{m_1}(\gamma_1) \cdots \tau_{m_n}(\gamma_n) ; \underline{\eta} )
= \\
& F_{g-1,k}^{X_1/E,\bullet}( \taut' ; \tau_{m_1}(\gamma_1) \cdots \tau_{m_n}(\gamma_n) \tau_0\tau_0( \frac{1}{2} \Delta_{\p^1/\mathrm{pt}}^{\mathrm{rel}}) ; \underline{\eta} ) \\
& + 2 \sum_{\substack{ \{ 1, \ldots, n \} = S_1 \sqcup S_2  \\ m \geq 0 \\ g = g_1 + g_2 + m }}
\sum_{\substack{ b ; b_1, \ldots, b_m \\ \ell ; \ell_1, \ldots, \ell_m}}
\frac{\prod_{i=1}^{m} b_i}{m!}
\Bigg[
F_{g_1,k}^{X_1/E,\bullet}\Big( \taut_1 ; \prod_{i \in S_1} \tau_{m_i}(\gamma_i) ;
\big( (b, \frac{1}{2}), (b_i, \Delta_{E, \ell_i})_{i=1}^{m}\big) \Big) \\
& \hspace{7em} \times
F_{g_2, k}^{\p^1 \times E/E_0 \sqcup E_{\infty}, \bullet, \mathrm{rubber}}\Big( \taut_2 ; \prod_{i \in S_2} \tau_{m_i}(\gamma_i) ; \big( (b, 1), (b_i, \Delta^{\vee}_{E, \ell_i})_{i=1}^{m} \big), \underline{\eta} \Big) \Bigg] \\
& - 2 \sum_{i=1}^{n}
F_{g,k}^{X_1/E,\bullet}( \taut ; \tau_{m_1}(\gamma_1) \cdots
\tau_{m_{i-1}}(\gamma_{i-1}) \tau_{m_i+1}( \frac{1}{2} \pi^{\ast} \pi_{\ast}(\gamma_i)) \tau_{m_{i+1}}(\gamma_{i+1}) \cdots
 \tau_{m_n}(\gamma_n) ; \underline{\eta} ) \\
 & - 2 \sum_{i=1}^{l(\eta)} 
 F_{g,k}^{X_1/E,\bullet}( \psi_i^{\rel} \cdot \taut ; \tau_{m_1}(\gamma_1) \cdots \tau_{m_n}(\gamma_n) ;
  \big( (\eta_1, \delta_1), \ldots, \underbrace{(\eta_i, \frac{1}{2} \pi_E^{\ast} \pi_{E\ast} \delta_i )}_{i\text{-th}}, \ldots, (\eta_n, \delta_n)\big) \Big).
\end{align*}
%
\end{rmk}

\subsubsection{Degeneration} \label{subsec:Degeneration}
As discussed in \cite{MP} there exists a degeneration 
\[ Y \rightsquigarrow R \cup_{E} X_1 \]
which respects the elliptic fibrations.
More precisely, there exists a morphism
\[ \epsilon : \CY \to \Delta \]
over an open disk $\Delta \subset \BC$ satisfying the following condition:
\begin{enumerate}
\item[(i)] $\CY$ is smooth and $\epsilon$ is flat projective, smooth away from $0$,
\item[(ii)] $\epsilon^{-1}(1) = Y$
\item[(iii)] $\epsilon^{-1}(0) = R \cup_E X_1$ is a normal crossing divisor
\item[(iv)] There exists a flat morphism $\widetilde{\epsilon} : \CB \to \Delta$ satisfying (i-iii) such that $\widetilde{\epsilon}^{-1}(1) = \p^1$
and $\widetilde{\epsilon}^{-1}(0) = \p^1 \cup_{\mathrm{pt}} \p^1$,
\item[(v)] There is an elliptic fibration $\CY \to \CB$ which restricts over $Y,R,X_1$ to the given elliptic fibrations.
\end{enumerate}

The degeneration can be constructed by
degenerating the Horikawa model (see \cite[Sec. VIII.18]{BHPV}) of the Enriques surface
which we recall now: Consider the involution
\[ \iota = \iota_1 \times \iota_2: \p^1 \times \p^1 \to \p^1 \times \p^1, \quad ([s_0, s_1], [t_0, t_1]) \mapsto ([s_0, -s_1], [t_0, -t_1]). \]
Let $D$ be a $(4,4)$-divisor on $\p^1 \times \p^1$ defined by a generic element of the
vector space of $\iota$-invariant sections of $H^0(\p^1 \times \p^1, \CO(4,4))$.
Then $D$ is smooth. Let $X \to \p^1 \times \p^1$ be the double cover branched along $D$.
The involution lifts naturally to an involution $\widetilde{\iota} : X \to X$
which commutes with the covering involution $\mathrm{cov}$.
The composition $\tau := \mathrm{cov} \circ \widetilde{\iota}$ is a fixed-point free involution, whose quotient $Y := X / \langle \tau \rangle$ is a generic Enriques surface.

Let $x,y \in \p^1$ be distinct points interchanged by the involution $\iota_1 : [s_0, s_1] \to [s_0, -s_1]$.
Consider the degeneration
$\mathrm{Bl}_{(x,0), (y,0)} (\p^1 \times \BA^1) \to \BA^1$.
The central fiber is the chain $\p^1 \cup_{x} \p^1 \cup_{y} \p^1$.
The involution $\iota_1$ lifts to a fiberwise involution $\widetilde{\iota}_1$ of the degeneration which on the central fiber acts by $\iota_1$ on the middle copy of $\p^1$, and interchanges the two outer components.
Taking the product with $\p^1$, we obtain a degeneration
\[ P := \mathrm{Bl}_{(x,0), (y,0)} (\p^1 \times \BA^1) \times \p^1 \to \BA^1 \]
such that the involution $I := \widetilde{\iota}_1 \times \iota_2$
acts on the central fiber $(\p^1 \cup_{x} \p^1 \cup_{y} \p^1) \times \p^1$
by $\iota$ on the middle component $\p^1 \times \p^1$ and by swapping the two outher components. Choose a lift of $D$ to a $I$-invariant divisor $\CD$ on $P$ which restricts to a smooth divisor of type $(2,4)$, $(0,4)$, $(2,4)$ over the components of the central fiber.
Let $\CX \to \BA^1$ be the double cover of $P$ branched along $\CD$,
let $\widetilde{I} : \CX \to \CX$ be the lift of $I$, and let $\mathrm{Cov}$ be the covering involution.
The $(2,4)$-double cover of $\p^1 \times \p^1$ is a rational elliptic surface,
and the $(0,4)$-double cover is $\p^1 \times E$.
The involution $T = \mathrm{Cov} \circ \widetilde{I}$
then swaps the two rational elliptic surfaces,
and acts as $\iota_1 \times t_E$ on the $\p^1 \times E$ components.
Let $\CY := \CX / \langle T \rangle$.
Then after possible shrinking $\BA^1$ to an open disk $\Delta$,
in order to make $\CX$ smooth and $T$ fixed-point free,
we obtain the desired degeneration $\CY \to \Delta$.
Note that the elliptic fibration on $\CY$ is precisely the lift of the projection to the first factor of $\mathrm{Bl}_{(x,0), (y,0)} (\p^1 \times \BA^1) \times \p^1$. On the other hand, the quotient by $\langle T \rangle$ 
of the double cover of
\[ \mathrm{Bl}_{(x,0), (y,0)} (\p^1 \times \BA^1) \times \{ 0 \} \]
defines a natural divisor $s_{\CY} \in H^2(\CY)$ which restricts to
$s \in H^2(Y, \BZ)$ on $Y$.
On the central fiber it restricts to $s_{X_1}$ on $X_1$,
and to a divisor $s_R$ which is the preimage of $\p^1 \times 0$
under the $(2,4)$-branched cover $R \to \p^1 \times \p^1$.
Since $s_R$ is rationally equivalent to the preimage of $\p^1 \times \{ t \}$ for any $t \in \p^1$ and this becomes reducible for some $t$, we see that
$s_R = B_{1} + B_2 \in \Pic(R)$ for two sections $B_1, B_2$ of $\pi : R \to \p^1$ meeting in a point. We can choose the distinguished section of the rational elliptic surface to be $B := B_1$. This implies that $B_2$ is the section associated to a $(-4)$-class $\alpha_0 \in E_8(-1) \subset H^2(R,\BZ)$, that is
$B_2 = B_{\alpha_0} := B + 2F + \alpha_0$.
Hence
\[ s_R = B + B_{\alpha_0} = 2B + 2F + \alpha_0 = 2W + F + \alpha_0. \]
This proves the first part of the following:

\begin{lemma} \label{lemma:liftings}
\begin{itemize}
\item[(i)] There exists a class $\widetilde{s} \in H^2(\CY,\BQ)$ such that
\[ \widetilde{s}|_{Y} = s,
\quad \widetilde{s}|_{X_1} = s_{X_1},
\quad \widetilde{s}|_{R} = B + B_{\alpha_0} = 2B + 2F + \alpha_0. \]
\item[(ii)] For every $\alpha \in E_8(-1) \subset H^2(Y,\BZ)$ there
exists a class $\widetilde{\alpha} \in H^2(\CY,\BQ)$ such that
\[ \widetilde{\alpha}|_{Y} = \alpha, \quad \widetilde{\alpha}|_{X_1} = 0, \quad \widetilde{\alpha}|_{R} = j(\alpha) = \alpha - \frac{1}{2} (\alpha \cdot \alpha_0) F. \]
\item[(iii)] There exists a class $\widetilde{f} \in H^2(\CY,\BQ)$ and $\widetilde{\pt} \in H^{\ast}(\CY,\BQ)$ such that
\[ \widetilde{f}|_{Y} = f, \quad \widetilde{f}|_{X_1} = 0, \quad \widetilde{f}|_{R} = \frac{1}{2} F. \]
\[ \widetilde{\pt}|_{Y} = \pt, \quad \widetilde{\pt}|_{X_1} = 0, \quad \widetilde{\pt}|_{R} = \pt, \]
where $\pt$ denotes the point class on $Y$ and $R$.
\end{itemize}
\end{lemma}
\begin{proof}
We prove (ii). The degeneration $\CX \to \Delta$ is a Type II degeneration of a K3 surface. The associated Clemens-Schmid exact sequence is well-understood and can be found for example in \cite{Greer}. By taking the invariant part, we obtain the following Clemens-Schmid exact sequence for the degeneration $\CY \to \Delta$:
\[ 0 \to H^0(Y,\BQ) \to H_4(\CY_0,\BQ) \to H^2(\CY_0,\BQ) \to H^2(Y,\BQ) \to 0. \]
In particular, there is no vanishing cohomology and every class in $H^2(Y,\BQ)$ can be lifted. If $\alpha \in E_8(-1) \subset H^2(Y,\BZ)$, the class $s_{\alpha} = s + \alpha - (\alpha^2/2) f$ is a $2$-section of $\pi : Y \to \p^1$ which is the half-fiber of another elliptic fibration on $Y$.
By considering the Horikawa model for the pair $(f,s_{\alpha})$ we see that
there is in fact an integral class $\widehat{s}_{\alpha} \in H^2(\CY, \BZ)$ such that $\widehat{s}_{\alpha}|_{Y} = s_{\alpha}$, and hence an integral class $\widehat{\alpha} \in H^2(\CY, \BZ)$ such that $\widehat{\alpha}|_{Y} = \alpha$ (like $s$, the class $f$ can be easily lifted to an integral class).
Alternatively, the existence of an integral lift of $\alpha$ also follows since the Clemens-Schmid exact sequence for $\CX \to \Delta$ is exact over $\BZ$, see \cite{Friedman}.
Now, by the Clemens-Schmid sequence, the pair of restrictions
\[ (\widehat{\alpha}|_{X_1}, \widehat{\alpha}|_{R}) \in H^2(X_1,\BZ) \times H^2(R,\BZ) \]
is unique up to adding an integral multiple of $(2 f_{X_1}, -F)$.
We can lift $2f$ to classes $\hat{F}_1, \hat{F}_2 \in H^2(\CY,\BZ)$ such that
\[ (\hat{F}_1|_{X_1}, \hat{F}_1|_{R})=(2f_{X_1},0), \quad 
(\hat{F}_2|_{X_1}, \hat{F}_2|_{R})=(0,F). \]
From $\int_Y \alpha \cdot f = 0$ we hence obtain that
\[ 0 = \int_Y \alpha \cdot 2f = \int_{Y} \widehat{\alpha} \cup \widehat{F}_i
= \int_{R} \widehat{\alpha}|_R \cdot \widehat{F}_i|_R +
\int_{X_1} \widehat{\alpha}|_{X_1} \cdot \widehat{F}_i|_{X_1}. \]
Inserting $i=1,2$ we get that
\[ \widehat{\alpha}|_{X_1} \in \BZ f, \quad 
\widehat{\alpha}|_R \in \BZ F \oplus E_8(-1) \subset H^2(R,\BZ). \]
Replace now $\widetilde{\alpha}$ with the unique half-integral class
$\widetilde{\alpha} + \frac{\ell}{2} (-\hat{F}_1 + \hat{F}_2)$ 
for some $\ell \in \BZ$ such that 
\[ \widetilde{\alpha}|_{X_1} = 0, \quad \widetilde{\alpha}|_{R} = b F + g(\alpha) \in \frac{1}{2} \BZ F \oplus E_8(-1) \]
where $g(\alpha) \in E_8(-1)$.
Since $s \cdot \alpha = 0$, we get $\widetilde{\alpha}|_{R} \cdot \widetilde{s}|_{R} = 0$, and therefore $b = -(g(\alpha) \cdot \alpha_0)/2$.
The map
\[ \alpha \mapsto \widetilde{\alpha}|_{R} = g(\alpha) -(g(\alpha) \cdot \alpha_0)/2 F \xrightarrow{j^{-1}} E_8(-1) \]
is then given by $\alpha \mapsto g(\alpha)$ and an integral isometry of the $E_8$-lattice.
After changing our identification of $H^2(R,\BZ) \cong U \oplus E_8(-1)$ we hence can assume $g=\id$. Thus we get as desired
\[ \widetilde{\alpha}|_{R} = \alpha - \frac{1}{2} (\alpha \cdot \alpha_0) F. \qedhere \]
The case (iii) is immediate.
\end{proof}

\subsubsection{Degeneration formula}
\label{subsubsec:applying degeneration}
We now apply the degeneration formula of \cite{Li1, Li2} to the degeneration
$Y \rightsquigarrow R \cup_{E} X_1$ discussed in the last section.
Given cohomology classes $\gamma_1, \ldots, \gamma_n \in H^{\ast}(Y,\BZ)$
we lift the classes in $H^2(Y,\BZ) \oplus H^4(Y,\BZ)$ to the total space of the degeneration as constructed in Lemma~\ref{lemma:liftings}. The unit is lifted to the unit. This yields the following:

Define the isometry $\varphi:H^{\ast}(Y,\BQ) \to H^{\ast}(R,\BQ)$ by
\[ \varphi(1)=1, \quad \varphi(\pt)= \pt, \quad \varphi(s)=2W, \quad \varphi(f) = \frac{1}{2} F, \quad \varphi(\alpha)=\alpha \text{ for } \alpha \in E_8(-1) \subset H^2(Y,\BZ) \]
For every $\gamma \in H^{\ast}(Y,\BZ)$ with lift $\widetilde{\gamma}$ we then have
\[ e^{\frac{1}{2} t_{\alpha_0}}( \varphi(\gamma)) = \widetilde{\gamma}|_{R}. \]
Similarly, define $\psi : H^{\ast}(Y,\BZ) \to H^{\ast}(X_1,\BQ)$ by
\[ \psi(1)=1, \quad \psi(s) = s_{X_1}, \quad \psi(x)=0 \text{ for } x \in \{ f, \pt \} \sqcup E_8(-1). \]
Consider the series of disconnected\footnote{As before, we require that the every connected component $C'$ of the domain curve of the stable map $f:C\to Y$ either satisfies (i) $\pi \circ f|_{C'}$ is non-constant, or (ii) $C'$ has genus $g'$ and $n'$ markings with $2g'-2+n'>0$.} Gromov-Witten invariants of $Y$,
\[ F^{Y,\bullet}_{g,k}( \taut ; \tau_{m_1}(\gamma_1) \cdots \tau_{m_n}(\gamma_n) )
=
\sum_{\substack{ \beta \in H_2(Y,\BZ) \\ \beta \cdot f = k }}
q^{s \cdot \beta} \zeta^{\beta} \left\langle \taut ; \tau_{m_1}(\gamma_1) \cdots \tau_{m_n}(\gamma_n) \right\rangle^{Y,\bullet}_{g, \beta}
\]
The quasi-Jacobi form property and holomorphic anomaly equations
are equivalent for disconnected and connected series,
so we may consider the disconnected case, see \cite[Sec.3.2]{RES}.

The degeneration formula yields:
\begin{multline}
F_{g,k}^{Y, \bullet}( \taut; \tau_{m_1}(\gamma_1) \cdots \tau_{m_n}(\gamma_n) ) \\
=
\sum_{g_1, g_2, \eta} \sum_{ 
\{ 1, \ldots , n \} = A \sqcup B}
F_{g_1,k}^{X_1/E, \bullet}\left( \taut_1; \prod_{j \in A}\tau_{m_j}(
\psi(\gamma_j))\, ; \eta \right)
\widetilde{F}_{g_2,k}^{R/E, \bullet} \left( \taut_2; \prod_{j \in B} \tau_{m_j}(\varphi(\gamma_j)) \, ; \eta^{\vee} \right)
\label{degeneration formula}
\end{multline}
By Corollary~\ref{cor:F tilde} we have
\begin{equation*} \widetilde{F}_{g_2,2k}^{R/E, \bullet} \left( \taut_2; \prod_{j \in B} \tau_{m_j}(\gamma_j) \, ; \eta^{\vee} \right)
\in \Delta(q^2)^{-k} 
\QJac_{\frac{1}{2} k Q_{E_8}}(\Gamma_0(2) \ltimes (2 \BZ^8 \oplus \BZ^8)).
\end{equation*}
By Remark~\ref{rmk:X1 relative} we have
\[ F_{g_1,k}^{X_1/E, \bullet}\left( \taut_1; \prod_{j \in A}\tau_{m_j}(\gamma_j) \, ; \eta \right) \in \QMod(\Gamma_0(2)). \]
Hence together we obtain that:
\begin{equation} \label{dfsdf211}
F_{g,k}^{Y, \bullet}( \taut; \tau_{m_1}(\gamma_1) \cdots \tau_{m_n}(\gamma_n) )
\in \Delta(q^2)^{-k} \QJac_{\frac{1}{2} k Q_{E_8}}(\Gamma_0(2) \ltimes (2 \BZ^8 \oplus \BZ^8)).
\end{equation}

Moreover, the proof of \cite[Prop.21]{RES} shows the compatibility of the holomorphic anomaly equation for $Y$, $(X_1,E)$ and $(R,E)$ with the degeneration formula.
Hence the holomorphic anomaly equation for $(X_1,E)$ and $(R,E)$
given in Remark~\ref{rmk:X1 relative} and Corollary~\ref{cor:F tilde}
imply the holomorphic anomaly equation for $Y$ given in Theorem~\ref{thm:HAE}.

\subsubsection{Use of the monodromy} \label{subsubsec:monodromy}
In the last section we have seen that
\begin{equation} \label{dsfodfos} F_{g,k}( \taut; \tau_{m_1}(\gamma_1) \cdots \tau_{m_n}(\gamma_n) )
\in \Delta(q^2)^{-k} \QJac_{\frac{1}{2} k Q_{E_8}}(\Gamma_0(2) \ltimes (2 \BZ^8 \oplus \BZ^8)). \end{equation}
and satisfies the holomorphic anomaly equation \eqref{HAE}.
We prove here that the series is a quasi-Jacobi form for the larger group $\Gamma_0(2) \ltimes (\BZ^8 \oplus \BZ^8)$.

To do so, observe first that each
$F_{g,k}( \taut; \tau_{m_1}(\gamma_1) \cdots \tau_{m_n}(\gamma_n) )$
satisfies the elliptic holomorphic anomaly equation stated in Corollary~\ref{cor:elliptic anomaly}. Indeed, the proof of Corollary~\ref{cor:elliptic anomaly}, which is given in \cite[Sec.3.3]{RES}, does not require the larger group, because it is only uses relations coming from the Lie algebra of the Jacobi group.

Next we use the monodromy group of the Enriques surface.
For any $\lambda \in E_8(-1) \subset H^2(Y,\BZ)$ recall the operator:
\[ t_{\lambda}(x) = (f \cdot x) \lambda - (\lambda \cdot x) f. \]
Exponentiating we obtain an element in the monodromy group:
\[ e^{t_{\lambda}} \in O^{+}(H^2(Y,\BZ)) = \Mon(Y). \]
By using deformation invariance of Gromov-Witten invariants we get
\[ \left\langle \taut ; \prod_i \tau_{k_i}(\gamma_i) \right\rangle^{Y}_{g,ks + df + \alpha}
=
\left\langle \taut ; \prod_i \tau_{k_i}( e^{t_{\lambda}} \gamma_i) \right\rangle^{Y}_{g,k s + (d - \alpha \cdot \lambda - \frac{1}{2} k \lambda^2) f + \alpha + k \lambda}.
\]
Thus we have
\begin{align*}
F_{g,k}( \taut; \tau_{m_1}(\gamma_1) \cdots \tau_{m_n}(\gamma_n) )
& =
q^{-\frac{1}{2} k \lambda^2} \zeta^{k \lambda} F_{g,k}( \taut; \tau_{m_1}( e^{t_{\lambda}} \gamma_1) \cdots \tau_{m_n}( e^{t_{\lambda}} \gamma_n) )(x + \lambda \tau, \tau) \\
& = 
q^{-\frac{1}{2} k \lambda^2} \zeta^{k \lambda} 
e^{\xi_{\lambda}} F_{g,k}( \taut; \tau_{m_1}( \gamma_1) \cdots \tau_{m_n}( \gamma_n) )(x + \lambda \tau, \tau),
\end{align*}
where in the last line we used the elliptic anomaly equation.
This precisely says that the non-holomorphic Jacobi form corresponding to
\eqref{dsfodfos} satisfies the transformation law of Jacobi forms under the action of $(x,\tau) \mapsto (x + \lambda \tau, \tau)$.

\subsubsection{Prefactor} \label{subsubsec:prefactor}
In the last section we have seen that
\[ F_{g,k}( \taut; \tau_{m_1}(\gamma_1) \cdots \tau_{m_n}(\gamma_n) )
\in \Delta(q^2)^{-k} \QJac_{\frac{1}{2} k Q_{E_8}}(\Gamma_0(2) \ltimes (\BZ^8 \oplus \BZ^8)). \]
It remains to show that
\[ F_{g,k}( \taut; \tau_{m_1}(\gamma_1) \cdots \tau_{m_n}(\gamma_n) )
\in \left(\frac{ \eta^{8}(q^2) }{\eta^{16}(q)} \right)^k \QJac_{\frac{1}{2} k Q_{E_8}}(\Gamma_0(2) \ltimes (\BZ^8 \oplus \BZ^8)). \]

Let us denote $F = F_{g,k}( \taut; \tau_{m_1}(\gamma_1) \cdots \tau_{m_n}(\gamma_n) )$ and write
\[ F = \frac{1}{\Delta(q^2)^k} H, \quad \text{where} \quad H \in \QJac_{\frac{1}{2} k Q_{E_8}}(\Gamma_0(2) \ltimes (\BZ^8 \oplus \BZ^8)). \]
By Remark~\ref{rmk:GW stuff} $F$ has no terms with negative $q$-exponents. We find that $H(p,q) = O(q^{2k})$. 

Let $f(q) = \eta^{16}(q^2) / \eta^{8}(q)$. By Lemma~\ref{lemma:function vanishing at cusp} it follows that
\[ \widetilde{H} := \frac{H}{f^{2k}} \in \QJac_{\frac{1}{2} k Q_{E_8}}(\Gamma_0(2) \ltimes (\BZ^8 \oplus \BZ^8)). \]
We conclude that
\[ F = \frac{f(q)^{2k}}{\Delta(q^2)^k} \widetilde{H} 
= \left( \frac{\eta^8(q^2)}{\eta(q)^{16}} \right)^{k} \widetilde{H}. \]
This completes the proof of Theorem~\ref{thm:HAE}. \qed

\section{Donaldson-Thomas theory of the Enriques Calabi-Yau threefold}
\label{sec:DT of Enriques CalabiYau}
\subsection{Definition}
Let $X$ be a K3 surface and let $\tau : X \to X$ be a fixed-point free involution.
Let $E$ be an elliptic curve and consider the involution
\[ (\tau, -1) : X \times E \to X \times E, \quad (x,e) \mapsto (\tau(x), -e). \]
The Enriques Calabi-Yau threefold is the quotient:
\[ Q = (X \times E)/G, \quad G = \langle (\tau,-1) \rangle. \]
By projecting to the second factor $Q$ has an isotrivial K3 fibration with 4 double Enriques fibers:
\[ p : Q \to \p^1 = E/\langle -1 \rangle. \]
By projecting to the first factor we have an isotrivial elliptic fibration
\[ \pi_Y : Q \to X/\langle \tau \rangle = Y. \]
The fibration $\pi_Y$ has $4$ sections 
index by the $2$-torsion points $a \in E[2]$,
\[ \iota_a : Y_a \hookrightarrow Q, \quad Y_a := (X \times a)/\BZ_2. \]
Since we work modulo torsion, the pushforward of classes
\[ \iota_{a \ast} : H_{\ast}(Y,\BZ) \to H^{\ast}(Q,\BZ) \] 
is independent of $a$ and we often drop $a$ from the notation.

\subsection{Overview} \label{subsec:overview}
In this section we study the invariants of the Enriques Calabi-Yau threefold,
in particular, the Gromov-Witten and Pandharipande-Thomas invariants and the correspondence between them (Section~\ref{subsec:correspondence}),
and the generalized Donaldson-Thomas invariants of Joyce and Song (Section~\ref{subsec:DT fiber}).
We consider only fiber curve classes and fiber sheaves
with respect to the fibration $p : Q \to \p^1$.
For the Donaldson-Thomas invariants
we follow closely the work of Toda \cite{Toda}
on the local K3 surface $K3 \times \BC$.
First, in Theorem~\ref{prop:DT dependence}
we establish that the DT invariants are unchanged under the derived monodromy group and hence only depend on the square, divisibility and type of the class.
Then in Theorem~\ref{thm:Todas formula} we prove Toda's formula relating Pandharipande-Thomas invariants and $2$-dimensional DT invariants. Applications are discussed in Section~\ref{subsec:Consequences of Toda} of which the most important is that the Gromov-Witten invariant depend on the curve class only through the square and the divisibility (Proposition~\ref{prop:GW dependence}).

A key observation throughout the section is that Toda's methods from \cite{Toda} carry over here almost literally 
by applying them to $K3 \times E$ instead and then taking every step $G$-equivariantly.
This is based on the priciple:
\vspace{3pt}
\begin{center} Geometry of $Q$\quad $=$\quad $G$-equivariant Geometry of $X \times E$.\\
\end{center}
\vspace{3pt}
For example, one has the equivalence of the (derived) category of coherent sheaves on $Q$
with the (derived) category of $G$-equivariant sheaves on $X \times E$,
\[ \Coh(Q) \cong \Coh_G(X \times E), \quad D^b(\Coh(Q)) \cong D^b_G(\Coh(X \times E)). \]
Under this equivalence, Gieseker stability of sheaves on $Q$ correspond to Giesker stability of sheaves on $X \times E$ with respect to the pullback of the polarization. Moreover, $G$-invariant Bridgeland stability conditions (on subcategories of) $D^b(X \times E)$ induce stability conditions on (corresponding subcategories of) $D^b(Q)$ \cite{MMS}.
$G$-equvariant autoequivalences on $X$ induce $G$-equivariant auto-equivalences of $X \times E$, which then induce auto-equivalences of $Q$ \cite{Ploog}.
A $G$-invariant semi-orthogonal decomposition induces a semi-orthogonal decomposition on the equivariant category \cite{Elagin}.
Moduli stacks of semi-stable sheaves on $Q$ are the fixed loci of the induced $G$-action on moduli stacks of semistable sheaves on $X \times E$ \cite{BO1}, etc.
Hence at several steps below we will just refer to Toda's work, instead of rewriting every detail. For an introduction to equivariant categories and further references we refer to \cite{BO2}.

\subsection{Gromov-Witten theory}
For $\beta \in H_2(Y,\BZ)$ consider the Gromov-Witten invariant
\[ N^Q_{g,\beta} = \int_{[ \Mbar_{g}(Q,\iota_{\ast} \beta) ]^{\vir}} 1. \]
By \cite{MP} one has the following relationship between the Gromov-Witten invariants of $Q$ and the invariants \eqref{Ngbeta} of the Enriques $Y$, which we denote here for clarity by $N_{g,\beta}^Y$.

\begin{prop}[{\cite[Lemma 2]{MP}}] \label{prop:GW QvsY}
 $N_{g,\beta}^Q = 4 N_{g,\beta}^Y$. \end{prop}
\begin{proof}
We give a sketch. Let $\mathrm{inv}_{\p^1}$ be an involution on $\p^1$ and consider the threefold
\[ T = (X \times \p^1)/ \langle (\tau, \mathrm{inv}_{\p^1}) \rangle. \]
The projection $T \to Y$ is isomorphic to the $\p^1$-bundle $\p(\CO_Y \oplus \omega_Y) \to Y$. There exists\footnote{The quotient map $E \to E/\langle -1 \rangle = \p^1$ has $4$ branch points. Degenerate $\p^1$ to the union of two $\p^1$'s meeting at a point,
where two of the branch points specialize to each component. By taking the double cover of the total space of this degeneration branched along this locus and the corresponding covering involution we obtain a degeneration of the pair $(E,-1)$ to
$(\p^1 \cup_{x,y} \p^1, \mathrm{inv})$, where $\mathrm{inv}$ acts by $\mathrm{inv}_{\p^1}$ on each component and interchanges the two gluing points $x,y$.
The degeneration $Q \rightsquigarrow T \cup_X T$ is constructed from this by taking the product with $X$ and taking the quotient with respect to the diagonal action.}
 a degeneration $Q \rightsquigarrow T \cup_{X} T$,
which gives (with the obvious notation)
\[
N_{g,\beta}^Q = 2 N_{g,\beta}^{T/X} = 2 N_{g,\beta}^{T} = 4 N_{g,\beta}^{Y}.
\]
where the second equality follows by the degeneration formula for
the normal cone degeneration $T \rightsquigarrow T \cup_X (X \times \p^1)$
and since the Gromov-Witten invariants of $X \times \p^1$ vanish by a cosection argument, and the third equality follows by a localization argument with respect to the fiber $\BC^{\ast}$-action of $T \to Y$.
\end{proof}

Let also $N^{Q,\prime}_{g,\beta}$ denote the Gromov-Witten invariants of $Q$ with disconnected domain, but with no collapsed connected components \cite{PP}.
The connected and disconnected Gromov-Witten invariants can be related by the
transformation
\[ \exp\left( \sum_{g \geq 0} \sum_{\beta>0} N_{g,\beta} u^{2g-2} q^{\beta} \right) = \sum_{g, \beta} N^{Q, \prime}_{g,\beta} u^{2g-2}  q^{\beta}. \]

\subsection{GW/PT correspondence} \label{subsec:correspondence}
Consider the Pandharipande-Thomas invariant \cite{PT}
\[ \PT_{n,\beta} := \int_{[ P_{n,\iota_{\ast} \beta }(Q) ]^{\vir} } 1, \]
where we let $P_{n,\beta}(Q)$ denote the moduli space of stable pairs $(F,s)$ on $Q$
with $\chi(F)=n$ and $\ch_2(F) = \beta$.
In \cite[Sec.7.6]{PP}
the GW/PT correspondence \cite{PT} was proven for the Enriques Calabi-Yau threefold.
In particular, in the special case of fiber classes we obtain:
\begin{thm}[Pandharipande-Pixton, \cite{PP}]
\label{thm:GW correspondence}
For any $\beta \in H_2(Y,\BZ)$ the series $\sum_{n} \PT_{n,\beta} (-p)^n$ is the expansion of a rational function, and under the variable change $p=e^{z}$ we have
\[
\sum_{n} \PT_{n,\beta} (-p)^n = \sum_{g} N^{Q,\prime}_{g,\beta} (-1)^{g-1} z^{2g-2}.
\]
\end{thm} 

\subsection{Donaldson-Thomas invariants of fiber sheaves}
\label{subsec:DT fiber}
Consider a fixed class
\[ v = (r,\beta,n) \in H^{\ast}(Y) = H^0(Y,\BZ) \oplus H^2(Y,\BZ) \oplus H^4(Y,\BZ). \]
For any $v=(r,\beta,n)$ and $v'=(r',\beta',n')$ we define the Euler pairing\footnote{Previously we defined the Mukai pairing $(v,v') = \beta \beta' - r n' - r'n$. Here we work with the Chern character of a sheaf and use therefore a different pairing. The both pairings are related by 
$(v, v' \td_Y) = v \cdot v'$.}
\[
v \cdot v' := \int_{Y} \beta \beta' - r r' - r n' - r' n. 
\]
For any objects $E,F \in D^b(Y)$ we have
\[
\ch(F) \cdot \ch(G) = -\chi(F,G).
\]

Consider the cone of effective classes\footnote{The class $\ch(F)$ lies in $H^{\ast}(Y,\BZ)$ since the intersection form on $H^2(Y,\BZ)$ is even.}
\[ C(\Coh\, Y) := \mathrm{Im}( \ch : \Coh(Y) \to H^{\ast}(Y,\BZ) ). \]

\begin{defn}
Let $\omega \in \Pic(Q)$ be ample and let $v \in H^{\ast}(Y,\BZ)$.
Define the invariant $\DT_{\omega}(v) \in \BQ$ as follows:
\begin{itemize}
\item If $v \in C(\Coh Y)$, then let
$\DT_{\omega}(v)$ be the generalized Donaldson-Thomas invariant of $Q$ defined by Joyce and Song \cite{JS}, which counts $\omega$-Gieseker semi-stable sheaves $\CF$ on $Q$ with Chern character $\ch(\CF) = \iota_{\ast} v \in H^{\ast}(Q,\BQ)$ for the ample class $\omega$.
\item If $-v \in C(\Coh Y)$, then set $\DT_{\omega}(v) := \DT_{\omega}(-v)$.
\item If $\pm v \notin C(\Coh Y)$, then set $\DT_{\omega}(v) := 0$.
\end{itemize}
\end{defn}


For a non-zero vector $v$ in a lattice $L$, we 
write $\mathrm{div}(v)$ (or $\mathrm{div}_L(v)$ if we want to emphasize the lattice) for the largest positive integer such that $v/\mathrm{div}(v) \in L$.
If $v=0$ we also set $\mathrm{div}(v)=0$.
We analyse the dependence of $\DT_{\omega}(v)$ on the polarization $\omega$ and the Chern character $v$:

\begin{thm} \label{prop:DT dependence}
Let $v=(r,\beta,n) \in H^{\ast}(Y,\BZ)$.
\begin{enumerate}
\item[(i)] The invariant $\DT_{\omega}(v)$ does not depend on the choice of $\omega$. We write $\DT(v) = \DT_{\omega}(v)$.
\item[(ii)] The invariant $\DT(v)$ depends upon $v$ only through:
\begin{itemize}
\item the square $d := v \cdot v = \beta^2 - r^2 - 2rn$,
\item the divisibility\footnote{Please note that we multiply $n$ by $2$.} $m := m(v) := \mathrm{div}_{H^{\ast}(X,\BZ)^G}(\pi^{\ast} v) = \mathrm{gcd}(r,\mathrm{div}(\beta),2n)$.
\item the type $t:=t(v) \in \{ \text{odd}, \text{even}\}$ of
$\sqrt{\td_X} \cdot \pi^{\ast}(v)/\mathrm{div}(\pi^{\ast} v)$
in $H^{\ast}(X,\BZ)^G$, given by
\[
t = \begin{cases}
\text{ even } & \text{ if } \frac{r}{m},\frac{2n+r}{m} \text{ are both even } \\
\text{ odd } & \text{ otherwise }.
\end{cases}
\]
\end{itemize}
We write $\DT^{t}_{d,m} := \DT(v)$.
\end{enumerate}
\end{thm}

\begin{proof}
(i) This is parallel to \cite[Proof of Thm. 4.21]{Toda}.
For any two choices $\omega, \omega'$ the invariants $\DT_{\omega}(v)$, $\DT_{\omega'}(v)$ are related by Joyce's wall-crossing formula \cite{Joyce4}. Let $\chi(v,w) = \int_Q v^{\vee} w \cdot \td_{Q}$ be the Euler pairing on $H^{\ast}(Q,\BQ)$. Then because of the vanishing $\chi( \iota_{\ast} v, \iota_{\ast} v') = 0$, the wall-crossing terms do not contribute, and we have $\DT_{\omega}(v) = \DT_{\omega'}(v)$.

(ii) We use the notation of Section~\ref{subsec:derived monodromy group}.
Let $\Lambda_Y \subset H^{\ast}(Y,\BQ)$ be the lattice generated by all Mukai vectors of objects in $D^b(Y)$. Consider the isomorphism
\[ \varphi : H^{\ast}(Y,\BZ) \xrightarrow{\cdot \sqrt{\td_Y}} \Lambda_Y \xrightarrow{\pi^{\ast}} \pi^{\ast} \Lambda_Y \subset H^{\ast}(X,\BZ)^G. \]
Precisely, for $v=(r,\beta,n) \in H^{\ast}(Y,\BZ)$ we have
\[ \varphi(v) = \pi^{\ast}(v \sqrt{\td_Y}) = \pi^{\ast}(v) \sqrt{\td}_X = (r, \pi^{\ast} \beta, 2n+r) \in \pi^{\ast} \Lambda_Y \subset H^{\ast}(X,\BZ)^G. \]
For $w \in \pi^{\ast} \Lambda_Y$ define $J(w) := \DT(\varphi^{-1}(w))$.

Any derived autoequivalence $\Phi : D^b(Y) \to D^b(Y)$ of the Enriques surface
lifts to a $G$-equivariant derived autoequivalence $\widetilde{\Phi} : D^b(X) \to D^b(X)$ which induces a $G$-equivariant action 
$\widetilde{\Phi}^H : H^{\ast}(X,\BZ) \to H^{\ast}(X,\BZ)$.
By restriction we obtain
\[ \widetilde{\Phi}^H|_{H^{\ast}(X,\BZ)^G} : H^{\ast}(X,\BZ)^G \to H^{\ast}(X,\BZ)^G \]
which sends $\pi^{\ast}(\Lambda_Y)$ to itself.
Assume that $\widetilde{\Phi}$ preserves the distinguished component of the stability manifold $\Stab(X)$ constructed by Bridgeland in \cite{Bridgeland} and that $Y$ is generic. Then by Proposition~\ref{prop:automorphy property} below we have
\[
J( \widetilde{\Phi}^H w ) = J(w).
\]

Any deformation between two Enriques surfaces induces a deformation of the covering K3 surfaces, and hence a deformation of the associated Enriques Calabi-Yau threefolds.
By the deformation invariance of generalized Donaldson-Thomas invariants, it follows that $\DT(v) = \DT(gv)$ and $J(w) = J(\tilde{g} w)$ for any parallel-transport operator $g$ and lifted parallel transport operator $\tilde{g}$ between two Enriques surfaces respectively.

By Remark~\ref{rmk:distinguished component} we conclude the
basic invariance
\[ \forall\, g \in \widetilde{\DMon}(Y): \quad J(w) = J(gw). \]

With respect to the Mukai lattice we have
\[ H^{\ast}(X,\BZ)^G \cong U \oplus U(2) \oplus E_8(-2), \]
and by Corollary~\ref{cor:Dmon}
the image of $\widetilde{\DMon}(Y)$ under the restriction to the invariant part is
$O^{+}(H^{\ast}(X,\BZ)^G)$.
By Corollary~\ref{cor:orbit of primitive vector in M}
we hence conclude that $J(w)$ only depends on the divisibility
of $w$ in the lattice $H^{\ast}(X,\BZ)^G$,
the square $(w,w)$ with respect to the Mukai lattice,
and the type of $w/\mathrm{div}(w)$.
With $w=\varphi(v) = (r, \pi^{\ast} \beta, 2n+r)$ for $v=(r,\beta,n)$ we have (see the beginning of Section~\ref{subsec:orbit of vectors} for the type):
\begin{gather*} \mathrm{div}(w) = \mathrm{gcd}(r, \mathrm{div}(\beta), 2n+r) = 
\mathrm{gcd}(r, \mathrm{div}(\beta), 2n) = \mathrm{div}_{H^{\ast}(X,\BZ)^G}(\pi^{\ast} v) = m \\
(w,w) = (\pi^{\ast} \beta)^2 - 2 r (2n+r) = 2 \Big( \beta \cdot \beta - r(2n+r) \Big) = 2 v \cdot v,
\\
\text{Type of } \frac{w}{\mathrm{div}(w)} \text{ in } H^{\ast}(X,\BZ)^G =
\begin{cases}
\text{ even } & \text{ if } \frac{r}{m},\frac{2n+r}{m} \text{ both even } \\
\text{ odd } & \text{ otherwise }.
\end{cases} \qedhere
\end{gather*}
\end{proof}

The following result was used in the proof above:
\begin{prop} \label{prop:automorphy property}
Let $X \to Y$ be the covering K3 surface of a generic Enriques surface $Y$, and let $\widetilde{\Phi} : D^b(X) \to D^b(X)$ be a $G$-equivariant Fourier-Mukai transform
which preserves the distinguished component $\Stab^{\circ}(X)$. Then
for any $w \in \pi^{\ast} \Lambda_Y$ we have
\[ J( \widetilde{\Phi}^H w ) = J( w ). \]
\end{prop}
\begin{proof}
We sketch the argument following \cite{Toda}: Let 
\[ \Coh(X \times E)_0 \subset \Coh(X \times E) \]
be the subcategory of coherent sheaves supported on fibers of $X \times E \to E$,
and let 
\[ \CD_0 = D^b(\Coh(X \times E)_0). \]
As explained in \cite[Thm.4.20]{Toda} 
there is a distinguished connected component
$\Stab^{\circ}( \CD_0 ) \subset \Stab(\CD_0)$
satisfying
\[ \Stab^{\circ}( \CD_0 ) \cong \Stab^{\circ}(X). \]
Let $\Coh(Q)_0$ be the subcategory of $\Coh(Q)$ consisting of sheaves supported on fibers of $p : Q \to \p^1$, and let $\CD_0^{Q}$ be its derived category.
There is then an equivalence of categories:
\[ \CD_0^Q \cong (\CD_0)_G. \]

Let $\sigma \in \Stab^{\circ}(X)$ be a stability condition and write 
$\sigma' \in \Stab^{\circ}(\CD_0)$ for the corresponding stability condition on $\CD_0$.
Since $Y$ is generic, by \cite[Prop.3.12]{MMS} $\sigma$ is $G$-invariant and hence so is $\sigma'$. Therefore, by the main result of \cite{MMS} the stability condition $\sigma'$ induces a stability condition $\sigma'_G$ on $(\CD_0)_G = \CD_0^Q$.
We let $J_{\sigma}(w) \in \BQ$ be the generalized Donaldson-Thomas invariant counting
$\sigma'_G$-semistable objects $\CE$ in $(\CD_0)_G \cong \CD_0^Q$ satisfying 
$v(\CE) = j_{\ast}(w) \in H^{\ast}(X \times E)$ where $j : X \to X \times E$ is the inclusion of a fiber.
(The existence of moduli stacks of semi-stable objects in $(\CD_0)_G$ and their boundedness can be seen as follows: First the discussion in \cite[p.33]{Toda} implies the existence/boundedness for the $\sigma'$-semistable sheaves in $\CD_0$,
and the $G$-equivariant case follows then by \cite[Sec.3.6]{BO1}).
By an argument as in Theorem~\ref{prop:DT dependence}(i) $J_{\sigma}(w)$ does not depend on $\sigma$ (compare \cite[Thm.4.21]{Toda}), and hence as in \cite[Thm.4.24]{Toda} we have 
\[ J_{\sigma}(w) = J(w). \]

The $G$-equivariant Fourier-Mukai transform $\widetilde{\Phi} : D^b(X) \to D^b(X)$ induces a $G$-equivariant Fourier-Mukai transform
$\widetilde{\Phi}' : D^b(X \times E) \to D^b(X \times E)$.
Indeed, if $\widetilde{\Phi}$ has $G$-equivariant kernel $\CF \in D^b(X \times X)_G$,
we have the $G$-equivariant kernel $\CF \boxtimes \CO_{\Delta_{E}} \in D^b((X \times E)^2)_G$ which defines $\widetilde{\Phi}'$.
Since $\widetilde{\Phi}$ preserves $\Stab^{\circ}(X)$, $\widetilde{\Phi}'$ preserves $\Stab^{\circ}(\CD_0)$. Hence as in \cite[Sec.4.10]{Toda} we get
\[ J(w) = J_{\sigma}(w) = J_{\widetilde{\Phi} \sigma}( \widetilde{\Phi}^H \sigma) = 
J( \widetilde{\Phi}^H \sigma). \]
\end{proof}

\begin{example}
Let $\beta \in H^2(Y,\BZ)$ be primitive with $\beta \cdot \beta = 0$. Then
\[ v_1 = (0,0,1), \quad v_2 = (0,\beta,0) \]
are primitive classes $H^{\ast}(Y,\BZ)$ of square zero,
which lie in different orbits of the derived monodromy group.
Indeed, we have $\mathrm{div}(\pi^{\ast} v_1) = 2$, $t(v_1)$ odd,
and $\mathrm{div}(\pi^{\ast} v_2) = 1$, $t(v_2)$ even.
\end{example}

\begin{rmk} \label{rmk:divisibility 1 odd even}
Assume that $v = (r,\beta,n) \in H^{\ast}(Y,\BZ)$ has divisibility $\mathrm{div}(\pi^{\ast} v) = \mathrm{gcd}(r, \beta, 2n) = 1$. If $v$ is odd,
then $r$ is odd,
and then $v \cdot v = \beta^2 - 2rn - r^2$ is odd.
If $v$ is even, then $r$ is even, and $v^2$ is even.
Hence $\DT^{\mathrm{odd}/\mathrm{even}}_{d,1}$ is non-zero only for $d$ odd/even.
\end{rmk}

\begin{rmk}
We will give representatives of the $\widetilde{\DMon}(Y)$ orbits on 
$H^{\ast}(Y,\BZ)$ and $H^{\ast}(X,\BZ)^G$ which will be useful later on;
we write $v \sim v'$ if two vectors lie in the same orbit, or equivalently, have the same square, divisibility, and type.
For any $d \in \BZ$, let $\alpha_d \in H^2(Y,\BZ)$ be a primitive class of square $2d$.
We first consider representatives of the orbits of {\em primitive} vectors $w \in H^{\ast}(X,\BZ)^G$ which are given as follows:
\begin{itemize}
\item[(i)] $w$ even, primitive, then $w \sim (0,\pi^{\ast} \alpha_d,0)$ for some $d$
\item[(ii)] $w$ odd, primitive, then $w \sim (1,0,n)$ for some $n \in \BZ$.
\end{itemize}
However, instead of (ii) one can also consider the following:
\begin{itemize}
\item[(ii'-a)] $w$ odd, primitive, $4|(w,w)$, then $w \sim (0, \pi^{\ast} \alpha_d, 1)$ for some $d$,
\item[(ii'-b)] $w$ odd, primitive, $4 \nmid (w,w)$, then $w \sim (1,\alpha_d,1)$ where $d$ is odd.
\end{itemize}
Vectors $w$ as in (ii'-a) do not lie in $\pi^{\ast} \Lambda_Y$,
only their even multiplies do. Hence if $v \in H^{\ast}(Y,\BZ)$
has $(t(v),m(v),v \cdot v) = (\text{odd},m,d)$ and $d/m^2$ is even, then $m$ must be even.
We get the following orbits for
$v \in H^{\ast}(Y,\BZ)$ and the corresponding vector $w=\varphi(v) \in H^{\ast}(X,\BZ)$:
\begin{itemize}
\item[(i)] $v$ even, divisibility $m$, then $v \sim (0, m \alpha_d, 0)$ for some $d$
\item[(ii-a)] $v$ odd, divisibility $m$, $v^2/m^2$ odd, then $v \sim (m,m\alpha_d,0)$ and $w \sim (m,m \pi^{\ast} \alpha_d, m)$.
\item[(ii-b)] $v$ odd, divisibility $m$, $v^2/m^2$ even, then $m=2m'$ even and
$v \sim (0, 2m' \alpha_d, m')$ and $w \sim (0,2 m' \pi^{\ast} \alpha_d, 2m')$.
\end{itemize}
\end{rmk}

The last remark gives the following Corollary:
\begin{cor} \label{cor:orbits}
Let $\alpha_d \in H^2(Y,\BZ)$ be a primitive class of square $2d$. The invariants $\DT(v)$ for all $v \in H^{\ast}(Y,\BZ)$ of divisibility $\mathrm{div}(\pi^{\ast} v) \leq m_0$ are determined by the following set of invariants (where $d$ runs over all integers):
\begin{enumerate}
\item[(i)] $\DT(0, m\alpha_d,0)$ for $1 \leq m \leq m_0$,
\item[(ii)] $\DT(0, 2 m' \alpha_d, m')$ for $m' \geq 1$ with $2m' \leq m_0$,
\item[(iii)] $\DT(m, m \alpha_d, 0)$ for $1 \leq m \leq m_0$
\end{enumerate}
\end{cor}

\begin{prop} \label{prop:primitive DT invariants}
The primitive DT invariants are determined as follows:
\begin{align} \forall d \in \BZ \text{ odd: }& \DT^{\mathrm{odd}}_{d,1} = 8 e(\Hilb^{(d+1)/2}(Y)) = 8 \left[ \frac{1}{\eta(\tau)^{12}} \right]_{q^{d/2}} \tag{i} \\
\forall d \in \BZ \text{ even: }&  \DT^{\mathrm{even}}_{d,1} = 0. \tag{ii}
\end{align}
Equivalently, for all $v \in H^{\ast}(Y,\BZ)$ of divisibility $\mathrm{div}(\pi^{\ast} v) = 1$,
\[ \DT(v) = 8 \left[ \frac{1}{\eta(\tau)^{12}} \right]_{q^{v \cdot v/2}}. \]
\end{prop}

\begin{proof}
(i) For $\DT^{\mathrm{odd}}_{d,1}$ to be non-zero, we must have $d$ odd (see Remark~\ref{rmk:divisibility 1 odd even}), let us say $d=2n-1$. 
It hence suffices to prove that
\[ \DT(1,0,-n) = 8 \left[ \frac{1}{\eta(\tau)^{12}} \right]_{q^{n-1/2}}
=
8
\left[ \prod_{r \geq 1} \frac{1}{(1-q^r)^{12}} \right]_{q^{n}}. \]
Any semi-stable sheaf in class $v=(1,0,-n)$ is the pushforward by $\iota_{a} : Y_a \to Q$ of a semi-stable sheaf $E$ on $Y$ with $\ch(E) = (1,0,-n)$, where $a \in E[2]$ is a $2$-torsion point. Hence $E$ is stable torsion free sheaf on $Y$ with $c_1(E) = 0 \in H^2(Y,\BZ)$. There are two possibilities: Either $E = I_Z$ for some length $n$ subscheme $Z \subset Y$, or $E = I_Z \otimes \omega_Y$.
We see that the moduli space of semi-stable sheaves on $Q$ in class $\iota_{\ast} (1,0,-n)$ is isomorphic to $8$ copies of the Hilbert schemes $\Hilb^n(Y)$. 
If semi-stability equals stability, 
the generalized DT invariant is by definition just given by the Behrend-function weighted Euler characteristic of the scheme.
Moreover, since the $\Hilb^n(Y)$ is even-dimensional and smooth, the Behrend function is $+1$ everywhere \cite{JS}. Hence we obtain $\DT_{\omega}(v) = 8 e(\Hilb^n(Y))$.
The desired equality follows therefore by G\"ottsche's formula \cite{Goettsche}.

(ii) Let us denote $\DT_{n,\beta} = \DT(0,\beta,n)$.
By Bridgeland \cite{Br1} and Toda \cite{Toda2, Toda3}
the generating series of PT invariants can be expanded as
\begin{equation} \label{TE}
\sum_{n,\beta \in H_2(Y,\BZ)_{\geq 0}} \PT_{n,\beta} (-p)^n q^{\beta}  
=
\exp\left(-\sum_{n > 0, \beta>0} n \DT_{n,\beta} p^n q^{\beta} \right)
\cdot \sum_{n, \beta} \mathsf{L}_{n,\beta} (-p)^n q^{\beta}
\end{equation}
where for every $\beta$ the invariants $\mathsf{L}_{n,\beta} \in \BZ$ satisfy
$\mathsf{L}_{n,\beta} = \mathsf{L}_{-n,\beta}$ for all $n$,
and $\mathsf{L}_{n,\beta} = 0$ for all $n \gg 0$.
Moreover, for each $\beta > 0$ we have $\DT_{n,\beta} = \DT_{n+\mathrm{div}(\beta), \beta}$ and $\DT_{n,\beta} = \DT_{-n,\beta}$.

Assume $\beta$ is primitive. Then from this conditions we get
\[
\sum_{n > 0, \beta>0} n \DT_{n,\beta} p^n = 
\DT_{1,\beta} \frac{p}{(1-p)^2}.
\]
Hence the $q^{\beta}$-coefficient of the logarithm of \eqref{TE} reads
\begin{equation} \label{dfsd-0f9i0-}
-\DT_{1,\beta} \frac{p}{(1-p)^2} + \sum_{n} \mathsf{L}_{n,\beta} (-p)^n
\end{equation}
All genus zero Gromov-Witten invariants of the Enriques $Y$ vanish for dimension reasons. Hence the same holds by Proposition~\ref{prop:GW QvsY} for the Enriques threefold $Q$.
By the Gromov-Witten correspondence (Theorem~\ref{thm:GW correspondence})
we conclude that
\eqref{dfsd-0f9i0-} does not have a $z^{-2}$ under the variable change $p=e^{z}$, and hence no pole at $p=1$. So $\DT_{1,\beta}=0$.
\end{proof}

\begin{lemma} \label{lemma:dt for div 2}
For all even $d$, we have
$\DT^{\mathrm{odd}}_{4d,2} = -\DT^{\mathrm{even}}_{4d,2}$.
\end{lemma}
\begin{proof}
Let $d$ even and let $\beta=2s+df \in H_2(Y,\BZ)$ which is of divisibility $2$.
We then have
\[ \DT^{\mathrm{even}}_{4d,2} = \DT(0,\beta,0), \quad
\DT^{\mathrm{odd}}_{4d,2} = \DT(0,\beta,1). \]
We argue similarly as in the proof of Proposition~\ref{prop:primitive DT invariants}(ii). By $\DT(0,\beta,n) = \DT(0,\beta,n+2)$ we find
\[
\sum_{n > 0} n \DT(0,\beta,n) p^n
=
\DT(0,\beta,1) \frac{p (1+p^2)}{(1-p^2)} + \DT(0,\beta,0) \frac{2 p^2}{(1-p^2)^2}. \]
By the GW/PT correspondence this series
does not have any pole at $p=1$, which implies the claim.
Indeed, if $\DT(0,\beta,0) = -\DT(0,\beta,1)$ the series becomes
$\DT(0,\beta,1) p/(1+p)^2$. 
\end{proof}
\begin{rmk}
The Gopakumar-Vafa finiteness conjecture implies that
if all positive degree genus 0 Gromov-Witten invariants of a Calabi-Yau threefold
vanish, then also the $1$-dimensional generalized Donaldson-Thomas invariants vanish
for all curve classes and Euler characteristics.
The finiteness conjecture was recently proven in \cite{DIW},
so we would get the vanishing of $\DT_{d,m}^{\text{even}}$ for free.
However the proof in \cite{DIW} uses methods from symplectic geometry,
so we prefer here to give a direct algebraic argument.
The vanishing of $\DT_{d,m}^{\text{even}}$ for all $d,m$ follows later from
Corollary~\ref{cor:consequence of main thm}.
\end{rmk}

\begin{lemma} \label{DT vanishing}
Let $\beta \in H_2(Y,\BZ)$ be non-zero.
The invariant $\DT(r,\beta,n)$ is non-zero only if
\[ \beta^2 + 4 (\beta \cdot \omega)^2 \geq 2 rn + r^2. \]
\end{lemma}
\begin{proof}
If $F \in \Coh(Y)$ is $\omega$-Gieseker semi-stable with $\ch(F) = \iota_{\ast}(r,\beta,n)$, then $\pi_Q^{\ast}(F)$ is $\pi_Q^{\ast}(\omega)$-semistable with class
$\ch( \pi_Q^{\ast} F) = \iota_{\ast}(r, \beta, 2n)$, see e.g. \cite[Lemma 2.8]{BO1};
Here $\pi_Q : X \times E \to Q$ denotes the projection.
Hence by \cite[Lemma 2.5]{Toda} we have
\begin{equation} (\pi^{\ast} \beta)^2 + 2 (\pi^{\ast}(\beta) \cdot \pi^{\ast} \omega)^2 \geq 2 r (2n+r). \label{efsdf} \end{equation}
We see that if the inequality \eqref{efsdf} is violated, then the moduli space of semi-stable sheaves of Chern character $\iota_{\ast} (r,\beta,n)$ is empty, and so $\DT(r,\beta,n)=0$.
\end{proof}

\subsection{Toda's formula}
Define the generating series in fiber classes
\[ \PT(Q) = \sum_{\beta \in H_2(Y,\BZ)} \sum_{n \in \BZ} \PT_{n,\beta} (-p)^n q^{\beta}. \]

\begin{thm}[Toda's formula] \label{thm:Todas formula}
\[ \PT(Q) = \prod_{\substack{r \geq 0 \\ \beta > 0 \\ n \geq 0 }} \exp\left( (-1)^{r-1} (n+r) \DT(r,\beta,n) q^{\beta} p^n \right) 
\times
\prod_{\substack{r>0 \\ \beta>0 \\ n>0}} \exp\left( (-1)^{r-1} (n+r) \DT(r,\beta,n) q^{\beta} p^{-n} \right)
\]
\end{thm}
\begin{proof}
We first remark that the formula is well-defined by Lemma~\ref{DT vanishing}.
To prove the formula, one argues exactly as in \cite{Toda}
with the only difference that one works on $X \times E$, but performs every step $G$-equivariantly. For example, given a semi-orthogonal decomposition, one considers the induced semi-orthongal decomposition on the $G$-equivariant category, etc. We refer to Section~\ref{subsec:overview} for references for the results that are needed. This is mostly straightforward, and hence we only highlight the main differences here:
\begin{enumerate}
\item The geometry $X \times \p^1$ in \cite{Toda} is replaced by the $G$-equivariant geometry of $X \times E$.
\item The pairing $\chi : \Gamma \times \Gamma_0 \to \BZ$ in \cite[2.7]{Toda} has to be replaced by the $G$-equivariant pairing, which gives
\[ \chi(\ch(F_1), \ch(F_2)) = \sum_{i} (-1)^i \dim \Ext^i_{X \times E}(F_1, F_2)^G, \]
or equivalently, which computes the Euler pairing on the quotient $Q = (X \times E)/\BZ_2$. In particular, 
the wall-crossing factor $\chi(\CO_{X \times \p^1}, \iota_{\ast}(r,\beta,n)) = (n+2r)$ which appears in the main formula in \cite[Thm.1.1]{Toda} is replaced in our case by the pairing
\[ \chi_{Q}( \ch(\CO_Q), \iota_{\ast} (r,\beta,n)) = \int_{Y} (r,\beta,n) \td_Y = n+r. \]
\item The sheaf $p^{\ast} \CO_{\p^1}(r)$ in \cite[Defn.2.12]{Toda} has to be replaced by the set of $G$-equivariant sheaves on $X \times E$ which are pullbacks of degree $r$ line bundles of $E$, or in our words, by the set $(p^{\ast} \Pic^r(E))_G$.

Note that there are precisely 8 of them for each $r$: Indeed, for any $a \in E$, the line bundle $p^{\ast} \CO_E(a) = \CO_{X \times E}(X \times a)$ is $G$-invariant if and only if $a$ is $2$-torsion. Moreover, $\CO(X \times a)$ is simple,
so for any $2$-torsion point $a \in E[2]$ the line bundle $\CO(X \times a)$ admits precisely $2$ different $G$-linearizations, by \cite[Lemma 1]{Ploog}.
Hence there are $4 \cdot 2$ elements in $(p^{\ast} \Pic^1(E))_G$ which shows the claim for $r=1$.\footnote{The $8$ elements in $(p^{\ast} \Pic^0(E))_G$
correspond to the $8$ line bundles on $Q$ given by
$\CO_Q, \CO_Q(Y_a - Y_b)$ for $a,b \in E[2]$ with $a \neq b$,
and $\CO_Q(Y_{a_1} - Y_{a_2} + Y_{a_3} - Y_{a_4})$ for $E[2] = \{ a_1, a_2, a_3, a_4 \}$.}
The case of general $r$ follows since $(p^{\ast} \Pic^r(E))_G$
is non-canonically isomorphic to $(p^{\ast} \Pic^1(E))_G$.

As a consequence, $\widehat{M}_{\omega,\theta}(1,r,\beta)$ in \cite[Prop.3.17]{Toda} consists now of $8$ isolated reduced points if $\beta = 0$, and is empty for $\beta \neq 0$.
\item Right after \cite[Defn.4.2]{Toda}, the moduli space of stable pairs $(F,s)$ in class $(n,\beta)$ is identified with the moduli space of certain $2$-term complexes $\CL \to F$  in the derived category of Chern character $(1,0,-\beta,-n)$. In \cite{Toda} the condition on the Chern character implies $c_1(\CL) = 0$, so $\CL \cong \CO$. In the $G$-equivariant case, we work modulo torsion,
and there are precisely $8$ torsion line bundles, so
the stable pair invariant is $\frac{1}{8}$ times the invariant counting stable objects in the derived category. This cancels the factor of 8 which appears in point (3), so overall there is no change in the final formula.
\item The paper \cite{Toda} uses DT invariants, which are defined by unweighted (virtual) Euler numbers, that is, which do not carry any weight by the Behrend function. This simplification was made for technical reasons and the technical gaps were later filled in \cite{Toda3}. Hence the Behrend weight can be added in \cite{Toda} and we do the same here. This produces a sign $(-1)^{n+r-1}$ in the wall-crossing formula (the $(-1)^n$ factor of which is absorbed by the variable $p$).
\qedhere
\end{enumerate}
\end{proof}

\subsection{Consequences of Toda's formula}
\label{subsec:Consequences of Toda}
Define the invariants $\dt(v) \in \BQ$ for $v \in H^{\ast}(Y,\BZ)$ inductively by the equation
\begin{equation} \DT(v) = \sum_{\substack{k|v \\ k \geq 1 \text{ odd}}} \frac{1}{k^2} \dt\left( \frac{v}{k} \right) \label{DTdt} \end{equation}
Note here 'odd $k|v$' means $k$ is an odd integer $\geq 1$ such that $v/k \in H^{\ast}(Y,\BZ)$ or equivalently, $k|\gcd(r,\beta,n)$. Since $k$ is odd,
this is equivalent to $k| m(v) = \gcd(r,\beta,2n)$.
By Theorem~\ref{prop:DT dependence}
the invariant $\dt(v)$ again only depends upon $v$ through the square $d=v \cdot v$, the divisibility $m=\mathrm{gcd}(r,\beta,2n)$ and the type $t(v)$. Hence we again write
\[ \dt^{t}_{m,d} := \dt(v). \]

By taking the log in Toda's formula and inserting
\eqref{DTdt}, and interchanging sums, we obtain:
\[
\log \PT(Q)
=
\sum_{\beta>0} 
\sum_{\substack{k \geq 1 \\ k \text{ odd}}} \frac{1}{k} q^{k \beta}
\left[ 
\sum_{r \geq 0} \sum_{n \geq 0} (-1)^{r-1} (n+r)  \dt(r,\beta,n) p^{kn}
+
\sum_{r > 0 } \sum_{n > 0}
(-1)^{r-1}  (n+r) \dt(r,\beta,n) 
p^{-k n} \right]
\]
Define the series:
\begin{align} 
f_{\beta}^{\PT}(p) & =
\sum_{r \geq 0} \sum_{n \geq 0} (-1)^{r-1} (n+r)  \dt(r,\beta,n) p^{n}
+ \sum_{r > 0 } \sum_{n > 0} (-1)^{r-1} (n+r) \dt(r,\beta,n) p^{-n} \notag \\
& = \sum_{n>0} \sum_{r>0} (n+r) (-1)^{r-1} \dt(r,\beta,n) (p^n+p^{-n}) 
- \sum_{n>0} n \dt(0,\beta,n) p^n 
+ \sum_{r>0} (-1)^{r-1} r \dt(r,\beta,0) \label{230f0sdf0}
\end{align}
We find that
\[
\log \PT(Q) =
\sum_{\beta>0} 
\sum_{\substack{k \geq 1 \\ k \text{ odd}}} \frac{1}{k} q^{k \beta}
f_{\beta}^{\PT}(p^k).
\]

On the Gromov-Witten side, define invariants $n_{g,\beta}^{Q}$ inductively by
\[ N_{g,\beta}^{Q} = \sum_{\substack{k|\beta \\ k \geq 1 \text{ odd}}}
k^{2g-3} n_{g,\beta/k}^{Q}. \]
Define also
\[
f_{\beta}^{\GW}(z) = \sum_{g} n_{g,\beta}^Q (-1)^{g-1} z^{2g-2}
\]
By the Gromov-Witten/Pairs correspondence we have under the variable change $p=e^{z}$
\begin{align*}
\log \PT(Q) & = 
\sum_{g,\beta} N^Q_{g,\beta} (-1)^{g-1} z^{2g-2} q^{\beta} \\
& = \sum_{\beta>0} \sum_{\substack{k \geq 1 \\ k \text{ odd}}} \frac{1}{k} q^{k \beta} \sum_{g} n^Q_{g,\beta} (-1)^{g-1} (kz)^{2g-2} \\
& = \sum_{\beta>0} \sum_{\substack{k \geq 1 \\ k \text{ odd}}} \frac{1}{k} q^{k \beta}
f_{\beta}^{\GW}(kz).
\end{align*}

We hence find that under the variable change $p=e^{z}$ we have:
\[ f_{\beta}^{\GW}(z) = f_{\beta}^{\PT}(p). \]

\begin{prop} \label{prop:GW dependence}
The Gromov-Witten invariant
$N^{Q}_{g,\beta}$ depends on $\beta$ only through $\beta^2$ and the divisibility of $\beta$.
\end{prop}
\begin{proof}
It suffices to show that $n_{g,\beta}^{Q}$ only depends on $\beta^2$ and $\mathrm{div}(\beta)$. For $v=(r,\beta,n) \in H^{\ast}(Y,\BZ)$, the square of $v$, the divisibility of $\pi^{\ast}(v)$, and the type of $v$
only depends on $r,n,\mathrm{div}(\beta), \beta \cdot \beta$.
Hence by Theorem~\ref{prop:DT dependence}, the series $f_{\beta}^{\PT}(p)$ only depends upon $\beta$ through $\beta^2$ and $\mathrm{div}(\beta)$. The claim hence follows from $f_{\beta}^{\GW}(z) = f_{\beta}^{\PT}(p)$.
\end{proof}

Recall the coefficients $\omega_g(n)$ defined for all $g,n$ by
\[ \omega_g(n) = (-1)^{g-1} \left[
\frac{\Theta(z,2\tau)^2}{\Theta(z,\tau)^2} \frac{\eta(2 \tau)^{8}}{\eta(\tau)^{16}}
\right]_{z^{2g-2} q^n}
\]

\begin{prop} \label{prop:GW vs DT}
The following two statements are equivalent:
\begin{enumerate}
\item[(i)] For all $v \in H^{\ast}(Y,\BZ)$ of divisibility $\mathrm{div}(\pi^{\ast} v) \leq m$ we have
\[ \dt(v) = 8 [ \eta^{-12}(\tau) ]_{q^{v \cdot v/2}}. \]
\item[(ii)] For all effective $\beta \in H_2(Y,\BZ)$ of divisibility $\mathrm{div}(\beta) \leq m$ we have
\[
n_{g,\beta}^{Q} = 8 \omega_{g}\left( \frac{\beta^2}{2} \right).
\]
\end{enumerate}
\end{prop}

We start with the following basic computation: 
\begin{lemma} \label{lemma:computation}
For any $\beta \in H_2(Y,\BZ)$ we have
\begin{multline*}
\left[ \frac{\Theta(z,2\tau)^2}{\Theta(z,\tau)^2} \frac{\eta(2 \tau)^{8}}{\eta(\tau)^{16}} \right]_{q^{\beta^2/2}} \\
 = \sum_{\substack{n>0\\r>0 \text{ odd}}} (n+r) [ \eta^{-12}(\tau) ]_{q^{\beta^2/2 - rn-r^2/2}} (p^n + p^{-n})
+ \sum_{\substack{r > 0\\ r \text{ odd}}} r [\eta^{-12}(\tau)]_{q^{\beta^2/2 - r^2/2}} \end{multline*}
\end{lemma}
\begin{proof}
By Proposition~\ref{prop:theta identity} we have:
\begin{align*}
\eta^{-12}(\tau) \frac{ \Theta(z,2 \tau)^2 \eta(2 \tau)^8 }{\Theta(z,\tau)^2 \eta^{4}(\tau) } 
= & \eta^{-12}(\tau) \sum_{\substack{r > 0\\ r \text{ odd}}}
\left( \sum_{n>0} (n+r) (p^n+p^{-n}) q^{rn + r^2/2} + r q^{r^2/2} \right) \\
= & \sum_{\substack{d \geq 0, n>0\\r>0 \text{ odd}}} 
(n+r) [ \eta^{-12}(\tau) ]_{q^{d - rn-r^2/2}} (p^n + p^{-n}) q^{d-rn-r^2/2} q^{rn+r^2/2} \\
& + \sum_{d} \sum_{\substack{r > 0\\ r \text{ odd}}} r [\eta^{-12}(\tau)]_{q^{d - r^2/2}} q^{d-r^2/2} q^{r^2/2}.
\end{align*}
The claim follows by taking the $q^{\beta^2/2}$ coefficient.
\end{proof}

\begin{proof}[Proof of Proposition~\ref{prop:GW vs DT}]
Assume (i) first. 
Since the $q$-expansion of $\eta^{-12}(\tau)$ has only half-integral exponents, it follows that $\dt(v)=0$ for all $v$ of divisibility $\leq m$ with $v \cdot v$ even.
Let $\beta \in H_2(Y,\BZ)$ be of divisibility $\leq m$.
Since on the right hand side of \eqref{230f0sdf0} there are only DT invariants of classes $v=(r,\beta,n)$ of divisibility $\leq m$, we get
\begin{align*}
f_{\beta}^{\GW} = 
f_{\beta}^{\PT} & = 8 \sum_{\substack{n>0\\r>0 \text{ odd}}} (n+r) [ \eta^{-12}(\tau) ]_{\beta^2/2 - rn-r^2/2} (p^n + p^{-n})
+ 8 \sum_{\substack{r > 0\\ r \text{ odd}}} r [\eta^{-12}(\tau)]_{\beta^2/2 - r^2/2}  \\
& =
8 \left[ \frac{\Theta(z,2\tau)^2}{\Theta(z,\tau)^2} \frac{\eta(2 \tau)^{8}}{\eta(\tau)^{16}} \right]_{q^{\beta^2/2}}
\end{align*}
where the second line is 
Lemma~\ref{lemma:computation}.
Taking the $z^{2g-2}$-coefficient yields:
$n_{g,\beta}^{Q} = 8 \omega_{g}\left( \beta^2/2 \right)$.

Hence (i) implies (ii).
Conversely, we need to argue that the condition
\begin{multline} \label{condition}
\forall \beta \in H_2(Y,\BZ) \text{ with } \mathrm{div}(\beta) \leq m : \quad 
 \left[ 8 \frac{ \Theta(z,2 \tau)^2 \eta(2 \tau)^8 }{\Theta(z,\tau)^2 \eta^{16}(\tau) } \right]_{q^{\beta^2/2}} \\ =
\sum_{n>0} \sum_{r>0} (-1)^{r-1} (n+r) \dt(r,\beta,n) (p^n+p^{-n}) 
- \sum_{n>0} n \dt(0,\beta,n) p^n 
+ \sum_{r>0} (-1)^{r-1} r \dt(r,\beta,0)
\end{multline}
has at most one solution for $\dt(v)$ where $\mathrm{div}(\pi^{\ast} v) \leq m$.

By Proposition~\ref{prop:theta identity}
the $q$-coefficients of
$\Theta(z,2 \tau)^2 \eta(2 \tau)^8/\Theta(z,\tau)^2 \eta^{16}(\tau)$
are Laurent {\em polynomials} in $p$ invariant under $p \mapsto 1/p$.
Hence condition \eqref{condition} and Lemma~\ref{DT vanishing} imply
\[ \dt(0,\beta,n) = 0 \text{ for all } \mathrm{div}(\beta) \leq m \text{ and } n > 0. \]
Since $\dt(0,\beta,n) = \dt(0,\beta, n+\mathrm{div}(\beta))$ we get that
\begin{equation} \dt(0,\beta,n) = 0 \text{ for all } \mathrm{div}(\beta) \leq m, \quad n \in \BZ. \tag{$\dagger$} \label{dagger} \end{equation}

Let $v = (r,\beta,n) \in H^{\ast}(Y,\BZ)$ of divisibility $m'\leq m$.
Let $\alpha_d \in H^2(Y,\BZ)$ denote any primitive class of square $2d$.
We check the possible cases of $v$ according to Corollary~\ref{cor:orbits}:

\begin{itemize}
\item If $v$ is of even type, then $\dt(v) = \dt(0,m' \alpha_d,0)$ for some $d$.
By \eqref{dagger} this gives $\dt(v)=0$ which is as claimed by (i).
\item If $v$ is odd and $v^2/(m')^2$ is even,
then $m'=2m''$ is even and $\dt(v) = \dt(0,2m'' \alpha_d, m'')$ for some $d$.
Hence $\dt(v)=0$ as claimed by (i).
\item If $v$ is odd and $v^2/(m')^2$ odd,
then $\dt=\dt(m',m' \alpha_d,0)$ for some $d$.
Assume that $\dt(v')$ is already determined inductively if $v'$ is of divisibility $m(v') < m'$, or if $v'$ has divisibility $m(v') = m'$ but $(v')^2 < v^2$.
Then consider the $p^0$-coefficient of $f^{\PT}_{m' \alpha_d}(p)$:
\[
[ f^{\PT}_{m' \alpha_d}(p) ]_{p^0} = \sum_{r>0} (-1)^{r-1} r \dt(r,m' \alpha_d,0).
\]
The left hand side is determined from (ii).
If $r<m'$ then $(r,m' \alpha_d,0)$ has divisibility $<m'$, so $\dt(r,m' \alpha_d,0)$ is known. If $r>m'$, then $(r,m' \alpha_d,0)^2 = (m')^2 \alpha_d^2 - r^2 < v^2$, so is also known. Hence $\dt(m', m' \alpha_d,0)$ is the only undetermined term in the above equation, hence is a posteriori also determined.
\end{itemize}
Thus given (ii) for divisibility $\leq m$,
there is a unique way to fix $\dt(v)$ for divisibility $\leq m$,
and in the first part we have seen that this must be $\dt(v) = [ \eta^{-12}(\tau) ]_{q^{v^2/2}}$.
\end{proof}

\section{Putting everything together} \label{sec:putting everything together}
In this section
we conclude the Klemm-Mari\~{n}o formula from what we have done before.
\subsection{Statement of result}
Consider the generating series of Gromov-Witten invariants:
\[
F^{\GW}_{\ell} = \sum_{g \geq 1} (-1)^{g-1} z^{2g-2} F^{\GW}_{g,\ell},
\quad \quad 
F^{\GW}_{g,\ell}(\zeta,q)
:=
\sum_{d \geq 0} \sum_{\alpha \in E_8(-1)} N_{g,\ell s+df+\alpha}^{Q} \zeta^{\alpha} q^{d}
\]
where for $\ell=0$ we assume $g>1$ or $d f + \alpha > 0$. 
In the language of the last section, we have
\begin{equation} \label{d0fsdf}
F^{\GW}_{\ell} = \sum_{d \geq 0} \sum_{\alpha \in E_8(-1)} q^d \zeta^{\alpha}
\sum_{\substack{ k|(\ell,d,\alpha) \\ k \text{ odd}}}
\frac{1}{k} f^{\GW}_{\frac{\ell}{k} s + \frac{d}{k} f + \frac{\alpha}{k}}(kz)
\end{equation}

Consider also the analogues expected from the Klemm-Mari\~{n}o formula:
\[
F^{\KM}_{\ell} = \sum_{g \geq 1} (-1)^{g-1} z^{2g-2} F^{\KM}_{g,\ell},
\quad \quad
F^{\KM}_{g,\ell}
:=
\sum_{\substack{d \geq 0 \\ \alpha \in E_8(-1)}} \sum_{\substack{k|(\ell,d,\alpha) \\ k \text{ odd}}} 8 k^{2g-3} \omega_g\left( \frac{ 2 \ell d + \alpha^2 }{2 k^2 } \right) q^d \zeta^{\alpha}
\]
where 
\[ \sum_{n} \omega_g(n) q^n = (-1)^{g-1} \left[ \frac{\Theta(z, 2 \tau)^2}{\Theta(z,\tau)^2} \frac{\eta(2 \tau)^8}{\eta(\tau)^{16}} \right]_{z^{2g-2}}. \]

For the cases $\ell>0$ we have equivalently:
\begin{align*}
F^{\KM}_{g,1}(\zeta,q) 
& = 8 \Theta_{E_8}(\zeta,q) \sum_{n} \omega_g(n) q^n  \\
F^{\KM}_{g,\ell} & = F^{\KM}_{g,1}|_{2g-2, \frac{1}{2} Q_{E_8}}V_\ell
\end{align*}
where the Hecke operator was defined in \eqref{Hecke:main text}. If we also let
\begin{equation} f^{\KM}_{\beta} := \left[ \frac{\Theta(z,2\tau)^2}{\Theta(z,\tau)^2} \frac{\eta(2 \tau)^{8}}{\eta(\tau)^{16}} \right]_{q^{\beta^2/2}} 
\label{f beta KM} \end{equation}
then the analogue of \eqref{d0fsdf} is:
\[
F^{\KM}_{\ell} = \sum_{d \geq 0} \sum_{\alpha \in E_8(-1)} q^d \zeta^{\alpha}
\sum_{\substack{ k|(\ell,d,\alpha) \\ k \text{ odd}}}
\frac{1}{k} f^{\KM}_{\frac{\ell}{k} s + \frac{d}{k} f + \frac{\alpha}{k}}(kz)
\]

Our goal in this section is to prove that all of these above are equal:
\begin{thm} \label{thm: main thm restated}
The difference
\[
\widehat{F}_{g,\ell} := F^{\GW}_{g,\ell}(\zeta,q) - F^{\KM}_{g,\ell}(\zeta,q).
\]
vanishes for all $g,\ell$. In particular,
\[ F^{\KM}_{\ell} = F^{\GW}_{\ell}, \quad F^{\KM}_{g,\ell} = F^{\GW}_{g,\ell}, \quad f^{\KM}_{\beta} = f^{\GW}_{\beta} = f^{\PT}_{\beta}. \]
\end{thm}

Using Proposition~\ref{prop:GW vs DT} we obtain the following corollary
(which implies Theorem~\ref{thm:main theorem} due to Proposition~\ref{prop:GW QvsY}):
\begin{cor}
\label{cor:consequence of main thm}
For any $\beta \in H_2(Y,\BZ)$, $g \geq 0$ and $v \in H^{\ast}(Y,\BZ)$ we have
\begin{alignat*}{2}
N^Q_{g,\beta} & = \sum_{\substack{k|\beta \\ k \geq 1 \text{ odd}}} 8 k^{2g-3} \omega_g\left( \frac{ \beta^2 }{2 k^2 } \right),
& 
\quad \quad \quad \quad \quad n^Q_{g,\beta} & = 8 \omega_g\left( \beta^2/2 \right) \\
\DT(v) & = 8 \sum_{\substack{k|v \\ k \geq 1 \text{ odd}}} \frac{1}{k^2} \left[ \frac{1}{\eta(\tau)^{12}} \right]_{q^{\frac{v/k \cdot v/k}{2}}},
& 
\dt(v) & = 8 \left[ \frac{1}{\eta^{12}(\tau)} \right]_{q^{v \cdot v/2}}
\end{alignat*}
\end{cor}

\subsection{Similarities between $F^{\GW}$ and $F^{\KM}$}
Before starting the proof of Theorem~\ref{thm: main thm restated}
we recall some basic properties of our generating series from earlier sections:
\label{subsec:similarities}
\begin{prop} \label{prop:Fhat HAE}
For $\ell>0$, $F^{\KM}_{g,\ell}(\zeta,q)$ and $F^{\GW}_{g,\ell}(\zeta,q)$
are elements of weight $2g-2$
in the space of quasi-Jacobi forms
$\frac{1}{\Delta(q)^{\ell}} \QJac_{\frac{1}{2} \ell Q_{E_8}}(\Gamma_0(2))$
satisfying
\[
\frac{d}{dG_2} F^{\ast}_{g,\ell} = -\ell F^{\ast}_{g-1,\ell}, \quad 
\xi_{\lambda} F^{\ast}_{g,\ell} = 0 \text{ for all } \lambda \in E_8(-1), \quad 
\text{ for }\ast \in \{ \GW, \KM \}.
\]
\end{prop}
\begin{proof}
For $F^{\GW}_{g,\ell}$ this follows
from the example discussed in Section~\ref{subsubsec:Hodge integrals},
and the comparision of GW invariants of $Q$ and $Y$
in Proposition~\ref{prop:GW QvsY}.

On the KM-side, we consider first
\[ F^{\KM}_{g,1} = 8 \Theta_{E_8}(\zeta,q) (-1)^{g-1} \left[ \frac{\Theta(z, 2 \tau)^2}{\Theta(z,\tau)^2} \frac{\eta(2 \tau)^8}{\eta(\tau)^{16}} \right]_{z^{2g-2}}. \]
The right side is the weight $4$ Jacobi form $\Theta_{E_8}(\zeta,q)$
multiplied by a quasi-modular form in $\frac{1}{\Delta(q)} \QMod(\Gamma_0(2))$ of weight $2g-2-4$, see \eqref{example modualr forms for gamma02} and \eqref{theta gk expansion}.
Hence we have
\[
F^{\KM}_{g,1} \in \frac{1}{\Delta(q)} \QJac_{\frac{1}{2} Q_{E_8}}(\Gamma_0(2)),
\]
Moreover, $F^{\KM}_{g,1}$ is of weight $2g-2$ and satisfies
$\xi_{\lambda} F^{\KM}_{g,1} = 0$. 
Since by \eqref{theta gk expansion} we have
\[
\frac{\Theta(z, 2 \tau)^2}{\Theta(z,\tau)^2}
= \exp\left( G_2(\tau) z^2 + \text{ terms involving only modular forms} \right),
\]
we further find
\[ \frac{d}{dG_2} F^{\KM}_{g,1} = - F^{\KM}_{g-1,1}. \]
The general case follows by applying the Hecke operators $|V_\ell$ and 
Proposition~\ref{prop:Hecke operator}.
\end{proof}
\begin{prop} \label{prop:dependence of Fgr}
For $\ast \in \{ \KM, \GW \}$, 
the coefficient of $q^d \zeta^{\alpha}$ of $F_{g,\ell}^{\ast}$ only depends on $g$, the square $2 \ell d + \alpha^2$ and the divisibility $\mathrm{gcd}(\ell,d, \mathrm{div}(\alpha))$. In other words,
there exists a function $a^{\ast}(g,D,m)$ such that that
$\big[ F_{g,\ell}^{\ast} \big]_{q^d \zeta^{\alpha}} = a^{\ast}(g,2 \ell d + \alpha^2,\mathrm{gcd}(\ell,d, \mathrm{div}(\alpha)))$
for all $g,\ell,d,\alpha$.
\end{prop}
\begin{proof}
For the Gromov-Witten invariants, this was proven in Proposition~\ref{prop:GW dependence}.
The KM-side follows by inspection.
\end{proof}

Consider the difference of the invariant $\dt(r,\beta,d)$ and the expected answer:
\[ \widehat{\dt}(r,\beta,d) 
:= 
\dt(r,\beta,d) - 8 [ \eta^{-12}(\tau) ]_{\beta^2/2-r^2/2 - rn} \]
Set also
\[ \widehat{f}_{\beta} := 
 f_{\beta}^{\GW}(z) - f_{\beta}^{\KM}(z)
 =
f_{\beta}^{\PT}(p) - f_{\beta}^{\KM}(p) \]
where we identify here a rational function in $p$
again with the Taylor expansion at $z=0$ under the variable change $p=e^{z}$.

\begin{lemma} \label{fhat beta} For any $\beta$ we have
\[ \widehat{f}_{\beta} 
= \sum_{n>0} \sum_{r>0} (-1)^{r-1} (n+r)  \widehat{\dt}(r,\beta,n) (p^n+p^{-n}) 
- \sum_{n>0} n \widehat{\dt}(0,\beta,n) p^n 
+ \sum_{r>0} (-1)^{r-1} r \widehat{\dt}(r,\beta,0).
\]
\end{lemma}
\begin{proof}
Recall the description from \eqref{230f0sdf0}:
\[
f_{\beta}^{\PT}(p) 
 = \sum_{n>0} \sum_{r>0} (-1)^{r-1} (n+r) \dt(r,\beta,n) (p^n+p^{-n}) 
- \sum_{n>0} n \dt(0,\beta,n) p^n 
+ \sum_{r>0} (-1)^{r-1} r \dt(r,\beta,0)
\]
By the definition \eqref{f beta KM} and Lemma~\ref{lemma:computation}
we also have
\[
f^{\KM}_{\beta}
= \sum_{\substack{n>0\\r>0 \text{ odd}}} (n+r) [ \eta^{-12}(\tau) ]_{q^{\beta^2/2 - rn-r^2/2}} (p^n + p^{-n})
+ \sum_{\substack{r > 0\\ r \text{ odd}}} r [\eta^{-12}(\tau)]_{q^{\beta^2/2 - r^2/2}}.
\]
Subtracting both terms yields the result (observe that again, $\eta^{-12}(\tau)$ is a power series with half-integer exponents only).
\end{proof}

\subsection{Proof of Theorem~\ref{thm: main thm restated}}
\label{proof:main thm}
We argue by induction on $\ell$.

\vspace{4pt}
\noindent
\textbf{Case $\ell=0$.}
By Maulik and Pandharipande's computation of $N_{g,df}$ in \cite{MP}
which was recalled in Section~\ref{subsubsec:fiber classes} we have
\[
F^{\GW}_{g,0} = \delta_{g=1} \sum_{d>0} 8 \sum_{\text{odd } k|d} \frac{1}{k} q^d.
\]
This equals $F^{\KM}_{g,0}$ by observing that
\[
\omega_g(0) = (-1)^{g-1} \left[
\frac{\Theta(z,2\tau)^2}{\Theta(z,\tau)^2} \frac{\eta(2 \tau)^{8}}{\eta(\tau)^{16}} \Bigg|_{q=0}
\right]_{z^{2g-2}} = (-1)^{g-1} [ 1 ]_{z^{2g-2}} = \delta_{g=1}.
\]

\vspace{4pt}
\noindent
\textbf{Case $\ell=1$.}
By Proposition~\ref{prop:primitive DT invariants} we know all fiber DT invariants of $Q$ for primitive classes, and by Proposition \ref{prop:GW vs DT} 
this implies that for all primitive curve classes $\beta$ we have
\[
N_{g,\beta}^{Q} = n_{g,\beta}^{Q} = 8 \omega_{g}\left( \frac{\beta^2}{2} \right).
\]
In particular, this applies to $\beta = s + df+ \alpha$ and immediately gives $\widehat{F}_{g,1}=0$.

\vspace{4pt}
\noindent
\textbf{Case $\ell=2$.}
By Proposition~\ref{prop:dependence of Fgr} and the cases $\ell \in \{ 0,1 \}$ we have for all $g$:
\begin{itemize}
\item The $q^d \zeta^{\alpha}$-coefficient of $\widehat{F}_{g,2}$ vanishes
whenever $2s+df+\alpha$ is primitive,
\item The coefficient $q^0$ of $\widehat{F}_{g,2}$ vanish.
\end{itemize}
Hence $\widehat{F}_{g,2}$ is a linear combination of the monodmials $q^{2d} \zeta^{2\alpha}$ for some $d \geq 1$ and $\alpha \in E_8(-1)$.
Since $\mathrm{gcd}(\ell,2d, 2\alpha)$ is always equal to $2$,
by Proposition~\ref{prop:dependence of Fgr} the coefficient of this monomial depends only on 
\[ \frac{1}{8} (2s + 2df+2\alpha)^2 = d+\alpha^2/2. \]
Therefore we may write
\[
\left[ \widehat{F}_{g,2} \right]_{q^{2d} \zeta^{2\alpha}} = \alpha_g(d + \alpha^2/2)
\]
for some cofficients $\alpha_g(n)$. Summing over $d, \alpha$ we get:
\begin{equation} \label{Fhatg2}
\begin{aligned}
\widehat{F}_{g,2} & = \sum_{d,\alpha} \alpha_g(d + \alpha^2/2) q^{2d} \zeta^{2 \alpha} \\
& = \sum_{d,\alpha} \alpha_g(d + \alpha^2/2) q^{2(d+\alpha^2/2)} \zeta^{2 \alpha} (q^2)^{-\alpha^2/2} \\
& = \left( \sum_{n} \alpha_g(n) q^{2n} \right) \Theta_{E_8}(\zeta^2,q^2).
\end{aligned}
\end{equation}
Hence:
\begin{itemize}
\item It suffices to prove the vanishing of $\widehat{F}_{g,2}$ for its $q^{2d} \zeta^0$-coefficients.
\end{itemize}

We can conclude one more vanishing from \eqref{Fhatg2}: Recall that
\[ \widehat{F}_{g,2} \in \frac{1}{\Delta(q)^2} \QJac_{Q_{E_8}}(\Gamma_0(2)). \]
By standard facts on Jacobi forms (e.g. \cite[Lemma 2.20]{HilbHAE} or the Appendix) we also have
\[ \Theta_{E_8}(\zeta^2, q^2) \in \QJac_{Q_{E_8},4}(\Gamma_0(2)). \]
We conclude that
\[ H_g(q) := \sum_{n} \alpha_g(n) q^{2n} \in \frac{1}{\Delta(q)^2} \QMod(\Gamma_0(2)). \]
The holomorphic anomaly equation for $\widehat{F}_{g,2}$ (Proposition~\ref{prop:Fhat HAE}) shows that
\begin{equation} \label{H_g HAE}
\frac{d}{dG_2} H_g(q) = -2 H_{g-1}(q).
\end{equation}
Since $\widehat{F}_{g,2}$ is of weight $2g-2$ and $\Theta_{E_8}(\zeta^2,q^2)$ is of weight $4$, the weight of $H_g$ is $2g-6$.

\begin{lemma}
$H_g(q) = 0$ for $g \leq 8$.
\end{lemma}
\begin{proof}
We argue by induction on $g$, with $H_0=0$ the base.
By induction and \eqref{H_g HAE} we may assume that $H_g$ lies in $\frac{1}{\Delta(q)^4} \Mod(\Gamma_0(2))$.
Since $H_g(q)$ has only even $q$-exponents,
the quasi-modular form $H_g(q)$ satisfies $H_g(\tau+\frac{1}{2}) = H_g(\tau)$ in the variable $\tau$ where $q=e^{2 \pi i \tau}$, that is it satisfies the modular transformation property also for the matrix $\begin{pmatrix} 1 & 1/2 \\ 0 & 1 \end{pmatrix}$. By Remark~\ref{rmk:m=2 generation} it follows that
$H_g(\tau)$ is $h_g(2 \tau)$ for a function $h_g(\tau)$ that satisfies the
modular transformation properties for $\SL_2(\BZ)$.
Since $h_g$ has only a single cusp, and $h_g(2 \tau) = H_g(\tau)$ is bounded for $q \to 0$, we get that $h_g(\tau) \in \Mod = \BC[G_4, G_6]$.
Since the $q^0$ coefficient of $H_g$ vanishes, also the $q^0$-coefficient of $h_g$ vanishes.
Hence $h_g=0$ if its weight $2g-6 \leq 10$, or $g \leq 8$.
\end{proof}
By the lemma we conclude:
\begin{itemize}
\item $\widehat{F}_{g,2} = 0$ for $g \leq 8$.
\end{itemize}
In particular,
the series $\widehat{F}_2 = \sum_{g} (-1)^{g-1} z^{2g-2} \widehat{F}_{g,2}$
satisfies:
\begin{equation} \label{3sdf-}
\widehat{F}_2|_{z=0} = 0, \quad 
\frac{d^2 \widehat{F}_2}{d^2 z}\Big|_{z=0} = 0.
\end{equation}

We want to apply Lemma~\ref{lemma:dt for div 2}
to the coefficient
\[ \left[ \widehat{F}_2 \right]_{q^{2d} \zeta^0} = \widehat{f}_{\beta_d}, \quad \text{ where } \quad \beta_d := 2s+2df. \]
For all odd $r$ we have that $(r,\beta_d,n) \in H^{\ast}(Y,\BZ)$ has divisibility $m(r,\beta_d,n)=1$.
Hence by the case $\ell=1$ above, and Theorem~\ref{prop:DT dependence} and~\ref{prop:GW vs DT} we have for all odd $r$,
\[ \widehat{\dt}(r,\beta_d,n) = 0. \]
Moreover, Lemma~\ref{lemma:dt for div 2} and using the definition of $\dt$ observe also that
\[ \dt^{\mathrm{odd}}_{4t,2} = \DT^{\mathrm{odd}}_{4t,2} = -\DT^{\mathrm{even}}_{4t,2} = -\dt^{\mathrm{even}}_{4t,2} \quad \text{ for all even } t \in \BZ \]
and hence
\[ \widehat{\dt}^{\mathrm{odd}}_{4t,2} = -\widehat{\dt}^{\mathrm{even}}_{4t,2} \quad \text{ for all even } t \in \BZ. \]
where we write $\widehat{\dt}^{t}_{s,m}$ for $\widehat{\dt}(v)$ if $v$ is of type $t$, square $v^2=s$ and divisibility $\mathrm{div}(\pi^{\ast}(v)) = m$.

Define
\[ a(s) := \widehat{\dt}^{\mathrm{odd}}_{8s,2}. \]
For even $r$, lets say $r=2 \tilde{r}$,
we then have $\text{gcd}(r,\beta_d,2n) = 2$ and so:
\begin{align*}
\widehat{\dt}(2 \tilde{r},\beta_d,n) & = 
\begin{cases}
\widehat{\dt}^{\mathrm{odd}}_{\beta_d^2 - 4 \tilde{r}^2 - 4 \tilde{r} n, 2} & \text{ if } \tilde{r} \text{ odd or } n \text{ odd }\\[7pt]
\widehat{\dt}^{\mathrm{even}}_{\beta_d^2 - 4 \tilde{r}^2 - 4 \tilde{r} n, 2} & \text{ if } \tilde{r}, n \text{ even } 
\end{cases}
\\
& = 
(-1)^{(\tilde{r}-1) (n-1)}
a\left( d - \frac{\tilde{r} n}{2} - \frac{\tilde{r}^2}{2} \right).
\end{align*}
Inserting this into Lemma~\ref{fhat beta}
we get that
\begin{gather*}
\widehat{f}_{\beta_d}
=
- \sum_{n, \tilde{r}>0} (-1)^{(\tilde{r}-1) (n-1)}
a\left( d - \frac{\tilde{r} n}{2} - \frac{\tilde{r}^2}{2} \right)
(n+2 \tilde{r}) (p^n + p^{-n}) \\
- \sum_{\tilde{r}>0} (-1)^{\tilde{r}-1} a\left( d - \tilde{r}^2/2 \right)
- a(d) \frac{p}{(1+p)^2}.
\end{gather*}
We argue now by induction on $d$ that $a(d)=0$ for all $d$.
The base of the induction is $d<0$, where clearly $a(d)=0$ 
(since all Gromov-Witten invariants vanish for curve classes of negative square,
and then use the argument in the proof of Proposition~\ref{prop:GW vs DT}).
Assume we have $a(d')=0$ for all $d' < d$.
Then by \eqref{3sdf-} we have
\[ 0 = \widehat{f}_{\beta_d}\Big|_{p=1} = 
-\frac{1}{4} a(d) - a(d-\frac{1}{2}) + (\ldots) \]
\[
0 = \left( \left( p \frac{d}{dp} \right)^2 \widehat{f}_{\beta_d} \right)\Big|_{p=1} = 
\frac{1}{8} a(d) + (\ldots)
\]
where $(\ldots)$ stands for terms $a(s)$ where $s \leq d-1$, so for terms which vanish by induction.
We conclude that both $a(d)=0$ and $a(d-1/2)=0$.

In summary, the vanishing of $a(d)$ implies $\widehat{f}_{\beta_d}=0$ for all $d$,
and hence the vanishing of the $\zeta^0$-coefficient of $\widehat{F}_2$.
However, we have already seen that this implies the vanishing of $\widehat{F}_2$, so we are done with this step.

\vspace{4pt}
\noindent
\textbf{Case $\ell>2$.}
Assume that we have
$\widehat{F}_{g,\ell'} = 0$ for all $\ell' < \ell$.
By Proposition~\ref{prop:dependence of Fgr} and induction it follows that
the $q^d \zeta^{\alpha}$-coefficient of $\widehat{F}_{g,\ell}$ vanishes
unless the class $\ell s+df+\alpha$ is divisible by $\ell$.
Moreover, by Proposition~\ref{prop:Fhat HAE} $\widehat{F}_{g,\ell}$ is a quasi-Jacobi form of a certain weight and index.
We conclude that $\widehat{F}_{g,\ell}$
is a quasi-Jacobi form with Fourier expansion of the form
\[ \widehat{F}_{g,\ell}(\zeta,q) = \sum_{d \geq 0} \sum_{\alpha \in E_8(-1)} b(\alpha,d) \zeta^{\ell \alpha} q^{\ell d}. \]
By Proposition~\ref{prop:quasi Jacobi form vanishing} and Remark~\ref{rmk:vanishing with poles} we get that
$\widehat{F}_{g,\ell} = 0$.
\qed

\begin{rmk}
In the proof above, we could have argued the case $\ell>2$ parallel to the $\ell=2$ case, but we instead chose the more general and simpler approach.
\end{rmk}

\section{Vafa-Witten theory} \label{sec:Vafa Witten}
Let $Y$ be an Enriques surface equipped with a generic polarization $\CO_Y(1)$.
Let $p:K_Y \to Y$ be the projection and let $\CO_{K_Y}(1) = p^{\ast} \CO_Y(1)$ be the induced polarization.

\begin{defn}[{\cite[3.1]{TT2}}]
Let $v \in H^{\ast}(Y,\BQ)$ and $n \gg 0$. A Joyce-Song pair $(E,s)$ consists of
a compactly supported coherent sheaf $\CE$ on $K_Y$ with $\ch(\pi_{\ast} \CE ) = v$,
and a non-zero section $s \in H^0(K_Y,\CE(n))$ such that
\begin{itemize}
\item $\CE$ is Gieseker semi-stable with respect to $\CO_{K_Y}(1)$, and
\item for any proper subsheave $\CF \subset \CE$ which destabilizes $\CE$, the section $s$ does not factor through $\CF(n) \subset \CE(n)$.
\end{itemize}
\end{defn}

The moduli space of Joyce-Song pairs $\CP_v(n)$ is fine, carries a symmetric perfect obstruction theory, and has proper fixed locus with respect to the induced $\BC^{\ast}$-action coming from scaling the fibers of $K_Y$, see \cite{TT1,TT2}.
Using equivariant localization we hence can define invariants
\[ P_v(n) := \int_{[ \CP_v(n)^{\BC^{\ast}} ]^{\vir}} \frac{1}{e(N^{\vir})}. \]
Since $H^{1,0}(Y) = H^{2,0}(Y)=0$
one has the following special case of a conjecture of Tanaka and Thomas:

\begin{conj}[{\cite[Conj.1.2]{TT2}}] \label{conj:TT}
There exist $\VW(v_i)\in\BQ$ such that for all $n\gg0$:
\[
P_{v}(n)\ =\ \mathop{\sum_{\ell\ge 1,\,(v_i=\delta_i v)_{i=1}^\ell:}}_{\delta_i>0,\ \sum_{i=1}^\ell\delta_i=1}
\frac{(-1)^\ell}{\ell!}\prod_{i=1}^\ell(-1)^{\chi(v_i(n))} \chi(v_i(n))\;\VW(v_i)
\]
where $v(n) := v e^{c_1(\CO_Y(n))}$ and $\chi(v) := \int_Y \td_Y v$.
\end{conj}

We prove here the following:
\begin{thm} \label{thm:Vafa Witten full}
Conjecture~\ref{conj:TT} holds and the invariants $\VW(v)$ are given by
\[ \VW(r,\beta,n) = 2 \sum_{\substack{k|(r,\beta,n) \\ k \geq 1 \text{ odd}}}
\frac{1}{k^2} b\left( \frac{\beta^2 - 2rn - r^2}{2k^2} \right),
\quad b(n) := \left[ \frac{1}{\eta^{12}(\tau)} \right]_{q^{n}}. \]
\end{thm}

Vafa and Witten predicted that for fixed rank
the generating series of Vafa-Witten invariants has modular behaviour \cite[Sec.1.4]{TT2}.
We obtain here the following corollary:
\begin{cor} \label{cor:modularity}
For each $r>0$ the series
\[ Z^{\VW}_{r,0}(q) = \sum_{n} \VW(r,0,n) q^{-2n - r} \]
is a (weakly-holomorphic) modular form for $\Gamma_0(4r)$ of weight $-6$.
\end{cor}
\begin{proof}
Let $c(n) = 2 b(n/2)$ where $b(n)$ is defined in Theorem~\ref{thm:Vafa Witten full}. Then we find
\[
Z^{\VW}_{r,0}(q) = \sum_{n} \sum_{\text{odd } k |(r,n)} \frac{1}{k^2} c\left( \frac{nr}{k^2} \right) q^n = 2 \eta^{-12}(2 \tau)|_{-1} V_r
\]
where $V_r$ is the Hecke operator in weight $-1$ for group $\Gamma_0(4)$,
see Section~\ref{subsec:Hecke operators appendix}.
By \cite[Exercise III \S 3.17]{Koblitz} the function $\eta^{12}(2 \tau)$ is a modular form of weight $6$ for $\Gamma_0(4)$, so the result follows
from the discussion of "wrong-weight" Hecke operators in \cite[Prop.13]{BB} or \cite[Sec.2.8]{HilbHAE}.
\end{proof}

\begin{rmk}
In order to generalize Corollary~\ref{cor:modularity} to non-vanishing first Chern class, it is tempting to consider the formal series
\[ Z_r^{\VW} = \sum_{\beta \in H_2(Y,\BZ)} \sum_{n \in \BZ} \VW(r,\beta,n) \zeta^{\beta} q^{-n-\frac{r}{2}}. \]
Since $H^2(Y,\BZ) \cong U \oplus E_8(-1)$ is indefinite,
the theta-like series
\[
\vartheta_{H^2(Y,\BZ)} = \sum_{\beta \in H^2(Y,\BZ)} \zeta^{\beta} q^{-\frac{1}{2} \beta^2}.
\]
does not converge. Nevertheless we may view $\vartheta_{H^2(Y,\BZ)}$ as a formal Jacobi form of weight $5$ and index $\frac{1}{2} H^2(Y,\BZ)$. Then Theorem~\ref{thm:Vafa Witten full} says that
\[
Z_r^{\VW} = \sum_{\beta} \sum_{n : n-r/2 \in \BZ} q^n \zeta^{\beta}
\sum_{k|(r,\beta,2n)} \frac{1}{k^2} b\left( \frac{rn + \beta^2/2}{k^2} \right)
= \left( 2 \vartheta_{H^2(Y,\BZ)} \eta^{-12}(\tau) \right)|_{-1} V_r
\]
so $Z_r^{\VW}$ may be viewed as a formal Jacobi form of weight $-1$ and index $\frac{r}{2} H^2(Y,\BZ)$.
To actually obtain a convergent function for $Z_r^{\VW}$ which satisfies (Mock-)Jacobi form behaviour, one needs to instead regularize the theta series.
We refer to \cite{Manschot} for a discussion in the case of Hirzebruch surfaces and further references. \qed
\end{rmk}

The proof of Theorem~\ref{thm:Vafa Witten full} follows from our computation of the generalized Donaldson-Thomas invariants $\DT(r,\beta,n)$ and a degeneration argument
as we now explain.

Let $Q=(R \times E)/\BZ_2$ be the Enriques Calabi-Yau threefolds,
where $X \to Y$ is the K3 cover,
and let $p : Q \to Y$ and $\iota : Y \to Q$ denote the projection
and one of the sections respectively.
Let $\CO_Q(1) = p^{\ast} \CO_Y(1)$.
For $n \gg 0$, let $\CP^Q_{v}(n)$ be the moduli space of Joyce-Song pairs $(\CE,s)$ on $Q$ with respect to $\CO_Q(1)$ satisfying
$\ch(\CE) = \iota_{\ast}(v)$, or equivalently, that $\CE$ is supported on the fibers of $Q \to E/\BZ_2 = \p^1$ and $\ch(p_{\ast} \CE) = v$.\footnote{The pairs space is formed here with respect to $\CO_Q(1)$ which is not ample, but only relative ample with respect to $Q \to \p^1$, but this suffices since the sheaves we consider are supported in the fibers.} The moduli space $\CP^Q_{v}(n)$ is proper and carries a symmetric perfect obstruction theory. Define the pairs invariants
\[ P^Q_v(n) := \int_{[ \CP^Q_{v}(n) ]^{\vir}} 1. \]
By the wall-crossing formula of Joyce-Song \cite[Thm.5.27]{JS} (compare \cite[Eqn. (3.4)]{TT2}) we have
\begin{equation} \label{dfs33333}
P^Q_{v}(n)\ =\ \mathop{\sum_{\ell\ge 1,\,(v_i=\delta_i v)_{i=1}^\ell:}}_{\delta_i>0,\ \sum_{i=1}^\ell\delta_i=1}
\frac{(-1)^\ell}{\ell!}\prod_{i=1}^\ell(-1)^{\chi(v_i(n))} \chi(v_i(n))\;\DT(v_i).
\end{equation}

Consider the generating series:
\[
f^Q_{v,n}(t) = 1 + \sum_{0 < \delta \leq 1} P^Q_{\delta v}(n) t^{\delta} + O(t^{1+}) \]
where $O(t^{1+})$ stands for considering the series modulo $t^{1+\epsilon}$ for all $\epsilon > 0$, that is as an element in the ring
\[ \BQ[ t^{\delta} | \delta \in \BQ_{>0} ]/ ( t^{1+\epsilon} )_{\epsilon>0} \]
Similarly, we can put the invariants of $K_Y$ in a generating series:
\[ f^{K_Y}_{v,n}(t) = 1 + \sum_{0 < \delta \leq 1} P_{\delta v}(n) t^{\delta} + O(t^{1+}).
\]

\begin{lemma}[Degeneration formula] \label{lemma:deg formula Pairs}
$f^Q_{v,n}(t) = f^{K_Y}_{v,n}(t)^4$
\end{lemma}
\begin{proof}
We use the same degeneration as in Proposition~\ref{prop:GW QvsY}.
By the degeneration formula for pairs \cite{LiWu} and with the obvious notation
we have
\[
f^Q_{v,n}(t) = f^{T/R}_{v,n}(t)^2 = f^{T}_{v,n}(t)^2 = f^{K_Y}_{v,n}(t)^4,
\]
where the second and third equality can be argued like in the proof of Proposition~\ref{prop:GW QvsY}; the modifications to the pairs space are explained in detail in \cite[Sec.4, Sec.7]{MT}.
\end{proof}

\begin{proof}[Proof of Theorem~\ref{thm:Vafa Witten full}]
Note that we can rewrite \eqref{dfs33333} in generating form as:
\[
f^Q_{v,n}(t)
=
\exp\left( \sum_{\delta>0} t^{\delta} (-1)^{ \chi( \delta v(n) )} \chi( \delta v(n) ) \DT( \delta v ) \right) + O(t^{1+}). \]
Hence by Lemma~\ref{lemma:deg formula Pairs} we obtain
\[
f^{K_Y}_{v,n}(t) = 
\exp\left( \sum_{\delta>0} t^{\delta} (-1)^{ \chi( \delta v(n) )} \chi( \delta v(n) ) \cdot \frac{1}{4} \DT( \delta v ) \right) + O(t^{1+}). \]
By taking the $t^1$-coefficient we obtain that $\VW(v)$ are well-defined and satisfy
$\VW(v) = \frac{1}{4} \DT( v )$.
The precise value of $\VW(v)$ now follows from Corollary~\ref{cor:consequence of main thm}.
\end{proof}

\appendix

\section{Background on quasi-Jacobi forms}
\label{sec:Appendix Background quasi-Jacobi forms}
In Section~\ref{subsec:appendix QJac defn} we first introduce slash operators
for quasi-Jacobi forms following partially work of Ziegler \cite{Ziegler}.
Subsequently we give two basic ways to modify quasi-Jacobi forms.
Then we discuss Hecke operators for quasi-Jacobi forms. This proves claims made in Section~\ref{sec:Modular and Jacobi forms}.

\subsection{Definition} \label{subsec:appendix QJac defn}
Write $R^{(m,n)}$ for the group of $m \times n$-matrices with coefficients in a ring $R$.
The Heisenberg group on $n$ variables is given by
\[ H_R^{(n)}  = \{ [(\lambda, \mu), \kappa] | \lambda, \mu \in R^{(n,1)}, \kappa \in R^{(n,n)}, (\kappa + \mu \lambda^{t}) \text{ symmetric} \}. \]
The group structure is defined by
\[ [(\lambda, \mu), \kappa] \cdot [(\lambda', \mu'), \kappa']
=
[(\lambda + \lambda', \mu + \mu'), \kappa + \kappa' + \lambda \mu^{\prime t} - \mu \lambda^{\prime t}]. \]
The group $\SL_2(\BR)$ acts on $H_{\BR}^{(n)}$ from the right by
\[ [(\lambda, \mu), \kappa] \cdot \gamma = [ (\lambda, \mu) \cdot \gamma, \kappa ], \quad \gamma \in \SL_2(\BR). \]
The associated semidirect product is called the Jacobi group:
\[ G_{\BR}^{(n)} := \SL_2(\BR) \ltimes H_{\BR}^{(n)}. \]
Explicitly the product in this group is written by
\[ (\gamma, [X, \kappa]) \cdot
(\gamma', [X', \kappa'])
=
(\gamma \gamma', [X \gamma' + X', \kappa + \kappa' + X \gamma' J X^{\prime t}]) \]
where $X=(\lambda, \mu)$, $X' = (\lambda', \mu')$ and $J=\binom{\phantom{-}0\ 1}{-1\ 0}$.

Let $\BH = \{ \tau \in \BC | \Im(\tau) > 0 \}$ be the upper half plane.
The group $G_{\BR}^{(n)}$ acts on $\BC^n \times \BH$ by
\[ 
(\gamma, [(\lambda, \mu), \kappa]) \cdot (x,\tau)
=
\left( \frac{x + \lambda \tau + \mu}{c \tau + d}, \frac{a \tau + b}{c \tau + d} \right),
\quad \text{ where } \gamma = \binom{a\ b}{c\ d} \text{ and } x = (x_1, \ldots, x_n)^t.
\]

Consider the following real analytic functions on $\BC^n \times \BH$:
\[ 
\nu(\tau) = \frac{1}{8 \pi \mathrm{Im}(\tau)},
\quad
\alpha_i(x,\tau) = \frac{x_i - \overline{x_i}}{\tau - \overline{\tau}} = \frac{\mathrm{Im}(x_i)}{\mathrm{Im}(\tau)},
\quad i = 1, \ldots, n.
\]
An \emph{almost holomorphic function} on $\BC^n \times \BH$ is a function
\[ \Phi(x, \tau) = \sum_{i \geq 0} \sum_{j = (j_1, \ldots, j_n) \in (\BZ_{\geq 0})^n}
\phi_{i, j}(x,\tau) \nu^{i} \alpha^j, \quad \quad \alpha^j = \alpha_1^{j_1} \cdots \alpha_n^{j_n} \]
such that each of the finitely many non-zero $\phi_{i, j}(x,\tau)$ is holomorphic on $\BC^n \times \BH$.
We write $\AH(\BC^n \times \BH)$ for the vector space of almost-holomorphic function.

\begin{defn}
Let $L$ be a symmetric rational $n \times n$ matrix. Let $k \in \BZ$.
For any $\gamma = \binom{a\ b}{c\ d} \in \SL_2(\BR)$, $[(\lambda, \mu),\kappa] \in H_{\BR}^{(n)}$ and $\Phi(x,\tau) \in \AH(\BC^n \times \BH)$ we define
the slash operator:
\begin{gather*}
(\Phi|_{k,L} \gamma)(x,\tau)
=
(c \tau + d)^{-k} e\left( - \frac{ c x^t L x }{c \tau + d} \right) \Phi\left( \frac{x}{c \tau + d}, \frac{a \tau + b}{c \tau + d} \right) \\
(\Phi|_{L} [(\lambda, \mu),\kappa])
=
e\left( \lambda^{t} L \lambda \tau + 2 \lambda^{t} L x + \lambda^{t} L \mu + \mathrm{Tr}(L \kappa) \right) \Phi( x + \lambda \tau + \mu, \tau ) 
\end{gather*}
where $e(x) := e( 2 \pi i x)$ for every $x \in \BC$.
\end{defn}

The functions $\nu$ and $\alpha = (\alpha_1, \ldots, \alpha_n)$ satisfy
the transformations:
\begin{align*}
\nu\left( \frac{a \tau+b}{c \tau+d} \right)
& = \frac{1}{\det(\gamma)} \left[ (c \tau+d)^2 \nu(\tau) + \frac{c (c \tau+d)}{4 \pi i} \right] \\
\alpha\left( \frac{ \det(\gamma) x}{c \tau+d}, \frac{a \tau+b}{c \tau+d} \right) 
& =
( c \tau + d) \cdot \alpha(x,\tau) - c x \\
\alpha(x + \lambda \tau + \mu,\tau) & = \alpha(x,\tau) + \lambda
\end{align*}
for all $\gamma = \binom{a\ b}{c\ d} \in \mathrm{GL}_2^{+}(\BR)$
and $\lambda, \mu \in \BR$.
This shows that the slash operator
sends (almost) holomorphic functions to themselves.
A further direct calculation shows: 

\begin{lemma}[{\cite[Lemma 1.2]{Ziegler}}] \label{lemma:Ziegler}
For $\gamma, \gamma' \in \SL_2(\BR)$, $\zeta, \zeta' \in H_{\BR}^{(n)}$ we have
\[ \Phi|_{k,L} \gamma |_{k,L} \gamma' = \Phi|_{k,L} (\gamma \gamma'),
\quad \quad 
\Phi |_L \zeta |_L \zeta' = \Phi|_{L} (\zeta \zeta') , \quad  \quad
\Phi|_{L} \zeta |_{k,L} \gamma = \Phi|_{k,L} \gamma |_{L} (\zeta \gamma) 
\]
\end{lemma}
\begin{cor}
The group $G_{\BR}^{(n)}$ acts on $\AH(\BC^n \times \BH)$ by
$\Phi \mapsto \Phi|_{k,L} (\gamma, \zeta) := \Phi|_{k,L} \gamma |_{L} \zeta$.
\end{cor}

Let $\Gamma \subset \SL_2(\BZ)$ be a congruence subgroup,
and let $\Lambda \subset \BZ^{(n,2)}$ be a finite index subgroup
which is preserved under the action of $\Gamma$ on $\BZ^{(n,2)}$ by multiplication on the right. 
Define the subgroup of $H_{\BZ}^{(n)}$ given by
\[ H_{\Lambda} = \left\{ [(\lambda, \mu), \kappa] \in H_{\BZ}^{(n)}\ \middle|\ (\lambda, \mu) \in \Lambda, \kappa \in \mathrm{Span}_{\BZ}( \mu \lambda^{t}, \lambda \mu^t | (\lambda, \mu) \in \Lambda), \kappa + \mu \lambda^t \text{ symmetric} \right\}. \]
We obtain the subgroup $\Gamma \ltimes H_{\Lambda} \subset G_{\BR}^{(n)}$.

\begin{defn} \label{defn:AHJ}
An almost holomorphic Jacobi form 
for $\Gamma \ltimes \Lambda$ of weight $k$ and index $L$  is an
almost-holomorphic function $\Phi(x,\tau)
= \sum_{i \geq 0} \sum_{j} \phi_{i, j}(x,\tau) \nu^{i} \alpha^j$ satisfying
\begin{itemize}
\item[(i)] $\Phi(x,\tau)|_{k,L} g = \Phi(x,\tau)$
for all $g \in \Gamma \ltimes H_{\Lambda}$,
\item[(ii)] for all $g \in \SL_2(\BZ)$, the almost-holomorphic function $\Phi|_{k,L} g$ is of the form $\sum_{i,j} \phi_{i,j} \alpha^i \nu^j$ such that
each of the finitely many non-zero holomorphic functions $\phi_{i,j}$ admits a Fourier expansion of the form $\sum_{v \geq 0} \sum_{r \in \BZ^n} c(v,r) q^{v/N} \zeta^{r/N}$
in the region $|q|<1$ for some $N \geq 1$,
and we used here the notation
$\zeta^r = e(x \cdot r)$.
\end{itemize}
\end{defn}

Any element $[(0,0), \kappa] \in H_{\Lambda}$ acts trivially on $\BH \times \BC^n$;
hence for an almost-holomorphic Jacobi form for group $\Gamma \ltimes \Lambda$ to be non-zero we must have that the index $L$ satisfies:
\begin{equation}
\tag{$\dagger$}
\Tr(L \kappa) \in \BZ \quad \text{for all} \quad [(0,0), \kappa] \in H_{\Lambda}
\end{equation}
In Section~\ref{subsec:Lattice index quasi Jacobi forms} we assumed that $L$ satisfies ($\dag$), hence 
Definition~\eqref{defn:AHJ} recovers and generalizes the previous Definition~\ref{defn:quasi jacobi forms main text}.

We will use the definitions of quasi-Jacobi forms, the vector spaces $\AHJ_{k,L}(\Gamma \ltimes \Lambda)$ and $\QJac_{k,L}(\Gamma \ltimes \Lambda)$,
and the holomorphic anomaly operators $\frac{d}{dG_2}$, $\xi_{\lambda}$ as in Section~\ref{subsec:Lattice index quasi Jacobi forms}.

\subsection{Modifications} \label{subsec:appendix modifications}
We describe two basic ways to modify quasi-Jacobi forms.

We start with two technical lemmata.
Extend the slash operators to the group $\GL_2(\BR)^{+}$ of $2 \times 2$-matrices with positive determinant as follows.
For any $\gamma = \binom{a\ b}{c\ d} \in \GL_2(\BR)^{+}$ let
$\overline{\gamma} := \gamma / \sqrt{\det \gamma }$ where we take the positive squareroot. Then set
\[ f|\gamma := f| \overline{\gamma}
=
\det(\gamma)^{k/2} (c \tau + d)^{-k} e\left( - \frac{ c x^t L x }{c \tau + d} \right) \Phi\left( \frac{\sqrt{\det \gamma} \cdot x}{c \tau + d}, \frac{a \tau + b}{c \tau + d} \right).
\]
If we let $\GL_2(\BR)^+$ act on $H_{\BR}^{(n)}$ by $\zeta \cdot \gamma := \zeta \cdot \overline{\gamma}$, then the slash operator for $\GL_2(\BR)^{+}$ satisfies the relations in Lemma~\ref{lemma:Ziegler} and we obtain an action of $\GL_2(\BR)^{+} \ltimes H_{\BR}^{(n)}$.

For $\ell \in \BR_{>0}$ define also
\[ (f|U_{\ell})(x,\tau) := f( \ell x, \tau). \]

\begin{lemma} \label{lemma:Uslash commutator}
For any $\gamma \in \GL_{2}^+(\BR)$ and $\zeta = [(\lambda,\mu),\kappa] \in H_{\BR}^{(n)}$ we have
\[ f|_{k,L} \gamma | U_{\ell} = f| U_{\ell} |_{k,\ell^2 L} \gamma,
\quad
f|_{L} [(\lambda,\mu),\kappa] | U_{\ell} = f|U_{\ell} |_{\ell^2 L} 
\left[ \left( \frac{\lambda}{\ell},\frac{\mu}{\ell} \right),\frac{\kappa}{\ell^2} \right]
\]
\end{lemma}
\begin{proof} Straightforward computation. \end{proof}

\begin{lemma} \label{lemma:GL2 expansion}
Let $\Phi$ be an almost-holomorphic Jacobi form for $\Gamma \ltimes \Lambda$ of weight $k$ and index $L$. Then for any $\gamma = \binom{a\ b}{c\ d} \in \GL_2^{+}(\BQ)$ we have that $\Phi|_{k,L}\gamma|U_{\sqrt{\det(\gamma)}}$ is of the form $\sum_{i,j} \phi_{i,j} \alpha^i \nu^j$ such that each of the finitely many non-zero holomorphic functions $\phi_{i,j}$ admits a Fourier expansion of the form $\sum_{v \geq 0} \sum_{r \in \BZ^n} c(v,r) q^{v/N} \zeta^{r/N}$ in the region $|q|<1$ for some $N \geq 1$,
\end{lemma}
\begin{proof}
Any $\gamma \in \GL_2^{+}(\BQ)$ can be written as $\gamma_1 \cdot \gamma_2$ where $\gamma_1 \in \SL_2(\BZ)$ and $\gamma_2 = \binom{a\ b}{0\ d}$, compare \cite[III.3, Lemma 2]{Koblitz}. By definition of an almost-holomorphic Jacobi form we have that 
\[ \Phi|_{k,L} \gamma_1 = \sum_{j \geq 0} \sum_{i=(i_1, \ldots, i_n) \in (\BZ_{\geq 0})^n} \alpha^i \nu^{j} \sum_{v \geq 0} \sum_{r \in \BZ^n} c_{i,j}(v,r) e\left( \frac{v \tau}{N} \right) e\left( \frac{rx}{N} \right) \]
for some $N \geq 1$. Hence we get that
\[ \Phi|_{k,L} \gamma_1 | \gamma_2 = 
\sum_{i,j} \left(\frac{d}{\sqrt{ad}} \right)^{-k+2j+\sum_{r} i_r} \alpha^i \nu^{j} \sum_{v \geq 0} \sum_{r \in \BZ^n} c_{i,j}(v,r) e\left( \frac{v (a \tau+b)}{N d} \right) e\left( \frac{\sqrt{ad} rx}{d N} \right). \]
Appling $U_{\sqrt{ad}}$ we hence get the desired form.
\end{proof}

Let $L$ be a fixed index and consider the subgroup
\[ V_{L} = \{ [(0,0),\kappa] \in H_{\BZ}^{(n)}\, |\, \Tr(L \kappa) \in \BZ \}. \]
Let $\zeta_X = [X, \kappa_X] \in H_{\BZ}^{(n)}$,
$X = (\lambda, \mu)$, and
let $\overline{\zeta_X}$ be its image in the coset $(H_{\Lambda} \cdot V_L) \setminus H_{\BZ}^{(n)}$. 
Let $\Gamma_{X} \subset \Gamma$
be the stabilizer of $\overline{\zeta_X}$
with respect to the induced action of $\Gamma$ on the coset.\footnote{$\Gamma_{X}$ is independent of the choice of $\kappa_X$.
}
With $J=\binom{\phantom{-}0\ 1}{-1\ 0}$ as before, consider also the subgroup
\[ \Lambda_{X} = \{ X' \in \Lambda \  |\  2 \cdot \Tr( L X J (X')^{t} ) \in \BZ \}. \]

\begin{lemma} \label{lemma:Jac2}
Let $\phi(x,\tau)$ be a quasi-Jacobi form for $\Gamma \ltimes \Lambda$
 of weight $k$ and index $L$.
 Then 
\[
\phi||_L (\lambda,\mu) :=
e\left( \lambda^t L \lambda \tau + 2 \lambda^t L x \right)
\left( e^{\xi_{\lambda}} \phi  \right)(x + \lambda \tau + \mu, \tau)
\] 
is a quasi-Jacobi form for $\Gamma_{X} \ltimes \Lambda_X$ of weight $k$ and index $L$.
Moreover,
\begin{equation} \frac{d}{dG_2} \left( \phi||_L (\lambda,\mu) \right) = \left( \frac{d}{dG_2} \phi \right) ||_{L} (\lambda,\mu). \label{G2 equationdfsdf} \end{equation}
\end{lemma}
\begin{proof}
Let $\Phi = \ct^{-1}(\phi)$ be the non-holomorphic completion of $\phi$. Observe that
\[ \phi||_L (\lambda,\mu) = \ct( \Phi|_{L} \zeta_X ). \]
To prove the first claim it hence suffices to prove that
$\Phi|_{L} \zeta_X$ is an almost-holomorphic Jacobi form
of weight $k$ and index $L$ for $\Gamma_X \ltimes \Lambda_X$.
This is a straightforward application of the slash operator: 
First, for $\gamma \in \Gamma_X$
we can write $\zeta_X \gamma = \zeta \circ [0,\kappa] \circ [X,\kappa_X]$ for some $\zeta \in H_{\Lambda}$ and $[0,\kappa] \in V_{L}$.
Then we get (with omitting the index $L$ from the notation):
\begin{align*}
\Phi| \zeta_X | \gamma = \Phi| \gamma | (\zeta_X \circ \gamma) = \Phi | (\zeta_X\circ  \gamma)
=
\Phi | \zeta | [0,\kappa] | \zeta_X = \Phi | \zeta_X.
\end{align*}
Second, 
for any 
 $\zeta_{X'} = [X', \kappa'] \in H_{\Lambda_X}$ we have 
\begin{align*}
\Phi| \zeta_X | \zeta_{X'}
& = \Phi | \zeta_{X'} | [0, -X' J X^t + X J X^{\prime t}] | \zeta_X \\
& = e( \Tr(L (-X' J X^t + X J X^{\prime t}))) \Phi| \zeta_X \\
& = e( 2 \Tr(L X J  X^{\prime t})) \Phi| \zeta_X \\
& = \Phi| \zeta_X.
\end{align*}
This shows that $\Phi$ satisfies part (i) of Definition~\ref{defn:AHJ}.
The part (ii) of the definition is immediate by using $\Phi|\zeta_X | g = \Phi |g | (\zeta_X g)$
for any $g \in \SL_2(\BZ)$ and condition (ii) for $\Phi$.

The equality \eqref{G2 equationdfsdf} follows immediately by observing
\[
\frac{d}{dG_2} \left( \phi||_L (\lambda,\mu) \right)
=
\frac{d}{dG_2} \ct( \Phi|_{L} \zeta_X ) = \mathrm{Coeff}_{\nu^1 \alpha^0}( \Phi|_{L} \zeta_X ). \qedhere
\]
\end{proof}

For an integer $N \geq 1$ define an operator $R_N$ 
on functions $f(x,\tau)$ by
\[ (f|R_N)(x,\tau) := f( x, N \tau). \]

\begin{lemma} \label{lemma:G2 derivative and scaling}
The operator $R_N$ restricts to a homomorphism
\[ \QJac_{k,L}(\SL_2(\BZ) \ltimes \BZ^{(n,2)}) \to \QJac_{k,L/N}( \Gamma_0(N) \ltimes (N \BZ^n \oplus \BZ^n)), \quad f \mapsto f|R_N \]
satisfying $\frac{d}{dG_2} (f|R_N) = \frac{1}{N} ( \frac{d}{dG_2} f)|R_N$
and $\xi_{\lambda} (f| R_N) = \frac{1}{N} (\xi_{\lambda} f)|R_N$.
\end{lemma}
\begin{proof}
Let $\gamma_N=\binom{N\ 0}{0\ 1}$. Then we have
\[ f|R_N = N^{-k/2} \cdot f|_{k,L}\gamma_N | U_{1/\sqrt{N}}. \]
Let $f \in \AHJ_{k,L}(\SL_2(\BZ) \ltimes \BZ^{(n,2)})$.
We check that $f|R_N$ satisfies the conditions of Definition~\ref{defn:AHJ}.
For $\gamma \in \Gamma_0(N)$ by Lemma~\ref{lemma:Uslash commutator} we have
\begin{align*}
(f|R_N)|_{k,L/N} \gamma 
& = N^{-k/2} f|_{k,L} \gamma_N | U_{1/\sqrt{N}} |_{k,L/N} \gamma \\
& = N^{-k/2} f|_{k,L} \gamma_N \gamma | U_{1/\sqrt{N}} \\
& = N^{-k/2} f|_{k,L} \gamma_N \gamma \gamma_N^{-1} |_{k,L} \gamma_N | U_{1/\sqrt{N}} \\
& = f|R_N,
\end{align*}
since $\gamma_N \Gamma_0(N)\gamma_N^{-1} \subset \SL_2(\BZ)$.
Similarly, for $\zeta=[(\lambda,\mu), \kappa] \in H_{N \BZ^n \oplus \BZ^n}$, we have
\begin{align*}
(f|R_N)|_{L/N} \zeta 
& = N^{-k/2} \cdot f|_{k,L} \gamma_N |_{L} \left[ \left( \frac{\lambda}{\sqrt{N}}, \frac{\mu}{\sqrt{N}} \right), \frac{\kappa}{N} \right] | U_{1/\sqrt{N}} \\
& = N^{-k/2} \cdot f |_{L} \left[ \left( \frac{\lambda}{\sqrt{N}}, \frac{\mu}{\sqrt{N}} \right), \frac{\kappa}{N} \right] \circ \gamma_{N}^{-1} | \gamma_N | U_{1/\sqrt{N}} \\
& = N^{-k/2} \cdot f|_{L}  \left[ \left( \frac{\lambda}{N}, \mu \right), \frac{\kappa}{N} \right] | \gamma_N | U_{1/\sqrt{N}} \\
& = f|R_N.
\end{align*}
Finally, for any $\gamma \in \SL_2(\BZ)$ we have
\[
f|_{k,L} \gamma_N|U_{1/\sqrt{N}} |_{k,L/N} \gamma
=
f|_{k,L} \gamma_a \gamma | U_{1/\sqrt{N}}
=
f|_{k,L} \gamma_a \gamma | U_{\sqrt{N}} | U_{1/N},
\]
hence property (ii) of Definition~\ref{defn:AHJ} follows from Lemma~\ref{lemma:GL2 expansion}.
We hence have an operator
\[ \AHJ_{k,L}(\SL_2(\BZ) \ltimes \BZ^{(n,2)}) \to \AHJ_{k,L/N}( \Gamma_0(N) \ltimes (N \BZ^n \oplus \BZ^n)). \]
The first claim follows from this and by noting that $\ct( f|R_N ) = \ct(f)|R_N$.
If $f = \sum_{i,j} f_{ij} \alpha^i \nu^j$ then
$f|R_N = \sum_{i,j} f_{ij} N^{-j - \sum_r i_r} \alpha^i \nu^j$,
hence the $\nu^1$ and $\alpha_i^1$ coefficient of $f|R_N$ are both multiplied by $1/N$, which shows the second claim.
\end{proof}

\begin{proof}[Proof of Lemma~\ref{lemma: QJac scaling to Gamma0(2)} in Section~\ref{sec:Modular and Jacobi forms}]
Let $f(x,\tau)$ be a quasi-Jacobi form for $\SL_2(\BZ) \ltimes (\BZ^8 \oplus \BZ^8)$
of weight $k$ and index $m Q_{E_8}$ where $m \in \BZ_{\geq 0}$.
By using the commutation between $R_2$ with $\xi_{\alpha_0}$ 
stated in Lemma~\ref{lemma:G2 derivative and scaling} we have
\[
M_{\alpha_0,m}(f) = ( f|R_2 )||_{\frac{m}{2} Q_{E_8}} (\alpha_0, 0).
\]
By Lemma~\ref{lemma:G2 derivative and scaling} $f|R_2$ is a quasi-Jacobi form
of weight $k$ and index $\frac{m}{2} Q_{E_8}$
for $\Gamma_0(2) \ltimes (2 \BZ^8 \oplus \BZ^8)$.
Hence by Lemma~\ref{lemma:Jac2},
$M_{\alpha_0,m}(f)$ is a quasi-Jacobi form of the same weight and index, for the group $\Gamma_{(\alpha_0,0)} = \Gamma_0(2)$ (use that for $\binom{a\ b}{c\ d} \in \Gamma_0(2)$ we must have $a$ odd), and lattice
\[
\Lambda_{(\alpha_0,0)} = \{ (\lambda', \mu') \in 2 \BZ^8 \oplus \BZ^8 | m (\alpha_0 \cdot_{E_8} \mu') \in \BZ \} = 2 \BZ^8 \oplus \BZ^8.
\]
The compatibility with the $G_2$-derivative follows likewise by
Lemmata~\ref{lemma:G2 derivative and scaling} and
\ref{lemma:Jac2}.
\end{proof}

\subsection{Hecke operators} \label{subsec:Hecke operators appendix}
Let $N \geq 1$ be fixed. For $\ell \geq 1$ consider the set
\[ S_{\ell} = \left\{ \begin{pmatrix} a & b \\ c & d \end{pmatrix} \in \BZ^{(2,2)}
\middle|
c \equiv 0(N), \gcd(a,N) = 1, ad-bc=\ell \right\}. \]
The set $S_{\ell}$ is preserved under multiplication by $\Gamma_0(N)$ from the left and right. 

For $f \in \AHJ_{k,L}(\Gamma_0(N))$ define the Hecke operator
\[ f|_{k,L} V_{\ell} := \ell^{k-1} \sum_{\gamma \in \Gamma_0(N) \setminus S_{\ell} }
(c \tau + d)^{-k} e\left( - \frac{c \ell x^t L x}{c \tau + d} \right) f\left( \frac{ \ell x}{c \tau + d}, \frac{a \tau + b}{c \tau + d} \right). \]
By observing that this can be rewritten as
\[
f|_{k,L} V_{\ell} = \ell^{k/2-1} \sum_{ \gamma \in \Gamma_0(N) \setminus S_{\ell} }
f|_{k,L} \gamma | U_{\sqrt{\ell}}
\]
one observes that this is well-defined, i.e. independent of the representative of $\gamma$.

By using Lemma~\ref{lemma:Uslash commutator}, the properties of the slash operation,
and Lemma~\ref{lemma:GL2 expansion} one checks that the Hecke operator gives a map
\[
\AHJ_{k,L}(\Gamma_0(N)) \to \AHJ_{k,\ell L}(\Gamma_0(N)), f \mapsto f|_{k,L} V_{\ell}.
\]
We obtain a Hecke operator on quasi-Jacobi forms:
\begin{equation} \label{dsf0df}
\QJac_{k,L}(\Gamma_0(N)) \to \QJac_{k,\ell L}(\Gamma_0(N)), f \mapsto f|_{k,L} V_{\ell} := \ct\left( \ct^{-1}(f)|_{k,L} V_{\ell} \right).
\end{equation}

A system of representatives of $\Gamma_0(N) \setminus S_{\ell}$ is given by the collection of matrices \cite[A.24]{JS_CHL} 
\[
\begin{pmatrix} a & b \\ 0 & d \end{pmatrix}, \quad ad=\ell,\ a,d>0, \ \gcd(a,N)=1, \ 0 \leq b \leq d-1.
\]
Then arguing precisely as in \cite[Proposition 2.18]{HilbHAE} proves that
if $f \in \QJac_{k,L}(\Gamma_0(N))$ has Fourier expansion
\[ f(x,\tau) = \sum_{v} \sum_{r \in \BZ^n} c(v,r) q^v \zeta^r \]
then
\begin{equation} \label{Fourier expansion}
 (f|_{k,L} V_{\ell})(x,\tau)
=
 \sum_{v} \sum_{r \in \BZ^n} q^n \zeta^r
 \sum_{\substack{ a|(v,r,\ell) \\ \mathrm{gcd}(a,N)=1 }}
 a^{k-1} c\left( \frac{\ell v}{a^2}, \frac{r}{a} \right).
 \end{equation}
Moreover, as in \cite[Proposition 2.18]{HilbHAE} for all $\lambda \in \BZ^n$ we get
\begin{gather*}
\frac{d}{dG_2} (f |_{k,L} V_{\ell}) = \ell \left( \frac{d}{dG_2} f \right)\Big|_{k-2,L} V_{\ell} \\
\xi_{\lambda} (f |_{k,L} V_{\ell}) = \ell (\xi_{\lambda} f)|_{k-1,L}
\end{gather*}

\begin{proof}[Proof of Proposition~\ref{prop:Hecke operator}]
The above arguments about the transformation properties of $f|V_{\ell}$, its Fourier expansion and the holomorphic anomaly equation
work identifically if we start with $f \in \frac{1}{\Delta(q)^s} \QJac_{k,L}(\Gamma_0(N))$. We hence only have to show that for $N=2$, the
function $\Delta(\tau)^{\ell s} ( f|_{k,L} V_{\ell})$ is a quasi-Jacobi form.

Let $F = \ct^{-1}(f)$ be the non-holomorphic completion of $f$.
We hence have to show that
\[ F' = \Delta(\tau)^{\ell s} ( F|_{k,L} V_{\ell}) \]
satisfies the cusp condition given by Definition~\ref{defn:AHJ}(ii).
We only need to check the conditions for representatives of the coset $\Gamma_0(2) \setminus \SL_2(\BZ)$. Since $\Gamma_0(2)$ has two cusps, at $\tau \in \{ 0, i \infty \}$, these representatives can be chosen to be the identity matrix and $S = \begin{pmatrix} 0 & 1 \\ -1 & 0 \end{pmatrix}$.
The case of the identity matrix is clear by the Fourier-expansion \eqref{Fourier expansion}. For $S$, a system of representatives of $\Gamma_0(2) \setminus S_{\ell}$ can be chosen by an elementary argument (see \cite[App.C]{JS_CHL}) to be
\begin{equation} \label{Dsfdsdddd}
I_1 = \left\{ \begin{pmatrix} a & 0 \\ c & d \end{pmatrix} \ \middle| \  \quad ad=\ell,\ a,d>0, \ \gcd(a,N)=1, \ c=Nc_0, 0 \leq c_0 \leq a-1 \right\}.
\end{equation}
if $\ell$ is odd, and the set $I_1 \sqcup I_2$ if $\ell$ is even, where
\begin{equation*}
I_2 = \left\{ \begin{pmatrix} a & -b \\ c & 0 \end{pmatrix} \ \middle| \  \quad bc=\ell,\ bc>0, \ \gcd(a,N)=1, \ c \equiv 0(N),\ 0 \leq a \leq c-1 \right\}.
\end{equation*}

We obtain that:
\begin{align*}
F|V_{\ell}|_{k, \ell L} S
& = \ell^{k/2-1} \sum_{\gamma \in I_1} (F|_{k,L} S) |_{k,L} (S^{-1} \gamma S) | U_{\sqrt{\ell}}  \\
& + 
\delta_{\ell \text{even}}
\ell^{k/2-1} \sum_{\gamma \in I_1} (F|_{k,L} \gamma S) | U_{\sqrt{\ell}} 
\end{align*}
By observing that $S^{-1} \gamma S = \begin{pmatrix} \delta & -\gamma \\ 0 & \alpha \end{pmatrix}$ for $\gamma \in I_1$,
and $\gamma S = \begin{pmatrix} -b & a \\ 0 & \gamma \end{pmatrix}$ for $\gamma \in I_2$, one checks immediately that the $\nu^i \alpha^j$-coefficients of $F|V_{\ell}|S$ 
have a Fourier expansion with lowest term at most $q^{-s \ell}$.
Hence $(\Delta(\tau)^{s \ell} (F|V_{\ell}))|S$ is bounded as $q \to 0$.
\end{proof}

\section{Elliptic holomorphic anomaly equation in relative geometries}
Let $\pi : X \to B$ be an elliptic fibration with section $\iota : B \to X$ and integral fibers such that $H^{2,0}(X) = 0$. Let $D \subset X$ be a smooth divisor and assume that $\pi$ restricts to an elliptic fibration $\pi_D : D \to A$ for a smooth divisor $A \subset B$. In \cite{RES} the $\pi$-relative Gromov-Witten classes of $(X,D)$ were studied and conjecturally linked to quasi-Jacobi forms.
A main conjecture is the holomorphic anomaly equation with respect to $G_2$ stated in \cite[Conj D]{RES}.
By the Lie algebra relations for quasi-Jacobi forms \cite[Sec.1]{RES}
the conjecture immediately implies an {\em elliptic} holomorphic anomaly equation.
The purpose of this subsection is to state this elliptic equation.
This is used in Theorem~\ref{thm:R HAE} in the main text.

We use the notation of \cite{RES} for Gromov-Witten classes.
In particular, let $N_{\iota}$ be the normal bundle of the section and
set $W=\iota_{\ast}[B] - \frac{1}{2} \pi^{\ast} c_1(N_{\iota}) \in H^2(X,\BQ)$.
Consider the natural splitting given in \cite[Sec.2.1.2]{RES},
\[ H^2(X,\BZ) \cong \BZ [B] \oplus \pi^{\ast} H^2(B,\BZ) \oplus \Lambda, \]
let $b_1, \ldots, b_n$ be a basis of the lattice $\Lambda$,
and identify $x=(x_1,\ldots, x_n) \in \BC^n$ with $\sum_i x_i b_i$. 

Let $\gamma_1, \ldots, \gamma_n \in H^{\ast}(X)$ and consider an ordered cohomology weighted partition
\[
\underline{\eta} = \big( (\eta_1, \delta_1), \ldots, (\eta_{l(\eta)}, \delta_{l(\eta)} ) \big), \quad \delta_i \in H^{\ast}(D).
\]
Let $\kk \in H_2(B,\BZ)$. The $\pi$-relative disconnected Gromov-Witten series of $X$ is then
\[
\CC_{g,\kk}^{\pi/D, \bullet}( \gamma_1, \ldots, \gamma_n ; \underline{\eta} )
=
\sum_{\pi_{\ast} \beta = \kk} e(x \cdot \beta) q^{W \cdot \beta}
\pi_{\ast} \left( \left[  \Mbar_{g,n}'(X/D, \beta; \eta) \right]^{\text{vir}} \prod_{i=1}^{n} \ev_i^{\ast}(\gamma_i) \prod_{i=1}^{l(\eta)} \ev_i^{\mathrm{rel} \ast}(\delta_i) \right).
\]
\begin{conj}[{\cite[Sec.4]{RES}}] \label{conj:RES} We have
\[ \CC_{g,\kk}^{\pi/D, \bullet}( \gamma_1, \ldots, \gamma_n ; \underline{\eta} )
\in H_{\ast}(\Mbar_{g,n}^{\bullet}(B/A,\kk;\eta)) \otimes \Delta(q)^{\frac{1}{2} c_1(N_{\iota}) \cdot \kk} \QJac_{\frac{1}{2} Q_{k}}. \]
Moreover, the cycle-valued quasi-Jacobi form $\CC_{g,\kk}^{\pi/D, \bullet}( \gamma_1, \ldots, \gamma_n ; \underline{\eta} )$
satisfies a holomorphic anomaly equation of the form
\[
\frac{d}{dG_2} \CC_{g,\kk}^{\pi/D, \bullet}( \gamma_1, \ldots, \gamma_n ; \underline{\eta} ) = (\text{4 explicit terms}),
\]
where we refer to \cite[Conj D]{RES} for the precise form of the terms.
\end{conj}

Here we are interested in the following consequence of this conjecture:

\begin{prop} \label{prop:elliptic anomaly appendix}
If Conjecture~\ref{conj:RES} holds for $X \to B$, then for any $\lambda \in \Lambda$ we have the elliptic holomorphic anomaly equation
\begin{align*}
\xi_{\lambda}  \CC_{g,\kk}^{\pi/D, \bullet}( \gamma_1, \ldots, \gamma_n ; \underline{\eta} )
& =
\sum_{i=1}^{n} \CC_{g,\kk}^{\pi/D, \bullet}( \gamma_1, \ldots, \gamma_{i-1}, t_{\lambda}(\gamma_i), \gamma_{i+1}, \ldots, \gamma_n ; \underline{\eta} ) \\
& +
\sum_{i=1}^{\ell(\eta)}
\CC_{g,\kk}^{\pi/D, \bullet}\left( \gamma_1, \ldots, \gamma_n ; \underline{\eta}\big|_{\delta_i \mapsto t_{\lambda|_{D}}(\delta_i)} \right)
\end{align*}
where $t_{\lambda}(\gamma) = \lambda \cup \pi^{\ast} \pi_{\ast}(\gamma) - \pi^{\ast} \pi_{\ast}( \lambda \cup \gamma)$,
and $\underline{\eta}\big|_{\delta_i \mapsto t_{\lambda|_{D}}(\delta_i)}$
stands for the cohomology weighted partition $\underline{\eta}$ but with the $i$-th weighting replaced by $t_{\lambda|_{D}}(\delta_i)$.

Moreover, if Conjecture~\ref{conj:RES} only holds numerically, i.e. after integrating against any tautological class pulled back from the moduli space of curves (see \cite{RES}),
then the elliptic anomaly equation above also holds numerically.
\end{prop}
\begin{proof}
This is argued precisely as in \cite[Lemma 16]{RES}.
If $\lambda=\sum_i \lambda_i b_i$ is the decomposition of $\lambda$ in our chosen basis, and $D_{\lambda} = \sum_i \lambda_i \frac{1}{2 \pi i} \frac{d}{dx_{i}}$, then one has the commutation relation:
\[ \left[ \frac{d}{dG_2}, D_{\lambda} \right] = -2 \xi_{\lambda}. \]
Let $p:\Mbar^{\bullet, \dag}_{g,n+1}(B,k) \to \Mbar^{\bullet}_{g,n}(B,k)$ be the morphism forgetting the $(n+1)$-th component, where the supscript '$\dag$' stands for the union of all connected components of $\Mbar^{\bullet}_{g,n+1}(B,k)$ where $p$ is defined.
Then by the divisor equation one has
\[ D_{\lambda} \CC_{g,\kk}^{\pi/D, \bullet}( \gamma_1, \ldots, \gamma_n ; \underline{\eta} ) =
p_{\ast}\left( \CC_{g,\kk}^{\pi/D, \bullet}( \gamma_1, \ldots, \gamma_n, \lambda ; \underline{\eta} )|_{\Mbar^{\bullet, \dag}_{g,n+1}(B,k)} \right). \]
Hence we obtain:
\begin{multline*}
-2 \xi_{\lambda} \CC_{g,\kk}^{\pi/D, \bullet}( \gamma_1, \ldots, \gamma_n ; \underline{\eta} )
=
p_{\ast}\left( \frac{d}{dG_2} \left( \CC_{g,\kk}^{\pi/D, \bullet}( \gamma_1, \ldots, \gamma_n, \lambda ; \underline{\eta} )|_{\Mbar^{\bullet, \dag}_{g,n+1}(B,k)} \right) \right) \\
- D_{\lambda} \frac{d}{dG_2} \CC_{g,\kk}^{\pi/D, \bullet}( \gamma_1, \ldots, \gamma_n ; \underline{\eta} ).
\end{multline*}
In this difference four terms contribute, corresponding to the four terms
of the $G_2$ holomorphic anomaly equation on $\Mbar^{\bullet}_{g,n+1}(B,k)$.

The first term arises when there is a contracted genus $0$ component with 3 markings labeld by some $i \leq n$, by $n+1$ and either $n+2$ or $n+3$ (these correspond to the relative diagonal insertion). 
By using the relative diagonal splitting (e.g. \cite[Sec.4.6]{Marked}) it evaluates to
\[
2 \sum_{i=1}^{n} \CC_{g,\kk}^{\pi/D, \bullet}( \gamma_1, \ldots, \gamma_{i-1}, \pi^{\ast}\pi_{\ast}(\lambda \cdot \gamma_i), \gamma_{i+1}, \ldots, \gamma_n ; \underline{\eta} ).
\]

The second term arises from the second term in the holomorphic anomaly equation when the marking $(n+1)$ lies on the bubble, and the bubble is contracted after forgetting the marking. Since the $(n+1)$-th insertion can lie on a connected component with relative markings either $(b, \Delta_{A,\ell}^{\vee})$ and $(\eta_i,\delta_i)$, or $(b_i, \Delta_{D,\ell}^{\vee}$ and $(\eta_i,\delta_i)$ we get the contribution:
\begin{gather*}
2 \sum_{i=1}^{\ell(\eta)}
\CC_{g,\kk}^{\pi/D, \bullet}\left( \gamma_1, \ldots, \gamma_n ; \underline{\eta}\big|_{\delta_i \mapsto \pi_D^{\ast}\pi_{D \ast}( \lambda|_{D} \cdot \delta_i)} \right) 
+ 
2 \sum_{\substack{i,j=1 \\ i \neq j}}^{\ell(\eta)}
\frac{\eta_i}{\eta_j}
\CC_{g,\kk}^{\pi/D, \bullet}\left( \gamma_1, \ldots, \gamma_n ; \underline{\eta}\big|_{\substack{\delta_i \mapsto \lambda \cdot \delta_i \\ \delta_j \mapsto \pi_D^{\ast} \pi_{D \ast}(\delta_j)}} \right).
\end{gather*}
The third term arises from comparing $p^{\ast}(\psi_i)$ with $\psi_i$, corresponding to contributions when the $i$-th marking lies on a rational tail together with the $(n+1)$-th markings. 
There are again two choices corresponding to the distributions of the relative markings, giving the contribution
\[ -2 \sum_{i=1}^{n} \CC_{g,\kk}^{\pi/D, \bullet}( \gamma_1, \ldots, \gamma_{i-1}, \lambda \cdot \pi^{\ast}\pi_{\ast}(\gamma_i), \gamma_{i+1}, \ldots, \gamma_n ; \underline{\eta} ).
\]
Finally, the last term comes from comparing the relative $\psi$-classes
$p^{\ast}(\psi_i^{\rel})$ and $\psi_i^{\rel}$ which yields contributions
from bubbles which are contracted when forgetting the $(n+1)$-th marking (which lies in the rubber necessarily). This yields:
\begin{gather*}
-2 \sum_{i=1}^{\ell(\eta)}
\CC_{g,\kk}^{\pi/D, \bullet}\left( \gamma_1, \ldots, \gamma_n ; \underline{\eta}\big|_{\delta_i \mapsto \lambda|_{D} \cdot \pi_D^{\ast}\pi_{D \ast}( \delta_i)} \right) 
-2  \sum_{\substack{i,j=1 \\ i \neq j}}^{\ell(\eta)}
\frac{\eta_i}{\eta_j}
\CC_{g,\kk}^{\pi/D, \bullet}\left( \gamma_1, \ldots, \gamma_n ; \underline{\eta}\big|_{\substack{\delta_i \mapsto \lambda \cdot \delta_i \\ \delta_j \mapsto \pi_D^{\ast} \pi_{D \ast}(\delta_j)}} \right).
\end{gather*}
Adding up then yields the claim.
The second claim follows by cupping the above computations by an arbitrary tautological class.
\end{proof}

\section{Including the torsion}
\label{appendix:torsion}
Let $Y$ be an Enriques surface. Let $\widetilde{H}_2(Y,\BZ)$ denote the integral second homology of $Y$,
and let $H_2(Y,\BZ) = \widetilde{H}_2(Y,\BZ)/\mathrm{torsion}$. There is a natural exact sequence
\[ 0 \to \BZ_2 \to \widetilde{H}_2(Y,\BZ) \xrightarrow{[ - ]}
H_2(Y,\BZ) \to 0 \]
where the first map sends the generator to $c_1(\omega_Y) \cap [Y]$,
and the second map is written as $\gamma \mapsto [\gamma]$.
The sequence splits but not canonically.

For $\beta \in H_2(Y,\BZ)$ the moduli space of stable maps $\Mbar_g(Y,\beta)$ is
a disjoint union of two components $\Mbar_g(Y,\gamma)$ corresponding to maps of degree $\gamma \in \widetilde{H}_2(Y,\BZ)$ with $[\gamma] = \beta$. We define the Gromov-Witten invariants of degree $\gamma$ to be
\[ N_{g, \gamma} = \int_{[ \Mbar_g(Y,\gamma) ]^{\vir}} (-1)^{g-1} \lambda_{g-1}. \]
If $\gamma_1$ and $\gamma_2 = \gamma_1 + c_1(\omega_Y)$ are the two lifts of $\beta$, then
\[ N_{g,\beta} = N_{g, \gamma_1} + N_{g,\gamma_2}. \]

An interesting question is to compute these refined invariants $N_{g,\gamma}$.
We give two basic results in this direction:

\begin{prop} \label{prop:not 2 divisible}
If $\beta \in H_2(Y,\BZ)$ is not $2$-divisibile (that is $\beta/2 \notin H_2(Y,\BZ)$), then $N_{g,\gamma_1} = N_{g,\gamma_2}$ where $\gamma_1, \gamma_2 \in \widetilde{H}_2(Y,\BZ)$ are the two lifts of $\beta$.
\end{prop}
\begin{proof}
By a recent result of Knutsen \cite{Knutsen} the pairs $(Y, \gamma_1)$ and $(Y,\gamma_2)$ are deformation equivalent if and only if $\beta$ is not $2$-divisible.
The claim hence follows from the deformation invariance of Gromov-Witten invariants.
\end{proof}

If $\beta$ is $2$-divisible, then the lifts $\gamma_1, \gamma_2$ are not deformation-equivalent: one of them is $2 \alpha$ for a class $\alpha \in \widetilde{H}_2(Y,\BZ)$, while the other is not.
Hence the invariants $N_{g,\gamma_1}$ and $N_{g, \gamma_2}$ may be different in this case. We give an example that they are in fact different:

Let $f_1, f_2$ be the half-fibers of an elliptic fibration, and let $f=[f_i]$.
The two lifts of $df$ are $d f_1$ and $(d-1) f_1 + f_2$.
Consider their difference
\[ N_{g,df}^{-} = N_{g, d f_1} - N_{g, (d-1) f_1 + f_2}. \]
\begin{prop}
\[
N_{g,df}^{-} = 
\begin{cases}
0 & \text{ if } d \text{ is odd}\\
N_{g, d f}
& \text{ if } d \text{ is even}
\end{cases}
\]
\end{prop}
\begin{proof}
The case $d$ odd follows from Proposition~\ref{prop:not 2 divisible} so let us assume that $d$ is even. Stable maps $h : C \to Y$ in class $df$ must map to a single fiber:
Hence $h_{\ast}[C]$ as a divisor is either $d f_1$, $d f_2$ or $d/2 \cdot F_x$ where $F_x$ is a fiber which is not a half-fiber. Thus the refined degree of $h$ is always $d f_1$ and we get $N_{g, (d-1) f_1 + f_2} = 0$. This shows that
$N_{g, df}^{-} = N_{g, d f_1} = N_{g, df}$.
\end{proof}

The last proposition shows that $N_{g, \gamma}$ contains non-trivial information.
It is plausible that the methods of this paper can be extended to also determine $N_{g,\gamma}$. However it would require to lift all steps in the proof to include torsion, which would take us too far here.

\end{document}